\documentclass[a4paper]{amsart}
\usepackage{amsmath,amssymb,amsfonts}
\usepackage{mathrsfs,latexsym,amsthm,enumerate}
\usepackage{tikz}
\usetikzlibrary{matrix,arrows}
\usetikzlibrary{positioning}
\usetikzlibrary{patterns}
\usepackage{amscd}
\numberwithin{equation}{section}

\newtheorem{theorem}{Theorem}[section]
\newtheorem{lemma}[theorem]{Lemma}
\newtheorem{corollary}[theorem]{Corollary}
\newtheorem{example}[theorem]{Example}
\newtheorem{proposition}[theorem]{Proposition}
\newtheorem{remark}[theorem]{Remark}

\let\oldtocsection=\tocsection

\let\oldtocsubsection=\tocsubsection

\let\oldtocsubsubsection=\tocsubsubsection

\renewcommand{\tocsection}[2]{\hspace{0em}\oldtocsection{#1}{#2}}
\renewcommand{\tocsubsection}[2]{\hspace{1.5em}\oldtocsubsection{#1}{#2}}
\renewcommand{\tocsubsubsection}[2]{\hspace{0.4em}\oldtocsubsubsection{#1}{#2}}

\title{A perspective on non-commutative frame theory}
\author{Ganna Kudryavtseva}
\address{
Ganna Kudryavtseva, Faculty of Civil and Geodetic Engineering, University of Ljubljana \\ Jamova cesta~2, SI-1000 Ljubljana, SLOVENIA}
\email{ganna.kudryavtseva@fgg.uni-lj.si}
\author{Mark V. Lawson}

\address{Mark V.Lawson, Department of Mathematics
and the
Maxwell Institute for Mathematical Sciences, 
Heriot-Watt University,
Riccarton,
Edinburgh~EH14~4AS, UNITED KINGDOM}
\email{m.v.lawson@hw.ac.uk}

\thanks{The first author was partially supported by  ARRS grant P1-0288.
 The second author was partially supported by EPSRC grant EP/I033203/1.
He would also like to thank Gracinda Gomes for inviting him to CAUL in Lisbon in December 2013 where some of the work in this paper was carried out,
and Pedro Resende for useful discussions and for pointing the authors  in the direction of \cite{ReRo}.}

\begin{document} 

\begin{abstract}

This paper extends the fundamental results of frame theory to a non-commutative setting where the role of locales is taken over by \'etale localic categories.
This involves ideas from quantale theory and from semigroup theory, specifically Ehresmann semigroups, restriction semigroups and inverse semigroups.
We prove several main results. 
To start with, we establish a duality between the category of complete restriction monoids and the category of \'etale localic categories. The relationship between monoids and categories is mediated by a class of quantales called restriction quantal frames. 
This result builds on the work of Pedro Resende on the connection between pseudogroups and  \'etale localic groupoids but in the process we both generalize and simplify:
for example, we do not require involutions and, in addition, 
we render his result functorial. 
A wider class of quantales, called multiplicative Ehresmann quantal frames, is put into a correspondence with those localic categories where the multiplication structure map is semiopen, and all the other structure maps are open.
We also project down to topological spaces and, as a result, extend the 
classical adjunction between locales and topological spaces to an adjunction between \'etale localic categories and \'etale  topological categories. In fact, varying morphisms, we obtain several adjunctions. 
Just as in the commutative case, we restrict these adjunctions to spatial-sober and coherent-spectral equivalences. 
The classical equivalence between coherent frames and distributive lattices is extended to an equivalence between coherent complete restriction monoids and distributive restriction semigroups. Consequently, we deduce several dualities between distributive restriction semigroups and spectral \'etale topological categories. We also specialize these dualities for the setting where the topological categories are cancellative or are groupoids. Our approach thus links, unifies and extends the approaches taken in the work by Lawson and Lenz and by Resende. 
\end{abstract}
\maketitle
\tableofcontents

\section{Introduction and preliminaries}\label{s:int:p}

\subsection{Introduction}
The first goal of this paper is to connect, unify and extend the two approaches adopted in \cite{Re} and \cite{LL1,LL2} in relating inverse semigroups with \'etale localic or topological groupoids. The paper \cite{Re} achieves this by showing how to construct \'etale localic groupoids from pseudogroups by means of a class of quantales, whereas the papers \cite{LL1,LL2} achieve this by relating distributive inverse semigroups and pseudogroups to \'etale topological groupoids making use of prime and completely prime filters. The paper \cite{Re} is clearly a generalization of classical frame and locale theory, whereas \cite{LL1,LL2} is more directly a generalization of classical Stone duality. Of course, classical Stone duality can itself be approached from frame and locale theory. See the first two chapters of \cite{J1}, for example. Thus it is entirely plausible that such a unification is possible. Both of these approaches were motivated by the theory of $C^*$-algebras but from slightly different traditions. The papers \cite{LL1,LL2} arose from the tradition going back to Renault�s influential monograph \cite{Ren} and the later book by Paterson \cite{Pat} via a sequence of papers: most notably \cite{Exel1,Exel2,Kell1,Kell2,Law3,Law4,LMS,Lenz,S}. In other words, it was an approach which arose from concrete examples and topological groupoids. The paper \cite{Re}, whilst certainly related to \cite{Pat,Ren}, can be seen as deriving more from a tradition starting in topos theory, and pursuing a route via the quantale theory introduced by Chris Mulvey \cite{Mulvey,Ros}, in which localic, rather than topological, groupoids are important.

Perhaps the key insight from \cite{Re} is that quantales play a role in mediating between semigroups and spaces.
Quantales themselves were self-consciously defined as non-commutative locales.
Whereas a topological space takes the notion of point as primary and that of open subset as secondary, 
the theory of locales takes the notion of open subset as primary and that of point as secondary.
These two versions of the notion of `space' do not quite match, instead their respective categories are, rather, related by means of an adjunction. This dichotomy between locale and space is actually the occasion for this paper, but viewed from a more general perspective.

The second goal of this paper  is to replace \'etale localic or topological {\em groupoids} by \'etale localic or topological {\em categories}.
But this raises the question of what class of semigroups should replace the inverse semigroups that seem to play such an indispensible role in the above theory.
The answer to this question begins with the observation that Resende's inverse quantal frames \cite{Re} are, 
in fact, examples of what are termed {\em Ehresmann semigroups} in \cite{Law1}\footnote{Our use of Ehresmann semigroups is not entirely coincidental.
We refer the reader to \cite{LL1} for further remarks on Ehresmann's influence on the development of the theory of frames, locales and pseudogroups.
See also \cite{Law1} for further remarks on Ehresmann's {\em Oeuvre} \cite{Ehresmann}.}.
A class of such semigroups, called {\em restriction semigroups}, will play the role in this paper that inverse semigroups play in \cite{Re} and \cite{LL1,LL2}.
Such semigroups have been around for a long time.
See the survey article by Chris Hollings \cite{Hollings} as well as the articles \cite{CGH,GH,G1,G2}.
Although our paper is a generalization of both \cite{LL1,LL2} and \cite{Re},
we would argue that working at this level of generality actually clarifies and simplifies the theory developed in those papers.
The use of localic categories generalized from \cite{Re} greatly sharpens some of our key results,
whereas the use of involutions in \cite{Re}, which we avoid in our generalization, renders the theory  superficially more complex.
One very important additional feature of our theory is that we also study morphisms and so our results are fully functorial,
something not achieved in the main result of \cite{Re}.

Our paper suggests a number of new research directions which time and space preclude our pursuing in more detail here.
First, just as non-commutative Stone duality has revolutionized the theory of inverse semigroups,
the recent work of Wehrung \cite{Wehrung} being a sign of this, so we anticipate that the theory developed in this paper will have 
an important influence on the theory of restriction semigroups: from constructing interesting examples of such semigroups to providing motivation for developing that theory. 
More generally, we wonder at how our theory might be extended to deal with restriction {\em categories} as developed by Cockett et al \cite{CGH, Cockett}.
Second, there is the important question of how our work fits into the theory of operator algebras;
the r\^ole of inverse semigroups and \'etale groupoids is of course well established
but it is natural to ask if a theory of combinatorial non-selfadjoint operator algebras associated to \'etale categories could be developed that extended the theory developed in \cite{Exel1}.
Third, there is the question of how our work fits into the broader picture provided by topos theory (particularly the programme being pursued by Olivia Caramello).
It is worth remembering that inverse semigroups, in their guise as pseudogroups, and \'etale groupoids arose from the foundations of differential geometry.
Specifically, both concepts lie at the base of Ehresmann's attempt \cite{Ehresmann} to construct the categorical foundations of the concept of a local structure.
However, the use of inverse semigroups and groupoids lead, essentially, to the construction of isomorphism classes of such structures.
Grandis \cite{G1,G2} showed how to construct more natural categories of such structures using restriction semigroups/categories.
We suspect that our theory might play a similar r\^ole within topos theory that \'etale groupoids play within the broader field of non-commutative geometry.
But above and beyond these specific research directions, there are what we might term the `philosophical'
implications of the duality theory we have established.
It has long been maintained that \'etale groupoids are concrete instances of that somewhat nebulous concept: a non-commutative topological space.
However, the fact that groupoids satisfy a notion of invertibility, important though this is,  nevertheless seems extraneous
to this generalized concept of a space. 
Our definition of an \'etale category, however, seems to be the more natural general definition of a non-commutative space
with \'etale groupoids now being seen as the invertible such spaces.

\subsection{Outline of the paper}\label{sub:outline} 
We now outline the structure of the paper and highlight the main results we obtain. We refer the reader to Subsection~\ref{sub:prelim} for the necessary background in the theory of frames and locales and in category theory as well as for references where more details can be found. For background in semigroup theory, we recommend \cite{H2}. In Subsection~\ref{sub:commutative} we provide an account of the relationship between locales and topological spaces, as well as of commutative Stone dualities which are deduced from this relationship.

In Section \ref{s:quantalization} we first address Ehresmann and restriction semigroups and prove some results needed in the sequel. We then study properties of complete restriction monoids, the latter being non-regular analogues of  (complete) pseudogroups. Further, we introduce and study Ehresmann quantales. These are unital quantales, whose multiplicative monoids carry the structure of  Ehresmann monoids. Restriction quantal frames form a subclass of Ehresmann quantal frames. The main result of Section \ref{s:quantalization} is the following.

\begin{theorem}[Quantalization Theorem]\label{th:int1}
For appropriately defined classes of mor\-phisms, 
the category of complete restriction monoids and the category of restriction quantal frames are equivalent.
\end{theorem}  

Both of the structures involved in the Quantalization Theorem are generalizations of frames. The category of restriction quantal frames is, in a way,  a more natural generalisation of  the category of frames and is thus well-suited to obtain a link with a category which generalizes locales. A complete restriction monoid can be viewed as a kind of reduction of a restriction quantal frame which, however, carries enough data for the reconstruction of a restriction quantal frame.
The Quantalization Theorem extends a corresponding result by Resende \cite{Re} obtained for the setting of pseudogroups and inverse quantal frames. The passage from a complete restriction monoid to a restriction quantal frame is carried over by an adaptation of Resende's construction of the enveloping quantal frame. For the passage in the reverse direction we introduce the notion of a {\em partial isometry} of a restriction quantal frame. This is a different notion than Resende's partial units, rather than a generalization of the latter, though we show that for inverse quantal frames both notions do coincide.  Partial isometries arise from the interplay between two partial orders that may be defined on an Ehresmann quantale.
One of them is the underlying order of the sup-lattice structure, and  the other one is defined in terms of the Ehresmann monoid structure.
In the spatial case,  partial isometries correspond precisely to open {\em local bisections} of the underlying topological categories.
This is the significant example: partial isometries should be viewed as abstract local bisections.

The two main results of Section~\ref{s:corresp} relate  classes of quantales to classes of localic categories.

\begin{theorem}[Correspondence Theorem]\label{th:int2}
There is a bijective correspondence between multiplicative Ehresmann quantal frames and a class of localic categories, 
called {\em quantal localic categories}. These categories have the property that the multiplication map is only assumed to be semiopen, whereas the other structure maps are open.
\end{theorem}

Theorem \ref{th:int2} is, however, too general, to provide a link of quantal localic categories with complete restriction monoids. To obtain such a link, we need to restrict ourselves to a subclass of quantal localic categories we call {\em \'etale localic categories}. These have the domain and the range maps  \'etale and the multiplication and the unit maps open. 

\begin{theorem}[Etale Correspondence Theorem]\label{th:int3}
There is a bijective correspondence between restriction quantal frames and \'etale localic categories.\end{theorem}

Combining Theorem~\ref{th:int1} with Theorem~\ref{th:int3}, we arrive at a (non-functorial) equivalence of three types of structures which extends the main result of  \cite{Re}. But in Section~\ref{s:env} we go further and, unlike Resende's paper, augment the Etale Correspondence Theorem by morphisms. We thus prove the following result.

\begin{theorem}[Duality Theorem]\label{th:int4} 
The following are equivalent.
\begin{enumerate}

\item The category of complete restriction monoids and proper $\wedge$-morphisms.

\item The category of restriction quantal frames and proper $\wedge$-morphisms.

\item  The opposite of the category of  \'etale localic categories and localic sheaf functors. 

\end{enumerate}
\end{theorem}

Our perspective is that Theorem \ref{th:int4} extends the classical duality between frames and locales where complete restriction monoids or restriction quantal frames are generalizations of frames, whereas \'etale localic categories are generalizations of locales.  

All the above results can be enriched by the addition of involutions to each class of objects which is done in Section~\ref{s:involutive}. As a result, all of Resende's main theorems~\cite{Re} can be derived from our more general standpoint, see Section~\ref{s:inv:setting}.

In Section~\ref{s:adjun} we project down from the localic world to that of topological spaces and prove the following result which is a non-commutative analogue of the classical adjunction between the categories of locales and topological spaces.

\begin{theorem}[Adjunction Theorem]\label{th:int5}
For suitably defined classes of morphisms, there are adjunctions between the category of \'etale localic categories and the category of \'etale topological categories. 
\end{theorem}

The adjunctions in Theorem \ref{th:int5} are established via an extension of the spectrum and the open set functors which establish the classical adjunction between locales and topological spaces. 
We highlight an important specialization of the Adjunction Theorem, to the setting where \'etale localic categories are cancellative.  Under the adjunctions these correspond to so-called ample complete restriction monoids. 
We also establish variations of the Adjunction Theorem to the setting where the categories are involutive or are groupoids. 
Combining Theorems~\ref{th:int4} and~\ref{th:int5} we can produce several adjunctions between the categories of complete restriction monoids and \'etale topological  categories. This leads to a new perspective on the adjunction between pseudogroups and \'etale topological groupoids from \cite{LL1} and provides a direct link between the aproaches of the papers \cite{Re} and \cite{LL1,LL2}. For example, the functor from \cite{LL1}, which assigns to a pseudogroup $S$ an \'etale topological groupoid, in fact produces the spectrum of the \'etale localic groupoid associated to $S$ via the approach of \cite{Re}.

In Section~\ref{s:dualities} we restrict the Adjunction Theorem to sober-spatial and spectral-coherent settings and prove several duality theorems. 

\begin{theorem}[Topological Duality Theorem]\label{th:int6} For sui\-tab\-ly defined classes of morphisms, the category of sober (resp. spectral, strongly spectral) \'etale topological categories
is dually equivalent to the category of spatial  (resp. coherent, strongly coherent) complete restriction monoids.
\end{theorem}

We then establish an equivalence between coherent complete restriction monoids and distributive restriction semigroups, which extends the classical equivalence between coherent frames and distributive lattices. This leads to the following duality theorem.

\begin{theorem}[Topological Duality Theorem II]\label{th:int8}
For sui\-tab\-ly defined classes of morphisms, the category of distributive restriction semigroups (resp. $\wedge$-semigroups, monoids, $\wedge$-monoids) is dually equivalent to the category of spectral (resp. strongly spectral, compact spectral, compact strongly spectral) \'etale topological categories.
\end{theorem}

Under the duality in Theorem \ref{th:int8}, distributive ample semigroups correspond to cancellative spectral \'etale topological categories. This may be important for the possible applications to operator algebras as restriction semigroups of partial isometries in Hilbert spaces are ample.

\subsection{Preliminaries}\label{sub:prelim}
In this section we provide an overview of the basic notions and results in frame and locale theory and category theory required for reading this paper. We also set up some notation that will be used throughout the paper.

\vspace{0.1cm}

\noindent{\bf Posets and adjoint maps.} 
Let $(X,\leq)$ be a poset and $A \subseteq X$.
Define 
$$A^{\downarrow} = \{ x \in X \colon \exists a \in A, x \leq a\}.$$
If $A = \{a\}$ we write $a^{\downarrow}$ instead of $\{a\}^{\downarrow}$.
If $A = A^{\downarrow}$ we say that $A$ is an {\em order ideal}.
Order ideals of the form $a^{\downarrow}$ are said to be {\em principal}.

A function $\theta \colon X \rightarrow Y$ between posets is said to be {\em monotone} if $x \leq y$ implies that $\theta (x) \leq \theta (y)$.
In posets, we will often use the convention that $\bigwedge$ and $\bigvee$ refer to arbitrary meets and joins, respectively,
whereas $\wedge$ and $\vee$ refer to finite meets and joins, respectively. If we have a pair of monotone maps
$$f \colon A \rightarrow B
\mbox{ and }
g \colon B \rightarrow A$$
such that
$$g(b) \leq a \Longleftrightarrow b \leq f(a),$$
we say that $g$ is a {\em left adjoint} of $f$ and $f$ is a {\em right adjoint} of $g$.
If a monotone map has an adjoint then that adjoint is unique.
We talk about an {\em adjoint pair} $(g,f)$.
In the following proposition, we collect some basic properties of adjoints that we shall use throughout this paper.

\begin{proposition} \label{prop: properties_of_adjoints} \mbox{}

\begin{enumerate}
\item  Let $f \colon A\to B$ and $g \colon B\to A$ be a pair of monotone maps where $f$ is left adjoint to $g$.
Then $fgf=f$ and  $gfg = g$ and both $fg$ and $gf$ are idempotents.
\item  Left adjoints preserve all joins, and right adjoints preserve all meets.
\item  Let $f \colon A\to B$ and $g \colon B\to A$ be a pair of monotone maps where $f$ is left adjoint to $g$.
If $b \in B$ then 
$$g(b) = \bigwedge\{a\in A \colon  b \leq f(a)\}.$$
If $a \in A$ then 
$$f(a)=\bigvee \{b \in B\colon g(b) \leq a\}.
$$
\item In the context of complete lattices, a monotone map has a right adjoint if and only if it preserves all joins
and a monotone map has a left adjoint if and only if it preserves all meets.
\end{enumerate}
\end{proposition}

\vspace{0.1cm}

\noindent{\bf Frames and locales.}
We have used all the following references for the theory of frames and locales at various times \cite{Bo,J1,J,MM,Vickers}. 
We briefly summarize what we need.

A {\em sup-lattice} $L$ is a poset in which any family of elements $A$ has a join $\bigvee A$. 
The join of an empty family is the least element of the poset and is denoted by $0$. 
The join of all elements of the poset is the greatest element, or {\em top}, and is denoted by $1$ or $1_{L}$. 
If $A$ is a set of elements in a sup-lattice, then the join of all elements $y$ such that $y\leq a$ for all $a\in A$ is the meet $\bigwedge A$. 
So a sup-lattice is in fact a complete lattice. 
A {\em sup-lattice map}, or {\em sup-map}, is defined as a map $f:A\to B$ between sup-lattices that preserves all joins (but not necessarily meets).

A {\em frame} is a sup-lattice $L$, satisfying the condition that for any of its elements $x_i$, where $i\in I$, and $y$ we have
$$
y\wedge \left(\bigvee_{i}x_i\right)=\bigvee_{i}(y\wedge x_i).
$$
A map $f \colon A\to B$ between frames is called a {\em frame map} if it preserves finite meets and any joins of elements of $A$. 
The lattice of open sets $\Omega(X)$ of a topological space $X$ is a frame.
Let $X,Y$ be topological spaces and $f \colon X\to Y$ a continuous map. 
Then the inverse image map $f^{-1} \colon \Omega(Y)\to \Omega(X)$ is a frame map.

The category of {\em locales} is defined as the opposite category to the category of frames. 
We use the following notational convention.
If $X$ is a locale then $O(X)$ is the frame. In fact, $X$ and $O(X)$ are two different ways to denote the same object, and we take a convention to write $X$ or $O(X)$ depending on if  we regard it as a locale or as a frame, respectively. 
Hence a locale map $f \colon X\to Y$ is defined by a frame map $f^* \colon O(Y)\to O(X)$.

Let $L,M$ be locales. A locale map $f:L\to M$ is called {\em semiopen} if the frame map $f^*:O(M)\to O(L)$ preserves
arbitrary meets. Then $f^*$ has a left adjoint $f_!:O(L)\to O(M)$ which is called the {\em direct image map} of $f$. The map $f_!$ preserves arbitrary sups as a left adjoint but it does not, in general, even preserve binary meets. 
A semiopen locale map $f:L\to M$ is called {\em open} if
$$
f_!(a\wedge f^*(b))=f_!(a)\wedge b
$$
for all $a\in O(L)$ and $b\in O(M)$. The condition above is called the {\em Frobenius condition}. If $f \colon X\to Y$ is an open continuous map of topological spaces then it is open as a locale map. 
The converse, however, does not hold in general.
We refer the reader to \cite{J} for the details behind this condition.
An open locale map $f:L\to M$ is called a {\em local homeomorphism} or an {\em \'etale map} if there is $C\subseteq O(L)$ such that 
\begin{equation*}\label{eq:cover}
1_{O(L)}=\bigvee C
\end{equation*}
and for every $c\in C$ the frame map $O(M)\to c^{\downarrow}$ given by $x\mapsto f^*(x)\wedge c$ is surjective.

Pushouts of frames will play an important role in this paper, and we refer the reader to \cite{Bo} for the details.  
Let $f^{*} \colon L \rightarrow A$ and $g^{*} \colon L \rightarrow B$ be frame maps.
Then they have a pushout $h^{*} \colon A \rightarrow A \otimes_{L} B$ and $k^{*} \colon B \rightarrow A \otimes_{L} B$
given by $h^{*}(a) = a \otimes 1$ and $k^{*}(b) = 1 \otimes b$.
The pushout  frame $A \otimes_{L} B$ is defined by the following relations:
\begin{enumerate}
\item $\bigvee_{i} (a_{i} \otimes b) = \left( \bigvee_{i} a_{i} \right) \otimes b$, and dually. 
\item $(a \wedge f^{*}(l)) \otimes b = a \otimes (g^{*}(l) \wedge b)$.
\item $(a \wedge b) \otimes (a' \wedge b') = (a \otimes a') \wedge (b \otimes b')$.
\end{enumerate}
If $L = \{0,1\}$ then we get the {\em coproduct} of $A$ and $B$.

\vspace{0.1cm}

\noindent{\bf Localic and topological categories.} The following definition is taken from \cite{M}. A {\em localic category} is an internal category in the category of locales.
This means, precisely, the following.
We are given  
$$C=(C_1,C_0,u,d,r,m)$$ 
where $C_{1}$ is a locale, called the {\em object of arrows},
and $C_{0}$ is a locale, called the {\em object of objects},
together with four locale maps
$$
u \colon C_0\to C_1, \quad d,r \colon C_1\to C_0, \quad m \colon C_1\times_{C_0} C_1\to C_1,
$$
 called {\em unit}, {\em domain}, {\em codomain}, and {\em multiplication}, respectively,
where $C_1\times_{C_0} C_1$ is the {\em object of composable pairs} defined by the following pullback diagram in the category of locales
\begin{equation*}
\begin{tikzpicture}[baseline=(current  bounding  box.center)]
\node (bl) {$C_1$};
\node (br) [node distance=4cm, right of=bl] {$C_0$};
\node (ul) [node distance=1.8cm, above of=bl] {$C_1\times_{C_0} C_1$};
\node (ur) [node distance=1.8cm, above of=br] {$C_1$};
\path[->]
(bl) edge node[below]{$d$} (br);
\path[->]
(ul) edge node[above]{$\pi_2$} (ur);
\path[->]
(ul) edge node[left]{$\pi_1$} (bl);
\path[->]
(ur) edge node[right]{$r$} (br);
\end{tikzpicture}
\end{equation*}
The codomain map is sometimes referred to as a {\em range map}.
The four maps are subject to conditions that express the usual axioms of a category:
\begin{enumerate}
\item[(Cat1)] $du=ru=id$. 
\item[(Cat2)] $m(u \times id) = \pi_{2}$ and $m(id \times u) = \pi_{1}$.
\item[(Cat3)] $r \pi_{1} = rm$ and $d \pi_{2} = dm$.
\item[(Cat4)] $m(id \times m) = m(m \times id)$.
\end{enumerate}

{\em Topological categories} are defined similarly, as internal categories in the category of topological spaces.
If $C=(C_1,C_0)$ is a topological category then the space of composable pairs $C_1\times_{C_0} C_1$
equals
$$
\{(a,b)\in C_1\times C_1\colon d(a)=r(b)\}.
$$

Let $C=(C_1,C_0)$ and $D=(D_1,D_0)$ be two internal categories. 
By an {\em (internal) functor} $f \colon C \to D$ we mean a pair of morphisms $f_1 \colon C_1\to D_1$ and $f_0 \colon C_0\to D_0$ in the given category that 
commute with the structure maps of the categories: 
\begin{enumerate}
\item[(Fun1)] $m(f_1\times f_1) = f_1m$.
\item[(Fun2)] $df_1=f_0d$.
\item[(Fun3)] $r f_1=f_0 r$.
\item[(Fun4)] $uf_{0} = f_{1}u$.
\end{enumerate}

\subsection{The commutative setting}\label{sub:commutative}
We provide a brief overview of the relationship between the theory of locales and the theory of topological spaces, as well as of non-commutative Stone dualities which can be deduced from this relationship \cite{J1,MM}.  These results will be generalized to a non-commutative setting in Section~\ref{s:adjun}.
The category of locales will be denoted by ${\mathbf{Loc}}$ and the category of topological spaces by ${\mathbf{Top}}$. 

We first recall the classical adjunction between the categories ${\mathbf{Loc}}$ and ${\mathbf{Top}}$. 
Let $X$ be a topological space. By $\Omega(X)$ we will denote the locale of opens of $X$. In order not to overload notation, we will denote the frame of opens of $\Omega(X)$ also by $\Omega(X)$. Let $f\colon X\to Y$ be a continuous map. Then $f^{-1}\colon \Omega(Y) \to \Omega(X)$ is a frame map and thus defines a locale map $\Omega(f)\colon\Omega(X)\to \Omega(Y)$. 

Let $L$ be a locale and ${\bf 2}=\{0,1\}$ be a two-element frame. A {\em point} of $L$ is a frame map $f:O(L)\to {\bf 2}$. Let ${\mathsf{pt}}(L)$ be the set of points of $L$. It is sometimes called the {\em spectrum} of $L$. 

\begin{remark}
{\em If $f\in {\mathsf{pt}}(L)$ then $f^{-1}(1)$ is a {\em completely prime filter} of $L$, that is a non-empty subset $F$ of $L$ such that (i) $0\not\in F$; (ii) $a\in F$ and $b\geq a$ imply that $b\in F$; (iii) $a,b\in F$ imply that $a\wedge b\in F$ and (iv) if $\bigvee A\in F$ then $a\in F$ for some $a\in A$. The assignment $f\mapsto f^{-1}(1)$ is a bijection between the points of $L$ and completely prime filters of $L$. Therefore, points of $L$ can be equivalently viewed as completely prime filters of $L$, the approach adopted in \cite{LL1,LL2}.}
\end{remark}

For each $a\in O(L)$ we set
 $$X_a=\{f\in {\mathsf{pt}}(L)\colon f(a)=1\}.$$
 It is immediate that 
$$X_{\bigvee_{i}a_i}=\bigcup_{i}X_{a_i}\,\,  \text{ and }\,\, X_{a\wedge b}=X_a\cap X_b.$$
 It follows that the sets $X_a$ constitute a topology on ${\mathsf{pt}}(L)$. We always consider ${\mathsf{pt}}(L)$ as a topological space with respect to this topology. Let $\varphi\colon L\to M$ be a locale map. We define 
 $${\mathsf{pt}}(\varphi)\colon {\mathsf{pt}}(L) \to {\mathsf{pt}}(M), \,\,\, f\mapsto \varphi f^*.$$
The assignments $\Omega$ and ${\mathsf{pt}}$ are functorial. Moreover, the following holds.

\begin{theorem}[Classical Adjunction Theorem] \label{th:pr1} The functor $\Omega\colon {\mathbf{Top}}\to {\mathbf{Loc}}$ is a left adjoint to the functor ${\mathsf{pt}}\colon {\mathbf{Loc}}\to {\mathbf{Top}}$. For each $X\in {\mathrm{Ob}}({\mathbf{Top}})$ the component $\eta_X$ of the unit $\eta\colon 1_{{\mathbf{Top}}}\to {\mathsf{pt}}\,\Omega$ of the adjunction is given by $\eta_X(a)(B)=1$ if and only if $a\in B$.
\end{theorem}

A locale $L$ is said to be {\em spatial} if the  frame map 
$$\epsilon_L^*\colon O(L)\to \Omega{\mathsf{pt}}(L), \,\, a\mapsto X_a,
$$
which defines the component $\epsilon_L$ of the counit of the above adjunction, is a bijection. A topological space $X$ is said to be {\em sober} if the map $\eta_X$ is a bijection.
Theorem \ref{th:pr1} restricts to the following equivalence of categories.

\begin{theorem}\label{th:pr1a} The category of sober spaces is equivalent to the category of spatial locales.
\end{theorem}

Let $L$ be a locale. 
 An element $a\in L$ is called {\em finite} if for every $M\subseteq L$ with $\bigvee M =a$ there exists a finite $F\subseteq M$ with $\bigvee F=a$. A locale $L$ is called {\em coherent}\footnote{We do not require that $1$ is a finite element.} if every element of $L$ is a join of finite elements and the meet of any two finite elements is finite. If in addition $1$ is a finite element, $L$ is said to be {\em compact coherent}.
A topological space $X$ is said to be {\em spectral}\footnote{We do not require a spectral space to be compact.} if it is sober and has a basis of compact-open sets that is closed under finite non-empty intersections.
A Hausdorff spectral space is called a {\em Boolean space}. A continuous map is {\em coherent} if the inverse images of compact-open sets are compact-open. Any continuous map between Boolean spaces is coherent.

By a {\em distributive lattice} (resp. a {\em Boolean algebra}) we understand one which possesses a bottom element $0$, but not necessarily a top element $1$. A {\em unital distributive lattice} (resp. a {\em unital Boolean algebra}) is a one which also possesses a top element.  A map between distributive lattices or Boolean algebras  is assumed to satisfy the condition that for every $a\in M$ there is $b\in L$ such that $f(b)\geq a$. If $L$ and $M$ are unital, the latter requirement reduces to $f(1_L)=1_M$. The following well-known results are consequences of  Theorem \ref{th:pr1}.

\begin{theorem}[Classical Stone Duality I]\label{th:pr2}
The following catego\-ries are equi\-valent:
\begin{enumerate}
\item The category of  spectral spaces (reps. compact spectral spaces) and coherent maps.
\item The category of coherent locales (resp. compact coherent locales) and coherent maps.
\item The opposite of the category of  distributive lattices (resp. unital distributive lattices) and their maps.
\end{enumerate}
\end{theorem}

Under the above equivalence compact-open sets of a spectral space correspond to finite elements of a coherent locale and form a distributive lattice.

\begin{theorem}[Classical Stone Duality II] \label{th:pr3}  The cate\-gory of Boolean al\-geb\-ras (resp. unital Boolean algebras) is dually equivalent to the category of  Boolean spaces (resp. compact Boolean spaces).
\end{theorem}

\section{The Quantalization Theorem}\label{s:quantalization}

The main goal of this section is to prove Theorem~\ref{th:quant_psgrps}, the Quantalization theorem. It will be then combined with the main theorem of the next section and lead to the Duality Theorem. 
\subsection{Ehresmann and restriction semigroups}\label{sub:semigroups}
In this subsections we introduce the two classes of semigroups that will play the main role in this paper:
the Ehresmann semigroups and their subclass called restriction semigroups. 

Let $S$ be a semigroup and $E(S)$ its set of idempotents. 
Let $E\subseteq E(S)$ be a fixed non-empty subset where we emphasize that it need not consist of all idempotents.
We call $E$ the set of {\em projections}.
We suppose in addition that there are two functions $\lambda$ and $\rho$ from $S$ to $E$, called the {\em structure maps}
or the {\em left and right supports}.
We say that $S$ is an {\em Ehresmann semigroup with respect to the set} $E$ if the following conditions are satisfied:
\begin{enumerate}[(ES1)]
\item\label{es1} $E$ is a commutative subsemigroup.
\item\label{es2} If $a\in E$ then $\lambda(a)=\rho(a)=a$.
\item\label{es3} $a \lambda (a) = a$ and $\rho (a)a = a$ for any $a\in S$.
\item\label{es4} $\lambda(\lambda(a)b)=\lambda(ab)$ and $\rho(a\rho(b))=\rho(ab)$ for any $a,b\in S$.
\end{enumerate}
It is easy to prove that $\lambda (a)$ is the smallest projection in the semilattice $E$ amongst all those
projections $e$ such that $ae = a$.
A dual result holds for $\rho$.
{\em Thus the structure maps are uniquely determined by the multiplication in the semigroup and the set of projections.}
The maps $\lambda$ and $\rho$ can be thought about as maps to $S$ whose range belongs to $E$, and thus as unary operations on $S$.

Ehresmann semigroups were introduced in  \cite{Law1} building mainly on the work of de Barros.
From the perspective of this paper, they can be regarded as wide-ranging generalizations of inverse semigroups.
In what follows, when the set $E$ is understood, we call an Ehresmann semigroup with respect to $E$ simply an {\em Ehresmann semigroup}. 
The reader is warned that the terminology surrounding classes of Ehresmann semigroups is  varied and potentially confusing.
This reflects the fact that they have been rediscovered in a number of different contexts, for example:
within Ehresmann's own work on ordered categories and their associated semigroups;
within category theory as part of attempts to formalize categories of partial maps;
and within semigroup theory as generalizations of the PP monoids studied by John Fountain and his students.

We fix $E\subseteq E(S)$ and all Ehresmann semigroups considered are with respect to such an $E$.
The following is an important motivating example.

\begin{example}\label{ex:bin_rel_new}
{\em Let $X$ be a non-empty set. Let $A\subseteq X\times X$ be a transitive and reflexive relation on $X$. 
By $\mathcal{P}(A)$ we denote the powerset of $A$. 
Since $A$ is transitive, $\mathcal{P}(A)$ is closed with respect to composition.
Since $A$ is reflexive, $\mathcal{P}(A)$ contains all subrelations of the identity relation. 
In particular, it is a monoid whose identity $e$ is the identity relation on $X$. 
As our set of projections, $E$,  we select all subrelations of $e$.
Clearly, $E$ is a commutative submonoid.
If $a \in \mathcal{P}(A)$, define 
$$\lambda (a) = \{(x,x) \in X \times X \colon  \exists y \in X \text{ such that } (y,x) \in a\}$$
and define
$$\rho (a) = \{(y,y) \in X \times X \colon  \exists x \in X \text{ such that } (y,x) \in a\}.$$
By construction, $\lambda (a),\rho (a)\in E$.
Observe that $a = \rho (a)a = a\lambda (a)$.
We show that $\lambda(\lambda(a)b)=\lambda(ab)$. 
Let $(x,x)\in \lambda(\lambda(a)b)$.
Then $(y,x)\in \lambda(a)b$ for some $y\in X$.
Hence $(y,y)\in\lambda(a)$ and $(y,x)\in b$.
Then $(t,y)\in a$ for some $t\in X$, and thus $(t,x)\in ab$.
It follows that $(x,x)\in\lambda(ab)$ and so $\lambda(\lambda(a)b)\subseteq\lambda(ab)$.
The reverse inclusion is shown similarly. 
Combined with a dual argument, we have shown that $\mathcal{P}(A)$ is an Ehresmann monoid.}
\end{example}

Ehresmann semigroups were explicitly introduced as a class of semigroups having a very close connection with categories.
Let $S$ be an Ehresmann semigroup with respect to $E$.
Define a partial binary operation on $S$, called the {\em restricted product}, as follows:
$a \cdot b$ exists precisely when $\lambda (a) = \rho (b)$ in which case it is equal to $ab$.
The proof of the following is straightforward.

\begin{proposition}\label{prop:restricted_product} Let $S$ be an Ehresmann semigroup with respect to $E$.
Then $S$ is a category with respect to the restricted product with set of identities $E$.
The semigroup product can be defined in terms of the category product since
$$ab = (ae) \cdot (eb)$$
where $e = \lambda (a)\rho (b)$.
\end{proposition}

It is possible to formalize precisely the categories which arise in this way.
See~\cite{Law1}.

An additional structure defined on Ehresmann semigroups that will play an important role in our work is the following.
On any Ehresmann semigroup $S$, define the relation $a \leq b$  if and only if 
\begin{equation*}\label{eq:def_order}
a=bf=eb \text{ for some } e,f\in E.
\end{equation*} 

\begin{lemma}\label{lem:aaab}\mbox{}
\begin{enumerate}
\item \label{i:aaa1} $a \leq b$ if and only if $a = \rho (a)b = b \lambda (a)$.
\item \label{i:aaa2} The relation $\leq$ is a partial order.

\end{enumerate}
\end{lemma}
\begin{proof} \eqref{i:aaa1} Only one direction needs proving.
Let $a = eb = bf$ for some projections $e$ and $f$.
Then $ea = a$ and so  $e \rho (a) = \rho (a)$ by (ES6) and (ES4).
Thus $$a = \rho (a)eb = e \rho (a)b = \rho (a)b.$$
The dual result follows by symmetry.

\eqref{i:aaa2} Since $a = \rho (a)a = a\lambda (a)$ we have that $a \leq a$.
Suppose that $a \leq b$ and $b \leq c$.
Then $a = \rho (a)b$ and $b = \rho (b)c$.
Thus $a = \rho (a)\rho (b)c$.
But $\rho (a) \rho (b) = \rho (a)$ and so $a = \rho (a)c$.
By symmetry we get  $a \leq c$.
Suppose that $a \leq b$ and $b \leq a$.
Then by definition
$a = \rho (a)b = b \lambda (a)$ and $b = \rho (b)a = a \lambda (b)$.
But then $a = \rho (a)b = \rho (a)\rho (b)a = \rho (b)a$.
Hence $a = b$, as required.
\end{proof}

The relation $\leq$ is called the {\em natural partial order} on $S$.

\begin{remark}
{\em The notation used to denote natural partial order on a restriction semigroup (defined later on, just before Lemma~ \ref{lem:apr4}) will vary
throughout this paper for the sake of convenience. We shall frequently use $\leq$ or $\leq'$.}
\end{remark}

\begin{lemma}\label{lem:e} Let $S$ be an Ehresmann semigroup with respect to $E$.
\begin{enumerate}

\item If $E$ has a maximum element $e$, then $e$ is the unit and $E=e^{\downarrow}$ with respect to the natural partial order $\leq$.

\item If $S$ is a monoid then the unit $e$ is the maximum projection so that $E=e^{\downarrow}$.

\end{enumerate}
\end{lemma}
\begin{proof} (1)  Let $a\in S$. 
We have $a=a\lambda(a)=a\lambda(a)e=ae$, and similarly $a=ea$. 
This proves that $e$ is a multiplicative unit. 
Assume $a\leq e$. 
Then $a=ef=ge$ for some $f,g\in E$. 
This shows that $a\in E$ and thus $E=e^{\downarrow}$.

(2) We prove that the unit $e$ is a projection.
Let $f \in E$.
Then $f = ef$ and so $f = \lambda (f) = \lambda (ef) = \lambda (e)f$ and,
since $\lambda (e)$ and $f$ are both projections, 
they commute and so $f=\lambda (e)f = f \lambda (e)$. Thus $f\leq \lambda(e)$.
Since units are unique, $e = \lambda (e)$, and so $e$ the maximum projection.
\end{proof}

In this paper, we shall only be interested in Ehresmann monoids.
By the lemma above, the set of projections of an Ehresmann monoid is precisely the set of elements
below the unit with respect to the natural partial order.
It follows that we do not have to spell out the set of projections explicitly.

We shall now define a special class of Ehresmann semigroups.
The following definition is from \cite{CGH}.
An element $a \in S$ of an Ehresmann semigroup is said to be {\em bi-deterministic} if 
for all $f \in E$ we have that 
\begin{equation*}
fa = a \lambda (fa) \,\, \text{ and }\,\, af = \rho (af)a.
\end{equation*}
A semigroup $S$ is called a {\em restriction semigroup with respect to the set} $E$ if $S$ is an Ehresmann semigroup with respect to $E$ and, in addition,
every element is bi-deterministic. The key feature of a restriction semigroup is that projections can, in some sense, be moved through elements from
left-to-right and from right-to-left.

In this section, denote the natural partial order in a restriction semigroup by~$\leq$.

\begin{lemma}\label{lem:apr4} Let $S$ be a restriction semigroup and $a,b\in S$.
\begin{enumerate}
\item \label{i:apr4a}  $a\leq b$ if and only if $a=b\lambda(a)$ if and only if $a=\rho(a)b$.
\item \label{i:apr4b} If $a\leq b$ and $\lambda(a)=\lambda(b)$ (or $\rho(a)=\rho(b)$) then $a=b$.
\item \label{i:apr4c} $af, fa\leq a$ for all $a\in S$ and $f\in E$.
\end{enumerate}
\end{lemma}

\begin{proof} 
\eqref{i:apr4a} Assume that $a=b\lambda(a)$. Then $a=\rho(b\lambda(a))b$ and thus $a\leq b$. 
\eqref{i:apr4b} We have $a=b\lambda(a)=b\lambda(b)=b$. \eqref{i:apr4c} We have $\lambda(af)=\lambda(\lambda(a)f)=\lambda(a)f$. It follows that $af=a\lambda(a)f=a\lambda(af)$ and so $af\leq a$ by part \eqref{i:apr4a}.
\end{proof}

We now record the fact that restriction semigroups are {\em partially ordered semigroups} with respect to their natural partial orders. This will be an important point when we come to construct quantales from restriction semigroups later.

\begin{lemma}\label{lem:aaa3} Let $S$ be a restriction semigroup.
The natural partial order  is compatible with the multiplication on the left and on the right.
\end{lemma}
\begin{proof} 
Let $a \leq b$.
Then $a = \rho (a)b$ and thus  $ca = c \rho (a)b=\rho (c\rho(a))cb$.
Thus $ca\leq cb$ by part \eqref{i:apr4a} of Lemma \ref{lem:apr4}.
The result now follows by symmetry.
\end{proof}

The following result is stated in \cite{CGH} and the proof is immediate from the definitions.

\begin{lemma} Let $S$ be an Ehresmann semigroup with respect to $E$.
Then the set of bi-deterministic elements of $S$ contains $E$ and forms a restriction semigroup with respect to $E$.
\end{lemma}

Two elements $a,b$ of  a restriction semigroup will be called {\em compatible},
denoted by $a \sim b$,
if $a\lambda(b)=b\lambda(a)$ and $\rho(b)a=\rho(a)b$. 
A non-empty set  $A$ of elements is said to be \emph{compatible} if any two elements in $A$ are compatible.
As already the case for inverse semigroups, the compatibility relation is not, in general, transitive.

\begin{lemma}\label{lem:aug27} Let $S$ be a restriction semigroup and $a,b,c,d\in S$. If $a\sim c$ and $b\sim d$ then  $ab\sim cd$. That is to say, the compatibility relation is compatible with multiplication from the right and from the left.
\end{lemma}

\begin{proof} Assume that $a\sim c$ and $b\sim d$. We calculate
\begin{align*}
ab\lambda(cd) = ab\lambda(\lambda(c)d)& = ab\lambda(d)\lambda(\lambda(c)d) & \text{since } \lambda(\lambda(c)d)\leq \lambda(d)\\
& = ad\lambda(b)\lambda(\lambda(c)d) & \text{since } b\sim d\\
&  ad\lambda(\lambda(c)d)\lambda(b) &  \text{since projections commute}\\
& = a\lambda(c)d\lambda(b) & \text{using } p\lambda(q)=q \text{ for } p\geq p\\
& = a\lambda(c)b\lambda(d) & \text{since } b\sim d
\end{align*}
and by symmetry we also have $cd\lambda(ab)=c\lambda(a)d\lambda(b)$. Since $a\sim c$ and $b\sim d$ we obtain the equality $ab\lambda(cd)=cd\lambda(ab)$. The dual equality follows by symmetry, so that the elements $ab$ and $cd$ are compatible.
\end{proof}

The following lemma shows that being compatible is a necessary condition to have an upper bound or a join.

\begin{lemma}\label{lem:aug28} Let $S$ be a restriction semigroup and $A\subseteq S$. 
\begin{enumerate}
\item If $c\geq a$ for all $a\in A$ then $A$ is a compatible family.
\item If $\bigvee A$ exists in $S$ then $A$ is a compatible family.
\end{enumerate}
\end{lemma}

\begin{proof} (1) Suppose that $c\geq a,b$. 
Then $c\geq a\lambda(b)$ and $c\geq b\lambda(a)$.  
From the first inequality we obtain 
$a\lambda(b)=c\lambda (a\lambda(b)) c\lambda (\lambda(a)\lambda(b))=c\lambda(a)\lambda(b)$.
Similarly, $b\lambda(a)=c\lambda(a)\lambda(b)$. 
Therefore, $a\lambda(b)=b\lambda(a)$. 
The equality $\rho(b)a=\rho(a)b$ follows by symmetry. 
It follows that $a$ and $b$ are compatible.
(2) follows from (1).
\end{proof}

The following lemma shows that functions $\lambda$ and $\rho$ are monotone. We use this fact throughout the paper without further mention.

\begin{lemma}\label{lem:mon22}
$a\leq b$ implies $\lambda(a)\leq \lambda(b)$ and $\rho(a)\leq \rho(b)$.
\end{lemma}

\begin{proof}
We have $a=b\lambda(a)$, and so $\lambda(a)=\lambda(b\lambda(a))=\lambda(b)\lambda(a)\leq \lambda(b)$. The inequality $\rho(a)\leq \rho(b)$ follows by symmetry.
\end{proof}

A restriction semigroup is said to be {\em ample} if $ac = bc$ implies $a\rho (c) = b \rho (c)$,
and dually.

\begin{example}\label{ex:bin_rel}
{\em We now return to Example~\ref{ex:bin_rel_new}.
Let $X$ be a non-empty set with at least two elements.
Let $A\subseteq X\times X$ be a transitive and reflexive relation on $X$,
and $\mathcal{P}(A)$ the powerset of $A$.  
Assume that there are distinct $i,j\in X$ such that $(j,i)\in A$. 
The relation $t=\{(i,i),(j,i)\}\in \mathcal{P}(A)$  is not bi-deterministic.
Indeed, let $t' = \{(j,i)\}$. 
We have $\lambda (t') = \{(i,i)\}$ and $\rho (t') = \{(j,j)\}$. 
Observe that $\rho(t')t= t'$ whereas $t\lambda(\rho(t')t) =  t\lambda(t')=t$.  
It follows that unless $A$ is the identity relation, $\mathcal{P}(A)$ is not a restriction monoid.

Denote by $\mathsf{I}(A)$ the subset of $\mathcal{P}(A)$ consisting
of all those relations $a$ that also satisfy the condition  that $(y,x),(y,x') \in a$ or $(x,y),(x',y) \in a$ implies that $x = x'$.
That is, $\mathsf{I}(A)$ consists of all partial bijections that are contained in $\mathcal{P}(A)$.
This is closed under multiplication and forms a restriction monoid.
It is easy to check that for any  $a,b,c\in \mathcal{I}(A)$ we have that
$ac=bc$ implies that $a\rho(c)=b\rho(c)$ and $ca=cb$ implies that $\lambda(c)a=\lambda(c)b$. 
Thus $\mathsf{I}(A)$ is, in fact, ample.
It is an inverse monoid precisely when the relation $A$ is symmetric.}
 \end{example}
 
In this paper, restriction semigroups will play the role that inverse semigroups played in \cite{LL1,LL2,Re} and, in fact,
a number of aspects of inverse semigroup theory can be easily generalized to restriction semigroups.

\begin{lemma} Let $S$ be a restriction semigroup with respect to $E$ and  $a,b\in E$. 
\begin{enumerate}
\item  $a\wedge b$ always exists in $S$, it is a projection, and equals $ab$. 
\item If $a\vee b$ exists in $S$ then it is a projection and is equal to the join $a\vee_E b$ of $a$ and $b$ in $E$.
\end{enumerate}
\end{lemma}
\begin{proof} (1) This follows from the easy result that the set of projections forms an order ideal in any restriction semigroup.

(2) Suppose that $a\vee b$ exists in $S$. 
Then since $a=\lambda(a)$ and $b=\lambda(b)$,
we have that 
$a,b \leq \lambda(a\vee b)$. 
It follows that $a\vee b\leq \lambda(a\vee b)$.
Thus $a \vee b$ is a projection.
The proof of the remaining claim is immediate.
\end{proof}

\begin{example} {\em The fact that $a\vee_E b$ exists does not in general imply that $a\vee b$ exists. Let $X=\{1,2,3,4\}$ and consider the following elements of the symmetric inverse monoid ${\mathcal I}_X$: $e=(1)(2)(3)(4)$, $a=(1)(2)(34)$, $f=(1)(2](3](4]$, $g=(1](2)(3](4]$ and $0=(1](2](3](4]$ (we used the standard cycle-chain notation, see \cite{GM}). Then $S=\{e,a,f,g,0\}$  is a restriction (and even inverse) semigroup with respect to $E=\{e,f,g,0\}$. We have $e=f\vee_E g$ but $f\vee g$ does not exist in $S$.} \end{example}

We now prove some results dealing with arbitrary non-empty joins in restriction semigroups.

\begin{lemma}\label{lem:joins22} Let $S$ be a restriction semigroup and $A\subseteq S$.
Suppose that  $\bigvee A$ and $\bigvee_{a\in A} \lambda(a)$ exist. Then
$$
 \lambda \left (\bigvee A\right )=\bigvee_{a\in A} \lambda(a)
$$
and dually.
\end{lemma}
\begin{proof} 
For each $a\in A$ from $a \leq \bigvee A$, we have that $\lambda (a) \leq  \lambda (\bigvee A)$.
Thus 
$\bigvee_{a\in A} \lambda(a) \leq \lambda (\bigvee A)$.
Put $x=(\bigvee A)(\bigvee_{a\in A} \lambda(a))$. 
We have $x\geq a$ for all $a\in A$, and so $x\geq \bigvee A$. 
On the other hand  $\bigvee A \geq x$, and so we have the equality
$x=\bigvee A$. 
Therefore, 
$$
\lambda\left (\bigvee A\right ) = \lambda (x) = \lambda\left (\bigvee A\right )\bigvee_{a\in A} \lambda(a)\leq \bigvee_{a\in A} \lambda(a).
$$
\end{proof}
  
\begin{lemma}\label{lem:molly} Let $S$ be a restriction semigroup with respect to $E$
and  assume that all  (respectively, all  finite) compatible joins exist.
Then multiplication distributes over arbitrary  (respectively, arbitrary finite) compatible joins if and only if the semilattice of projections $E$ is a frame
(respectively, a distributive lattice).
\end{lemma}
\begin{proof} Only one direction needs proving, and we shall also only prove the frame version since the other version is similar.
We begin with a special case.
Let $e$ be a projection and suppose that $\bigvee A$ exists.
We prove first that $$
e \left( \bigvee A \right) = \bigvee_{a\in A}ea.
$$
For each $a\in A$ from $a \leq \bigvee A$ we obtain $ea \leq e \left( \bigvee A \right)$.
By Lemma~\ref{lem:aug28} 
$\bigvee_{a\in A} ea$ exists and clearly $\bigvee_{a\in A} ea \leq e \left( \bigvee A\right)$.
We now calculate $\rho (\bigvee_{a\in A} ea)$ and $\rho (e \left( \bigvee A \right))$.
We use Lemma~\ref{lem:joins22} that $\rho$ preserves any joins that exist, 
and then use the fact that we have infinite distributivity for projections.
We deduce that  $$\rho \left(\bigvee_{a\in A} ea\right) = \rho \left(e \left( \bigvee A \right)\right)$$
and so  $e \left( \bigvee A \right) = \bigvee_{a\in A} ea$, as required.
We now deal with the general case.
We prove that 
$b \left( \bigvee A \right) = \bigvee_{a\in A} ba$.
As before, we quickly show that
$\bigvee_{a\in A} ba \leq b \left( \bigvee A \right)$.
We now calculate $\lambda$ of both sides and use our previous result to show that
$$ \lambda \left(\bigvee_{a\in A} ba\right) = \lambda \left(b \left( \bigvee A\right) \right).$$ 
It follows that 
$b \left( \bigvee A\right) = \bigvee_{a\in A} ba$, as required.
\end{proof}

Let $S$ be a restriction semigroup with respect to $E$. 
We say that it is {\em distributive} if $E$ is a distributive lattice and $S$ has joins of all non-empty finite compatible subsets.
We say that  it  is \emph{complete} if $E$ is a frame and $S$ has joins of all non-empty compatible subsets.
By Lemma \ref{lem:e},
it follows that multiplication distributes over finite and arbitrary non-empty joins, respectively.
Inverse semigroups are an important class of restriction semigroups.
A complete restriction monoid that is inverse is called a {\em pseudogroup}\footnote{Resende uses the term {\em abstract complete pseudogroup}.}.
Marco Grandis \cite{G1} defined totally cohesive categories which in our terminology  would be complete restriction {\em categories}.

\begin{proposition}\label{prop:meets}
Let $S$ be a complete restriction monoid. 
Then any non-empty family of elements of $S$ has a meet in $S$.
\end{proposition}
\begin{proof}
We adapt the argument to be found in the proof of \cite[part (2) of Proposition 2.10]{Re}. 
Let $E$ be the frame of projections of $S$.
Let $T$ be a family of elements of $S$. 
Define
$$G=\{g\in E\colon gx=gy \text{ for all } x,y\in T\}$$
and put $f= \bigvee G$. 
For any $s,t \in T$ we then have $fs=ft$ by distributivity.
It follows that if we define $z=fs$, where $s\in T$,
then $z$ is independent of the choice of $s\in T$.
Clearly, $z \leq t$ for all $t \in T$. 
Observe that $\rho(z)=\rho(fs)=\rho(f\rho(s))=f\wedge \rho(s)$. 
Now let $w\in T$ be such that $w \leq t$ for all $t \in T$. 
Then $w=\rho(w)t$ for all $t\in T$, and so  $\rho(w)\leq f$. 
Thus $w=\rho(w)z$ giving $w\leq z$ and so $z=\bigwedge T$.
\end{proof}

\subsection{Ehresmann quantales and quantal frames}

Let $Q$ be a sup-lattice equip\-ped with a binary associative multiplication operation denoted by concatenation.
Recall that $Q$ is called a {\em quantale} provided that the multiplication distributes over any joins.
That is, for any $a_i, i\in I$, and $b$ in $Q$ we have
$$
b \left( \bigvee_{i}  a_i \right)  = \bigvee_{i} b a_i 
\mbox{ and }
\left(\bigvee_{i} a_i \right) b= \bigvee_{i} a_i  b.
$$
In this subsection, the underlying partial order of $Q$ is denoted by $\leq$.
The following useful fact is a direct consequence of the definition and will be used many times in what follows without further reference.

\begin{lemma}\label{lem:compat} 
Let $Q$ be a quantale. Then its underlying partial order $\leq$ is compatible with quantale multiplication. 
\end{lemma}

A quantale $Q$ is called {\em unital} if with respect to multiplication it is a monoid. 
 
 \begin{lemma}\label{lem:l1} Let $Q$ be a quantale with unit $e$.
 Assume that $e^{\downarrow}$ consists of idempotents. 
Then for $a,b\in e^{\downarrow}$ we have that $ab=a\wedge b$. Hence $e^{\downarrow}$ is a frame. In particular, it is commutative. 
\end{lemma}
\begin{proof} 
Let $f,g\leq e$. 
We have $fg\leq fe=f$ and similarly $fg\leq g$. 
It follows that $fg\leq f\wedge g$. 
Assume $x\leq f,g$. 
Then $x^2\leq fg$. 
Therefore, $fg=f\wedge g$.
\end{proof}

Let $Q$ be a unital quantale with the top element $1=1_Q$ and unit $e=e_Q$.
We say that $Q$ is an {\em Ehresmann quantale} if it is equipped 
with maps $\lambda$, $\rho \colon Q\to e^{\downarrow}$ such that
the following axioms hold:
\begin{enumerate}[(EQ1)]
\item \label{l1} $\lambda$ and $\rho$ preserve arbitrary joins.
\item \label{l2} if $a\leq e$ then $\lambda(a)=\rho(a)=a$. 
\item \label{l3} $a = \rho(a)a$ and $a =  a\lambda(a)$ for all $a\in Q$.
\item \label{l4} $\lambda(ab)=\lambda(\lambda(a)b)$ and $\rho(ab)=\rho(a\rho(b))$ for all $a,b\in Q$.
\end{enumerate}
It is immediate by (EQ1) that $\lambda$ and $\rho$ are monotone, a fact we shall use many times in what follows.
Observe that $\lambda (1) = e = \rho (1)$. The axioms imply that $e^{\downarrow}$ consists of idempotents and thus
is a commutative subsemigroup of idempotents by Lemma \ref{lem:l1}. Therefore, under multiplication $Q$ is an Ehresmann monoid with respect to $e^{\downarrow}$.

\begin{remark} {\em Since $\leq$ is compatible with quantale multiplication, we have that $fa, af\leq a$ for any $f\leq e$ and $a\in Q$. We shall use this fact many times without further mention.}
\end{remark}

\begin{example}\label{ex:support}
{\em Any stably supported quantale in the sense of Resende \cite{Re} is an Ehresmann quantale
because if $\zeta$ is a stable support, then we may put $\rho=\zeta$ and $\lambda=\zeta \circ (-)^*$. 
This example motivated our paper.
For more details, see Proposition~\ref{prop:support}.}
\end{example}

Let $Q$ be a quantale. 
We will call it  a {\em quantal frame} if its underlying sup-lattice is a frame. 
To be precise,  for any elements $a$, $b_i$, where $i\in I$, of $Q$, the following equality holds
$$
a\wedge \left (\bigvee_{i} b_i\right )=\bigvee_{i} \left (a\wedge b_i\right ).
$$
By an {\em Ehresmann quantal frame} we mean an Ehresmann quantale that carries the structure of a quantal frame. 
We emphasize that an Ehresmann quantal frame carries three types of structure: 
of an Ehresmann monoid with respect to $e^{\downarrow}$ such that $\lambda$ and $\rho$ are sup-maps, 
of a unital quantale, 
and of a frame.

\begin{example} \label{ex:bin_rel1} {\em Observe that the Ehresmann semigroup ${\mathcal{P}}(A)$ from Example \ref{ex:bin_rel_new} is partially ordered by subset inclusion.
Being a powerset, it is a complete atomic Boolean algebra.
In particular it is a frame.
It is easy to see that the multiplication of binary relations distributes over the join, since the latter coincides with the union, and so  ${\mathcal{P}}(A)$ is a quantal frame.
It is evident that $\lambda$ and $\rho$ preserve joins.
Hence ${\mathcal{P}}(A)$ is an example of an Ehresmann quantal frame.}
\end{example}

\subsection{Partial isometries}
An Ehresmann quantale $S$ is equipped with two partial orders:
\begin{itemize}

\item  The underlying order of $S$ as a sup-lattice which in this subsection we denote by $\leq$. 

\item The natural partial order of its underlying Ehresmann semigroup which in this subsection we denote by $\leq'$.
\end{itemize}

The relationship between these two orders turns out to be very important.

\begin{lemma}\label{lem:l14} \mbox{} Let $S$ be an Ehresmann quantale with identity $e$.
\begin{enumerate} 
\item \label{i:a1} The order $\leq$ is a refinement of $\leq'$. That is, $a\leq' b$ implies $a\leq b$.
\item \label{i:a2} If $b\leq' e$, then $a\leq b \Leftrightarrow a\leq' b$. 
\end{enumerate}
\end{lemma}

\begin{proof}
\eqref{i:a1} Let  $a\leq' b$. Then $a=b\lambda(a) \leq be=b$.  
\eqref{i:a2} follows follows from  Lemma~\ref{lem:l1}.
\end{proof} 

\begin{example}\label{ex:bin_rel3} {\em Returning to Example \ref{ex:bin_rel},   we have $a\leq b$ in $\mathcal{P}(A)$ if and only if $a\subseteq b$. Assume that $A$ is not the identity relation. Then there are distinct $i$ and $j$ such that either $(i,j)\in A$ or $(j,i)\in A$.  Assume that $(i,j)\in A$. Consider the element $t= \{(i,i),(i,j)\}\in {\mathcal{P}}(A)$. Of course, $t'\leq t$. But $t'\not\leq' t$ since for any $f\in E$ the product $ft$ equals either $t$ or $0$.}
\end{example}

We now single out an important class of elements. Let $a\in S$. We say that $a$ is {\em partial isometry} if $b \leq a$ implies that $b \leq' a$. Hence $a$ is a partial isometry if and only if for any $b\in S$ we have that $b\leq a$ if and only if $b\leq' a$.  The set of partial isometries in $Q$ is denoted by ${\mathcal{PI}}(Q)$. By part \eqref{i:a2} of Lemma \ref{lem:l14} we have that $e^{\downarrow}\subseteq {\mathcal{PI}}(Q)$.

\begin{example}\label{ex:bin_rel2} {\em It is easy to see that in $\mathcal{P}(A)$ (see Examples \ref{ex:bin_rel},~\ref{ex:bin_rel1}) the set of partial isometries coincides with the set of all bi-deterministic elements and equals~the set of partial bijections $\mathsf{I}(A)$.}
\end{example}

In general, we have the following.

\begin{lemma} 
Every partial isometry is bi-deterministic.
\end{lemma}

\begin{proof} Let $a$ be a partial isometry and let $f$ be any projection.
Then $fa \leq a$ and so, since $a$ is a partial isometry, we have that $fa \leq'  a$.
Thus $fa = a\lambda (fa)$.
The dual result follows by symmetry.
\end{proof}

The converse, however, does not hold in general.

\begin{example}\label{ex:pm} {\em Let $M$ be a monoid with the unit $e$. The powerset ${\mathcal{P}}(M)$ is a monoid with respect to subset multiplication with the unit $\{e\}$. We define $E$ to be $\{e\}^{\downarrow}$ and functions $\lambda$, $\rho$: ${\mathcal{P}}(M)\to E$ given by $A\mapsto \{e\}$ where $A\neq\varnothing$, and $\varnothing\mapsto \varnothing$. It is straightforward to verify that ${\mathcal{P}}(M)$ becomes an Ehresmann quantal frame. All elements are trivially bi-deterministic. But partial isometries are only singletons and the zero. }
\end{example}

\begin{remark}
{\em In this paper, partial isometries play a more important role than bi-deterministic elements in general. 
The reason for this is that partial isometries encode the interconnection between the two orders, $\leq$ and $\leq'$, whereas bi-deterministic elements are defined in a coarser context of Ehresmann semigroups without any additional  partial order structure.}
\end{remark}

\begin{lemma}\label{le: order_ideal} The set of partial isometries of an Ehresmann quantale forms an order ideal (with respect to both of the orders $\leq$ and $\leq'$).
\end{lemma}

\begin{proof} We prove the statement for the order $\leq$. Then for $\leq'$ the statement also holds, since $\leq$ is finer. Let $a$ be a partial isometry and let $b \leq a$.
We prove that $b$ is a partial isometry.
Let $c \leq b$.
Then $c \leq a$.
Since $a$ is a partial isometry we have  $c = \rho (c)a = a \lambda (c)$.
Similarly, $b = \rho (b)a = a \lambda (b)$.
We also have that $\lambda (c) \leq \lambda (b)$ and $\rho (c) \leq \rho (b)$ since $\lambda$ and $\rho$ are sup-maps. 
Since $\leq$ and $\leq'$ coincide on $e^{\downarrow}$, 
it follows that $c = \rho (c)\rho (b)a = a \lambda (b) \lambda (c)$.
Thus $c \leq' b$, as required. 
\end{proof}

The set of partial isometries in an Ehresmann quantale is not, in general, closed under multiplication. 
Here is an example.

\begin{example}\label{ex:hogmannay}{\em
Let $S=\{1,x\}$ be a set. 
On ${\mathcal{P}}(S)$ we define a multiplication $\cdot$ as follows: $\{1\}$ is the unit, $\varnothing$ is the zero, $\{x\}\cdot \{x\} = \{x\}\cdot\{1,x\}=\{1,x\}\cdot \{x\}=\{1,x\}\cdot \{1,x\}=\{1,x\}$. 
It is straighforward to verify that ${\mathcal{P}}(S)$ is an Ehresmann quantal frame 
with set of projections $\{\varnothing, \{1\}  \}$
and maps $\rho$ and $\lambda$ sending all non-empty sets to $\{1\}$. 
The only non-projection partial isometry is $\{x\}$ but $\{x\} \cdot \{x\} = \{1,x\}$.}
\end{example}

\begin{proposition} \label{prop:pi6}
Let $Q$ be an Ehresmann quantale and let the set ${\mathcal{PI}}(Q)$ be closed under multiplication. Then ${\mathcal{PI}}(Q)$ is a complete restriction monoid.
\end{proposition}
\begin{proof}
Since partial isometries are bi-deterministic, ${\mathcal{PI}}(Q)$ is  a restriction monoid. Let $A\subseteq {\mathcal{PI}}(Q)$ be a compatible family. We show that $\bigvee A\in {\mathcal{PI}}(Q)$. Let $x\leq \bigvee A$ and show that $x=(\bigvee A)\lambda(x)$. Let $a\in A$. Since $x\wedge a\leq a$ and $\lambda(x\wedge a)\leq \lambda(x)$ it follows that 
\begin{equation}\label{eq:5apr}
x\wedge a=(x\wedge a)\lambda(x)\leq a\lambda(x).
\end{equation}
Similarly we obtain that $x\wedge a\leq x\lambda(a)$. On the other hand, $x\lambda(a)\leq x$ and
$x\lambda(a)\leq (a\vee b)\lambda(a)=a\vee b\lambda(a)=a\vee a\lambda(b)=a$. Therefore, $x\wedge a=x\lambda(a)$. Hence $\lambda(x\wedge a)=\lambda(x\lambda(a))=\lambda(x)\lambda(a)=\lambda(a\lambda(x))$. Hence, in view of \eqref{eq:5apr}, we obtain $x\wedge a=a\lambda(x)$ by part \eqref{i:apr4b} of Lemma \ref{lem:apr4}. It follows that
$$
x=\bigvee_{a\in A}(x\wedge a)=\bigvee_{a\in A}a\lambda(x)=\left (\bigvee A\right )\lambda(x),
$$
as required. The equality $x=\rho(x)(\bigvee A)$ is established similarly. We have proved that $\bigvee A\in {\mathcal{PI}}(Q)$.
\end{proof}

\subsection{The Quantalization Theorem: objects}\label{sub:qt}

We now begin the process of connecting semigroups to quantales.
Let $S$ be a restriction monoid.  
In this subsection we denote the natural partial order on $S$ by $\leq$.  
Define ${\mathcal L}(S)$ to be the set of all order ideals of $S$. 

\begin{proposition}\label{prop:dec18} Let $S$ be a restriction monoid.
Then with respect to subset multiplication,  ${\mathcal L}(S)$ is an Ehresmann quantal frame with unit $e^{\downarrow}$.  
In addition, the map $\eta: S\to {\mathcal L}(S)$ given by $s\mapsto s^{\downarrow}$ is an injective monoid homomorphism which preserves $\lambda$ and $\rho$. 
In particular, $S$ is isomorphic to $\eta(S)$ as a restriction monoid.
\end{proposition}
\begin{proof} Let $A,B\in {\mathcal L}(S)$. 
We show that $AB\in {\mathcal L}(S)$. 
Assume that $x\in AB$ and $y\leq x$. 
Then $x=ab$ where $a\in A$ and $b\in B$ and we may write $y=x\lambda(y)=ab\lambda(y)=a\cdot b\lambda(y)$. 
Note that $b\lambda(y)\leq b$ as $\lambda(y)\leq e$ and $\leq$ is compatible with multiplication. 
Hence  $b\lambda(y)\in B$, because $B$ is an order ideal. 
It follows that ${\mathcal L}(S)$ is a semigroup. 
We show that $AE=A$ for any $A\in {\mathcal L}(S)$. 
Let $a\in A$. We have $a=ae\in AE$. 
So $A\subseteq AE$. Let $af\in AE$. 
Then $af\leq a$ and so $af\in A$ since $A$ is an order ideal. 
Similarly one shows that $EA=A$.

It is clear that the map $\eta$ is injective. 
Let us show that $s^{\downarrow}t^{\downarrow}=(st)^{\downarrow}$. 
If $a\leq s$ and $b\leq t$ then $ab\leq st$ as $\leq$ is compatible with multiplication. 
If $x\leq st$ then $x=stf=s\cdot tf$. 
Since $tf\leq t$ we obtain that $x\in s^{\downarrow}t^{\downarrow}$.  

Note that ${\mathcal L}(S)$ is equipped with a natural partial order given by subset inclusion. 
It is straightforward to verify that it is in fact a frame where the meet and join operations coincide with set intersection and union, respectively.
It is now almost immediate that ${\mathcal L}(S)$ is a unital quantal frame with the unit  $E=e^{\downarrow}$.  

We define the functions $\overline{\lambda}, \overline{\rho}: {\mathcal L}(S) \to {\mathcal P}(e^{\downarrow})$ as follows:
\begin{equation}\label{eq:lam}
\overline{\lambda}(A)=\{\lambda(a)\colon a\in A\}\,\, \text{ and } \,\, \overline{\rho}(A)=\{\rho(a)\colon a\in A\}.
\end{equation}
Note that $\overline{\lambda}(A), \overline{\rho}(A)\in {\mathcal L}(S)$. 
To verify this, e.g., for $\overline{\lambda}(A)$ we let $b\leq \lambda(a)$ for some $a\in A$. 
Then $b\leq e$. We then have $b=\lambda(a)b=\lambda(\lambda(a)b)=\lambda(ab)$. 
It remains to note that $ab\leq ae=a$ and so $ab\in A$ as $A$ is an order ideal.
It is easy to verify that axioms (EQ1)--(EQ4) are satisfied. 
For example, we verify the equality $A=A\overline{\lambda}(A)$ which is a part of (EQ3). 
If $a\in A$ then $a=a\lambda(a)\in  A\overline{\lambda}(A)$. 
So $A\subseteq A\overline{\lambda}(A)$. 
Conversely, assume that $x\in A\overline{\lambda}(A)$. 
Then $x=a\lambda(b)$ with $a,b\in A$. 
But $a\lambda(b)\leq a$ yielding $x\in A$ since $A$ is an order ideal. 
It follows that $A\overline{\lambda}(A) \subseteq A$. 

We finally show that $\overline{\lambda}(\eta(s))=\eta(\lambda(s))$ for all $s\in S$. Clearly, $\lambda(s^{\downarrow})\subseteq (\lambda(s))^{\downarrow}$. Let $f\leq \lambda(s)$. 
Then $f=\lambda(sf)$, proving that $(\lambda(s))^{\downarrow}\subseteq \lambda(s^{\downarrow})$. 
The equality $\overline{\rho}(\eta(s))=\eta(\rho(s))$ follows by symmetry. 
\end{proof}

We now recall the following well-known fact (see, e.g., \cite[Proposition 2.1]{Re}) that ${\mathcal L}(S)$ is a join-completion of $S$ looked at as a poset.

\begin{lemma}\label{lem:4dec} For each monotone map $f$ from $S$ to a sup-lattice $L$ there is a unique join-preserving map $\overline{f}: {\mathcal{L}}(S)\to L$ such that $f=\overline{f}\eta$. The map $\overline{f}$ is given by $\overline{f}(U)=\bigvee_{x\in U}f(x)$.
\end{lemma}

Assume that $S$ is a complete restriction monoid. Note that in general the homomorphism $\eta: S\to {\mathcal L}(S)$ does not preserve existing joins, even joins of projections. 
Indeed, let $f,g\leq e$ be such that $f\not\leq g$ and $g\not\leq f$. 
Then $\eta(f\vee g) = (f\vee g)^{\downarrow}$ and $\eta(f)\vee \eta(g)=f^{\downarrow} \cup g^{\downarrow}$. 
We see that, as $f\vee g$ is neither below $f$ nor below $g$, the inclusion $\eta(f\vee g)\subseteq \eta(f)\vee \eta(g)$ does not hold. 
Note, however, that the reverse inclusion $\eta(f)\vee \eta(g)\subseteq \eta(f\vee g)$ always holds.
Our aim now will be to `correct' this drawback of ${\mathcal L}(S)$. 

We recall the material on closure operators on sup-lattices and quantales and their role in characterizing quotients. See \cite{Ros} for more details.

Let $L$ be a sup-lattice. 
A {\em closure operator} on $L$ is a map $j:L\to L$, which is (i) {\em monotone}: $a\leq b$ implies $j(a)\leq j(b)$; (ii) {\em extensive}: $id_L\leq j$ and (iii) {\em idempotent}: $j^2=j$. 
If $j: L\to L$ is a closure operator then the set of {\em fixed points} $L_j=\{a\in L \colon j(a)=a\}$ is a sup-lattice under the same order and $j$ is a surjective homomorphism, so that $L_j$ is a quotient of $L$.
Let  $\wedge^j$ and $\vee^j$ denote the meet and join operations on $L_j$. 
For any $a,b\in L_j$ we have $a\wedge^j b=a\wedge b$ and $a\vee^j b=j(a\vee b)$. 
Moreover, any quotient of $L$ arises in the described way: if $f:L\to M$ is a surjective homomorphism of sup-lattices then $M\simeq L_j$ for the closure operator $j=f_*f$ where $f_*$ is the right adjoint of $f$.

Let now $Q$ be a quantale. 
A {\em quantic nucleus} on $Q$ is a closure operator $j:Q\to Q$ such that $j(a)j(b)\leq j(ab)$ for all $a,b\in Q$. 
The quotient sup-lattice $Q_j=\{a\in Q \colon j(a)=a\}$ is then a quantale with the multiplication $\cdot^j$ given by $a\cdot^j b=j(a\cdot b)$ where $\cdot$ is the multiplication on $Q$. 
The map $j$ is a surjective homomorphism of quantales, and $Q_j$ is a quotient quantale of $Q$. 
Any quotient quantale of $Q$ arises in this way. 
The notion of a {\em frame nucleus} is a special case of the above definition: if $F$ is a frame then a {\em frame nucleus} on $F$ is a closure operator $j:F\to F$ such that $j(a)\wedge j(b)\leq j(a\wedge b)$ for all $a,b\in F$. If $j$ is a nucleus then $F_j=\{a\in F \colon j(a)=a\}$ is a quotient frame of $F$ and $j:F\to F_j$ is a surjective homomorphism of frames. 
Any quotient frame of $F$ arises in this way.

The definitions above can be unified as follows. 
Let $L$ be a sup-lattice and $f$ a binary operation on $L$. 
We define a {\em nucleus on} $L$ {\em with respect to} $f$ as a closure operator $j$ on $L$ such that $f(j(a),j(b))\leq j(f(a),f(b))$ for all $a,b\in L$. 
We also make use of a modification of this definition to unary operations. 
Let $f:L\to L$ be a unary monotone  operation on a sup-lattice $L$. 
A {\em nucleus on} $L$ {\em with respect to} $f$  is a closure operator $j$ on $L$ such that $f(j(a))\leq j(f(a))$. 
This definition can be extended to operations of greater arities but we will not need this in this paper.

\begin{lemma}\label{lem:aug29} Let $L$ be a sup-lattice, and $f:L\to L$ be a unary monotone operation on $L$. 
Assume that $j:L\to L$ is a nucleus on $L$ with respect to $f$. 
On the set of fixed-points $L_j$ we define the operation $f^j$ via $f^j(a)=j(f(a))$. 
Then $j$ maps $f$ to $f^j$ in the sense that 
\begin{equation}\label{eq:aug291}
j(f(a))=f^j(j(a))
\end{equation} 
for all $a\in L$.
\end{lemma} 

\begin{proof} 
Since $a\leq j(a)$ and $jf$ is monotone, we have $jf(a)\leq jf(j(a))$. 
For the converse inequality,  note that
$f(j(a))\leq j(f(a))$ since $j$ is a nucleus. 
Since $j$ is monotone, the latter inequality is stable under applying $j$. 
So, as $j$ is idempotent, we obtain $j(f(j(a)))\leq j(f(a))$, as required.
\end{proof}

Let $S$ be a {\em complete} restriction monoid. 
A subset $A$ of $S$ will be called {\em $\vee$-closed} provided that if $X\subseteq A$ is a compatible family of elements of $S$ then $\bigvee X \in A$. 
Let ${\mathcal L}^{\vee}(S)$ denote the set of all $\vee$-closed order ideals of $S$. 
For $A\in {\mathcal L}(S)$ we set 
\begin{equation}\label{eq:j29}
j(A)=\left \{\bigvee X\colon X\subseteq A \text{ where } X \text{ is compatible}\right \}.
\end{equation}
Note that $\bigvee X$ as above exists in $S$, since $X$ is compatible.
It is easy to verify (or see \cite[p.179]{Re}) that $j(A)$ is the smallest element of ${\mathcal L}^{\vee}(S)$ that contains $A$. It is immediate that $j$ is a closure operator on ${\mathcal L}(S)$. 

Let us consider the functions $\overline{\lambda}$, $\overline{\rho}$ defined in \eqref{eq:lam} as unary operations on ${\mathcal L}(S)$. 
It is immediate that both $\overline{\lambda}$ and $\overline{\rho}$ are monotone.

\begin{theorem}[Ehresmann Quantal Frame ${\mathcal L}^{\vee}(S)$]\label{th:lvees}\mbox{}
\begin{enumerate}
\item \label{i:lvees1} The closure operator $j$ is a quantic and frame nucleus on ${\mathcal L}(S)$. 
In addition, it is a nucleus with respect to both $\overline{\lambda}$ and $\overline{\rho}$.
\item \label{i:lvees2} The quotient quantal frame ${\mathcal L}(S)_j={\mathcal L}^{\vee}(S)$ is an Ehresmann quantal frame where the left and right support are given by the functions $\overline{\lambda}^j$ and $\overline{\rho}^j$.
\end{enumerate}
\end{theorem}

\begin{proof} \eqref{i:lvees1} It can be verified that $j$ is both a quantic and a frame nucleus repeating the arguments to be found in the proof of \cite[Lemma 3.23]{Re}. 
Let us show that $j$ is a nucleus with respect to $\overline{\lambda}$ (note that the arguments from \cite{Re} are not applicable since they rely on the involution). 
Since all elements of $\overline{\lambda}(A)$ are projections, any of its subset is compatible, implying that $j(\overline{\lambda}(A))=(\bigvee\{\lambda(a)\colon a\in A\})^{\downarrow}$. 
On the other hand, we have
\begin{align*}\overline{\lambda}(j(A))& =  \overline{\lambda}\left(\left \{\bigvee X\colon X\subseteq A, \, X \text{ is compatible}\right \}\right) & \text{by } \eqref{eq:j29}\\
& = \left\{\lambda \left (\bigvee X\right )\colon X\subseteq A, \, X \text{ is compatible}\right \} & \text{by } \eqref{eq:lam}\\
& = \left\{\bigvee_{x\in X}\lambda(x)\colon X\subseteq A, \, X \text{ is compatible}\right\}.
\end{align*}
We show that the latter set is a subset of $(\bigvee\{\lambda(a)\colon a\in A\})^{\downarrow}$. That is, we need to see that if $X\subseteq A$ is compatible, then $\bigvee\{\lambda(x)\colon x\in X\} \leq \bigvee\{\lambda(a)\colon a\in A\}$. 
But the latter trivially holds.

 \eqref{i:lvees2} We first note that the multiplicative identity of ${\mathcal L}^{\vee}(S)$ is $e^{\downarrow}$. 
We now show that the fucntions $\overline{\lambda}^j$ and $\overline{\rho}^j$ on ${\mathcal L}^{\vee}(S)$ satisfy the axioms (EQ1)--(EQ4). 
We do the verifications for $\overline{\lambda}^j$ only, because for $\overline{\rho}^j$ they are analogous.

(EQ1) holds.  Let $X\subseteq {\mathcal L}^{\vee}(S)$ and verify that $\overline{\lambda}^j(\bigvee^j X)=\bigvee^j \{\overline{\lambda}^j(a)\colon a\in X\}$:
\begin{align*}
\overline{\lambda}^j \left (\bigvee X \right ) & = \overline{\lambda}^j j \left (\bigcup X \right ) & \text{by the definition of }\bigvee^j\\
& = j\overline{\lambda}\left (\bigcup X\right ) & \text{by }\eqref{eq:aug291}, \text{ since } j  \text{ is a nucleus with respect to } \overline{\lambda};
\end{align*}
\begin{align*}
\bigvee^j_{a\in X} \overline{\lambda}^j(a) & = \bigvee^j_{a\in X} j \left (\overline{\lambda}(a)\right ) & \text{by the definition of }\overline{\lambda}^j\\
& = j\left (\bigcup_{a\in X} j(\overline{\lambda}(a))\right )& \text{by the definition of }\bigvee^j\\
&=jj\overline{\lambda}\left (\bigcup X \right ) & \text{since } j \text{ and } \overline{\lambda} \text{ are sup-maps}\\
& = j\overline{\lambda}\left (\bigcup  X \right ) & \text{since } j \text{ is idempotent}.
\end{align*}

(EQ2) holds. Assume that $A\in {\mathcal L}^{\vee}(S)$ and  $A\subseteq  e^{\downarrow}$. Note that $A$ consists only of projections of $S$. 
Since any set of projections is clearly compatible and $A$ is $\vee$-closed, $A=a^{\downarrow}$ for some $a\in E$. 
Then we have 
$$\overline{\lambda}^j(a^{\downarrow})=j\overline{\lambda}(a^{\downarrow})=j(a^{\downarrow})=a^{\downarrow},$$
and (EQ2) follows.

The equalities in (EQ3) and (EQ4) follow because $j$ is a nucleus both with respect to $\overline{\lambda}$ and the multiplication, and the needed equalities involve only two these operations. 
To make this precise, we now verify that (EQ4) holds, leaving the easier calculation for (EQ3) to the reader.

We need to make some preparation. 
Note that since $j$ is a nucleus with respect to the multiplication and by the definition of $\cdot^j$, for any $a,b\in {\mathcal L}(S)$
we have the equalities
\begin{equation}\label{eq:aug292}
j(a\cdot b)=j(a)\cdot^j j(b)=j(j(a)\cdot j(b))
\end{equation} Then, applying idempotency of $j$, we may write
\begin{equation}\label{eq:aug293}
j(j(a)\cdot b)=j(jj(a)\cdot j(b))= j(j(a)\cdot j(b))=j(a)\cdot^j j(b)=j(a\cdot b).
\end{equation}

We now turn to verifying that (EQ4) holds. We need to see that $$\overline{\lambda}^j(a\cdot^j b)=\overline{\lambda}^j(\overline{\lambda}^j(a)\cdot^j b).$$

Indeed, we have
\begin{align*}
\overline{\lambda}^j(a\cdot^j b) & = \overline{\lambda}^j j (a\cdot b) =  j\overline{\lambda} (a\cdot b)  &  \text{by } \eqref{eq:aug291} \text{ for } \overline{\lambda}\\
& =  j\overline{\lambda} (\overline{\lambda}(a)\cdot b) &   \text{by (EQ4) }  \text{for } \overline{\lambda};
\end{align*}
\begin{align*}
\overline{\lambda}^j(\overline{\lambda}^j(a)\cdot^j b) & = \overline{\lambda}^j j(j\overline{\lambda}(a)\cdot b) &  \text{by the definition of }\cdot^j \text{ and } \overline{\lambda}^j\\
& = \overline{\lambda}^j j(\overline{\lambda}(a)\cdot b)  & \text{by  } \eqref{eq:aug293}\\
& = j \overline{\lambda}(\overline{\lambda}(a)\cdot b)  &  \text{by } \eqref{eq:aug291} \text{ for } \overline{\lambda}.
\end{align*}
This completes the proof.
\end{proof}

We will call ${\mathcal L}^{\vee}(S)$ the {\em enveloping Ehresmann quantal frame} of $S$. 
From now on, we will work with ${\mathcal L}^{\vee}(S)$, and not with ${\mathcal L}(S)$. 
To avoid complicated notation, we take a convention to denote $\overline{\lambda}^j$, $\overline{\rho}^j$ and 
$\cdot^j$ just by $\overline{\lambda}$, $\overline{\rho}$ and $\cdot$, respectively. The set ${\mathcal L}^{\vee}(S)$
is partially ordered by subset inclusion, the meet operation $\wedge^j$ on it is just subset intersection $\cap$, whereas  $A\vee^j B$ is $j(A\cup B)$. 
In what follows, we will usually denote the meet and join in ${\mathcal L}^{\vee}(S)$ just 
 by $\wedge$ and $\vee$, respectively. 
This will not cause ambiguity as long as it is clear to which sets the elements, whose join is being considered, belong. Regardless of this, for the reader's convenience, we will sometimes explain the usage of $\wedge$ and $\vee$ in words. 

Observe that $\eta(s) = s^{\downarrow}\in {\mathcal L}^{\vee}(S)$ for any $s\in S$.  
So there is a well-defined map $\eta': S\to {\mathcal L}^{\vee}(S)$, given by $\eta'(s)=\eta(s)$, which is just $\eta$, whose range is restricted to ${\mathcal L}^{\vee}(S)$. 
From now on, we will work only with the map $\eta'$, and we will denote this map just by $\eta$, in order to simplify notation.

\begin{proposition}\label{prop:lvees} Let $S$ be a complete restriction monoid.
$\eta:S\to {\mathcal L}^{\vee}(S)$ is an injective monoid homomorphism, which preserves the unary operations and all joins that exist  in $S$. 
\end{proposition}

\begin{proof} In view Proposition \ref{prop:dec18}, we need only show that $\eta$ preserves existing joins.
Let $X$ be a compatible family of elements of $S$ and let $a=\bigvee X$. 
We need to show that $\eta(a)=\bigvee_{x\in X}\eta(x)$ (the latter join is in ${\mathcal L}^{\vee}(S)$). For all $x\in X$ the inequality $x\leq a$ implies that $\eta(x)\leq \eta(a)$. Hence $\bigvee_{x\in X}\eta(x)\subseteq \eta(a)$. 
For the reverse inclusion it suffices to show that $a\in  \bigvee_{x\in X}\eta(x)$. 
Indeed, then $a^{\downarrow}\subseteq \bigvee_{x\in X}\eta(x)$ since the latter set is an order ideal.
We have $\bigvee_{x\in X}\eta(x)=j(\bigcup_{x\in X} x^{\downarrow})$. 
Since $X$ is a compatible subset of $\bigcup_{x\in X} x^{\downarrow}$, from \eqref{eq:j29} we obtain $a=\bigvee X\in  j(\bigcup_{x\in X} x^{\downarrow})$, as required.
\end{proof}

We will also need the following universal property of ${\mathcal L}^{\vee}(S)$ as a sup-lattice.

\begin{proposition}\label{prop:un_pr} Let $L$ be a sup-lattice and $f: S\to L$ a monotone map which preserves all joins that exist in $S$ (these are precisely the joins of compatible families). 
Then there is a unique sup-map $\overline{f}: {\mathcal L}^{\vee}(S)\to L$, which agrees with $f$ on $S$ in that $f=\overline{f}\eta$. 
The map $\overline{f}$ is given by $\overline{f}(A)=\bigvee_{a\in A}f(a)$.
\end{proposition}

\begin{proof} The statement follows from Lemma \ref{lem:4dec} and the property of ${\mathcal L}^{\vee}(S)$ to be the maximum quotient of ${\mathcal L}(S)$, determined by the condition that compatible joins are preserved under the embedding  $\eta: S\to {\mathcal L}^{\vee}(S)$.
\end{proof}

We now come to significant properties of the Ehresmann quantal frames of the form  ${\mathcal L}^{\vee}(S)$.

\begin{proposition}\label{prop:p_units} Let $S$ be a complete restriction monoid.
\begin{enumerate}

\item \label{i:pu1a} The partial isometries of  ${\mathcal L}^{\vee}(S)$ are precisely the elements of the form $\eta(s)$ where $s\in S$.
Consequently, partial isometries of  ${\mathcal L}^{\vee}(S)$ form a complete restriction monoid isomorphic to $S$.

\item \label{i:pu1b} Each element of ${\mathcal L}^{\vee}(S)$ is a join of partial isometries.

\end{enumerate}
\end{proposition}

We will make use of the following lemma.

\begin{lemma}\label{lem:aug294} Let $F\in {\mathcal L}^{\vee}(S)$ and $F\subseteq e^{\downarrow}$. 
Then $F=f^{\downarrow}$ for some $f\leq e$.
\end{lemma}
\begin{proof} Note that $F$ consists only of projections and any two projections are compatible. 
Therefore, as $F$ is $\vee$-closed, $\bigvee F\in F$. 
If we set $f=\bigvee F$ then clearly $F=f^{\downarrow}$.
\end{proof}

\begin{proof}[Proof of Proposition \ref{prop:p_units}]
 \eqref{i:pu1a} 
Let  $s\in S$. 
Assume $A\subseteq s^{\downarrow}$ for some $A\in {\mathcal L}^{\vee}(S)$. 
Since  $s^{\downarrow}$ forms a compatible family of elements so does $A$. 
As $A$ is $\vee$-closed, we have that $\bigvee A\in A$. 
It follows that $A=t^{\downarrow}$ for some $t\leq s$. 
Therefore,  $t=fs=sg$, where $f,g$ are projections. 
It follows that $t^{\downarrow}=\eta(t)=\eta(f)\eta(s)=\eta(s)\eta(g)$. 
It follows that $\eta (s)$ is a partial isometry.

We now prove the converse.
Let $A\in {\mathcal L}^{\vee}(S)$ be a partial isometry.
We show that $A=\eta(s)$ for some $s\in S$. 
Let $a\in A$. 
We have $\eta(a)=a^{\downarrow}\subseteq A$ since $A$ is an order ideal. Until the end of this proof by $\circ$ we denote the multiplication in ${\mathcal L}^{\vee}(S)$, and by $\cdot$ the multiplication in ${\mathcal L}^{\vee}(S)$ which is the usual subset multiplication.
By the definition of a partial isometry and by Lemma \ref{lem:aug294}, we may write
$$\eta(a)=\eta(f)\circ A=A\circ \eta(g)$$ for some $f,g\leq e$.
We show that $\lambda(a)\leq g$. 
We have $\eta(a)=A\circ \eta(g) = j(A\cdot \eta(g))$.
If $x\in A\cdot \eta(g)$ then $x=xg$ and so $\lambda(x)\leq g$.
If  $x\in j(A \cdot \eta(g))$  then $x=\bigvee X$, where $X\subseteq A \cdot \eta(g)$ is a compatible family. 
It follows that 
$$\lambda(x)=\lambda\left (\bigvee X \right ) =\bigvee_{x\in X} \lambda(x)\leq g.$$ 
Since the latter inequality is established for an arbitrary $x\in A\circ \eta(g)$ and $a\in \eta(a)=A\circ \eta(g)$,  we obtain $\lambda(a)\leq g$. 
By symmetry we have that $\rho(a)\leq f$.

Now let $b\in A$ be arbitrary. 
Since $\lambda(a)\in \eta(g)$, we have $b\lambda(a)\in A\cdot \eta(g) \subseteq A\circ \eta(g)$. 
Thus $b\lambda(a)\in \eta(a)$ which means that $b\lambda(a)\leq a$. It follows that 
 $$b\lambda(a)=a\lambda(b\lambda(a))=a\lambda(\lambda(b)\lambda(a))=a\lambda(b)\lambda(a)=a\lambda(b).$$
 Similarly, from $\rho(a)\in \eta(f)$, we obtain the equality $\rho(a)b=\rho(b)a$. 
 It follows that $a\sim b$.
 We have proved that $a\sim b$ for any $a,b\in A$. 
 Thus $\bigvee A$ exists and thus is the top element of $A$ since $A$ is $\vee$-closed. 
 This implies that $A=\eta(\bigvee A)$.
 
\eqref{i:pu1b} Let  $A\in {\mathcal L}^{\vee}(S)$ and show that
\begin{equation*}\label{eq:o15}
A=\bigvee\left \{\eta(a)\colon a\in A\right \}.
\end{equation*}
If $a\in A$ then $a\in\eta(a)$, and so we have the inclusion  $A\subseteq \bigvee\{\eta(a)\colon a\in A\}$. 
For the reverse inclusion, 
let $y\in \bigvee\{\eta(a)\colon a\in A\}$. 
Then there is a compatible set  $Z\subseteq \bigcup\{\eta(a)\colon a\in A\}$ such that $y\in \bigvee Z$. 
But $z\in \eta(a)$ for $a\in A$ implies that $z\in A$ and $z\in \eta(z)$. 
So we may assume that  $y\in \bigvee_{x\in X} \eta(X)$ where $X\subseteq A$ and $\eta(x)$, $x\in X$, form a compatible family of elements. 
Since $\eta \colon S\to \eta(S)$ is an injective homomorphism of restriction monoids, compatibility  of the latter family implies that of the family $X$.  
Using \eqref{eq:j29} and the fact that $\eta$ preserves existing joins, we obtain 
$y\in \bigvee_{x\in X}\eta(x) = \eta\left (\bigvee X\right )$.
But $A$ is $\vee$-closed. Hence $y$ is below an element of $A$ yielding  $y\in A$, as required.
\end{proof}

We will call an Ehresmann quantal frame $Q$ {\em \'{e}tale} provided that $1$ is a join of partial isometries. It is easy to see that $Q$ is \'{e}tale if and only if every element in $Q$ is a join of partial isometries.
We now make our key definition. An Ehresmann quantal frame is said to be a {\em restriction quantal frame} if it is \'etale and the set of partial isometries is closed under multiplication.
If $S$ is a complete restriction monoid then Proposition \ref{prop:p_units} implies that ${\mathcal L}^{\vee}(S)$ is a restriction quantal frame. We now prove that any restriction quantal frame is of the form ${\mathcal L}^{\vee}(S)$ for some complete restriction monoid $S$.

\begin{proposition}\label{prop:11Jan} 
Every restriction quantal frame is isomorphic to a restriction quantal frame of the form
 ${\mathcal L}^{\vee}(S)$ for some complete restriction monoid $S$.
\end{proposition}
\begin{proof}
Let $Q$ be a restriction quantal frame.
Define 
$$\varepsilon_Q \colon {\mathcal L}^{\vee}({\mathcal{PI}}(Q))\to Q$$
by $\varepsilon(X)=\bigvee X$.
It is straightforward to verify that $\varepsilon_Q$ is a frame morphism and a morphism of Ehresmann monoids.
Let $s$ be a partial isometry of $Q$. 
Observe that $\varepsilon_Q(\eta(s))=s$, and so $\varepsilon_Q$ is surjective since $Q$ is \'{e}tale.
In addition, the restriction of $\varepsilon_Q$ to the elements of the form $\eta(s)$, where $s$ is a partial isometry, is injective.
Clearly, such elements form an order ideal with respect to the order that underlies the sup-lattice $Q$. 
It now follows from \cite[Proposition 2.2]{Re} that $\varepsilon_Q$ is injective. 
\end{proof}

\subsection{The Quantalization Theorem: morphisms}

To be able to formulate the main theorem of this section, we define appropriate morphisms between complete restriction monoids and between restriction quantal frames.  

Recall that a homomorphism of Ehresmann (or restriction) monoids is a semigroup homomorphism that commutes with the functions $\lambda$ and $\rho$. Let $S,T$ be complete restriction monoids, and $E_S$ and $E_T$ be their frames of projections,  respectively.
Let  $\varphi:S\to T$ be a map. 
We call $\varphi$ a {\em morphism of complete restriction monoids} if $\varphi$ is a homomorphism of restriction monoids and, 
restricted to $E_S$, 
is a frame morphism from $E_S$ to $E_T$.
We shall often refer to such maps as {\em morphisms}.

The following example shows that if $\varphi$ is a morphism of complete restriction monoids, it does not need to preserve binary meets. 
\begin{example} \label{ex:e1} {\em Let $G$ be a group with the identity $e$ and assume $|G|\geq 2$. 
Let $S$ be the group $G$ with the adjoined zero $0$.  
It is a complete restriction monoid with respect to $E=\{e,0\}$ with $\lambda$ and $\rho$ sending $0$ to $0$ and each non-zero element to $e$. 
Then the map $S\to S$, given $g\mapsto e$, $g\in G$, and $0\mapsto 0$,  is a morphism of complete restriction monoids. 
Let $g\in G$, $g\neq e$. Then $g\wedge e=0$ but at the same time
$0= \varphi (0)=\varphi(g\wedge e)\neq \varphi(e) \wedge \varphi(e)= e.$}
\end{example}

A morphism $\varphi:S\to T$ of complete restriction monoids will be called a $\wedge$-{\em morphism} if  $\varphi(s\wedge t)=\varphi(s)\wedge\varphi(t)$ for any $s,t\in S$. 
It will be called {\em proper}  if for any element $t\in T$, there is a subset $X\subseteq T$ such that $t=\bigvee X$ and for  each $x\in X$ there is $s\in S$ such that $\varphi(s)\geq x$.  By Lemma \ref{lem:aug28} such a set $X$, if it exists, is necessarily compatible. 
The morphism $\varphi$ from Example \ref{ex:e1} is not proper. 
Indeed,  let $g\in S$, $g\neq e$. 
Observe that $g$ is join-irreducible and there is no $x\in S$ with $\varphi(x)\geq g$.

The proof of the following adapts the arguments from \cite[Proposition 2.10]{Re}.

\begin{lemma}\label{lem:join_pres}
Let $\varphi:S\to T$ be a morphism of complete restriction monoids. 
Let $X$ be a compatible family of elements of $S$. 
Then the family of elements $\varphi(x)$, $x\in X$, is compatible and
$$
\varphi\left (\bigvee X \right )=\bigvee_{x\in X}\varphi(x).
$$
\end{lemma}

Let $Q_1$ and $Q_2$ be restriction quantal frames and $\varphi: Q_1\to Q_2$  a quantale morphism that is a morphism of Ehresmann monoids. The following example shows that $\varphi$ does not need to map partial isometries to partial isometries.

\begin{example} {\em Let $M_1$ be the semilattice $\{e,x\}$ with $e\geq x$, and $M_2$ the semilattice $\{e',a,b\}$ with $e'\geq a\geq b$. 
Consider the restriction quantal frames $Q_i={\mathcal P}(M_i)$, $i=1,2$, and
define a map $\varphi:Q_1\to Q_2$ by join-extension of the map $\{e\}\mapsto \{e'\}$, $\{x\}\mapsto \{a,b\}$. 
This is a morphism of quantales and a morphism of Ehresmann monoids but it does not map the partial isometry $\{x\}$ to a partial isometry.
}
\end{example}

Define a {\em morphism} $\varphi: Q_1\to Q_2$ between restriction quantal frames as a quantale morphism that is a morphism of Ehresmann monoids 
(in particular, $\varphi$ preserves both $\lambda$ and $\rho$) and satisfying the additional requirement that $\varphi$ maps partial isometries to partial isometries. 

We define {\em proper morphisms} as those which preserve the top element, $\wedge$-{\em mor\-phisms} as those which preserve finite non-empty meets.  Note that proper $\wedge$-morphisms preserve any finite meets (including the empty one) and are thus also frame morphisms.

We will be interested in the four types of morphisms between complete restriction monoids and restriction quantal frames:
\begin{itemize} 
\item type $1$:  morphisms; 
\item type $2$: proper morphisms; 
\item type $3$: $\wedge$-morphisms;
\item type $4$:  proper $\wedge$-morphisms.
\end{itemize}

\begin{theorem}[Quantalization Theorem] \label{th:quant_psgrps}
For each $k=1,2,3,4$ the category of complete restriction monoids and morphisms of type $k$ is equivalent to the category of restriction quantal frames and morphisms of type $k$.
\end{theorem}

\begin{proof} 
Let $S$ be a complete restriction monoid.
By Proposition~\ref{prop:p_units}, 
the map $\eta_S \colon S\to {\mathcal{PI}}({\mathcal L}^{\vee}(S))$ given by $s\mapsto \eta(s)$ 
is an  isomorphism of complete restriction monoids and so is both proper and a $\wedge$-morphism. 
Let $Q$ be a restriction quantal frame.
By Proposition~\ref{prop:11Jan}, 
the map $\varepsilon_Q \colon {\mathcal L}^{\vee}({\mathcal{PI}}(Q))\to Q$ given by $\varepsilon(X)=\bigvee X$ 
is an isomorphism of restriction quantal frames and so is both proper and a $\wedge$-morphism.

We now show that the two constructions above define the object parts of functors that describe an equivalence of categories.
Let 
$\varphi \colon S\to T$ be a morphism of complete restriction monoids.
Define
$\overline{\varphi}$ 
by
$$
\overline{\varphi}(A)=\bigvee _{a \in A} \varphi(a)^{\downarrow},
$$
where $A\in {\mathcal L}^{\vee}(S)$. 
It is straighforward to check that this defines a morphism of restriction quantal frames
and leads to a functor. 
If $\psi \colon Q\to R$ is a morphism between restriction quantal frames then
the restriction of $\psi$ to the partial isometries of $Q$ defines a morphism of complete restriction monoids.
It is easy to check that these two functors do indeed define an equivalence of categories.

We next prove that under this equivalence, proper morphisms correspond to proper morphisms.
The top element of ${\mathcal{L}}^{\vee}(S)$ is $S$, and the top element of ${\mathcal{L}}^{\vee}(T)$ is $T$. 
Thus  $\overline{\varphi}$ is proper if and only if  $\overline{\varphi}(S) = T$.
Let $\varphi\colon S\to T$ be proper and let $t\in T$. Let $X_t\subseteq T$ be a compatible set such that  $t=\bigvee X_t$ and for each $x\in X_t$ there is $s_x\in S$ satisfying $\varphi(s_x)\geq x$. Then 
$$
\overline{\varphi} \left(\bigvee\{s_x^{\downarrow}\colon x\in X_t, t\in T\}\right) = \bigvee \{\varphi(s_x)^{\downarrow}\colon s_x\in X_t, t\in T\} \geq \bigvee \{t^{\downarrow}\colon t\in T\}=T.
$$
Thus $\overline{\phi}(S)=T$. Conversely, assume that $\overline{\varphi}(S) = T$.
By the definition of $\overline{\varphi}$ this means that $\bigvee_{s\in S} \varphi(s)^{\downarrow} = \bigvee_{t\in T}t^{\downarrow}$. Let $t\in T$. Then we obtain
$\bigvee_{s\in S}(\varphi(s)\wedge t)^{\downarrow}=t^{\downarrow}$. It follows that $t=\bigvee_{s\in S} (\varphi(s)\wedge t)$ where we have $\varphi(s)\geq \varphi(s)\wedge t$.

We check that under the equivalence, $\wedge$-morphisms correspond to $\wedge$-morphisms.
Let $\varphi \colon S\to T$ be a $\wedge$-morphism. 
We show that $\overline{\varphi}$ preserves binary meets as well.
Let $A,B\in {\mathcal{L}}^{\vee}(S)$. 
Then 
$A = \bigvee_{a \in A} a^{\downarrow}$ 
and 
$B = \bigvee_{b \in B}b^{\downarrow}$.
We now make two observations.
First, we have that $A \wedge B = A \cap B$
and, since $S$ has all binary meets, 
$$A \cap B = \{a \wedge b \colon a \in A, b \in B \}.$$
Second, $\phi (a \wedge b)^{\downarrow} = \phi (a)^{\downarrow} \cap \phi (b)^{\downarrow}$.
The proof of the fact that $\overline{\varphi} (A \cap B) = \overline{\varphi} (A) \cap  \overline{\varphi} (B)$
now readily follows using the fact that our quantales are quantal frames.

It is now immediate that proper $\wedge$-morphisms correspond to proper $\wedge$-morphisms.
\end{proof}

\section{The Correspondence Theorem}\label{s:corresp}

\subsection{Localic categories from Ehresmann quantal frames}\label{s:s4}

The constructions in this section are an adaptation to our setting of the corresponding constructions from  \cite[Section 4]{Re}.

Let $Q$ be an Ehresmann quantal frame with structure maps $\lambda$ and $\rho$, and unit $e$.
Since both $Q$ and $e^{\downarrow}$ are frames, 
they define two locales: $C_1$ with $O(C_1)=Q$ and $C_0$ with $O(C_0)=e^{\downarrow}$. 
We shall turn these two locales into an object of arrows and an object of objects, respectively, of a localic category. 
In order to do this, we need to define structure maps. 
Our starting point will be the following maps each of which is a sup-map between frames and so has a right adjoint:
\begin{itemize}

\item $\nu \colon e^{\downarrow}\to Q$, the sup-lattice inclusion with right adjoint denoted by $u^{*}$.

\item $\lambda \colon Q \rightarrow e^{\downarrow}$, a  sup-lattice map with right adjoint denoted by $d^{*}$.

\item $\rho \colon Q \rightarrow e^{\downarrow}$, a sup-lattice map with right adjoint denoted by $r^{*}$.

\end{itemize}

\begin{lemma}\label{le: august} Let $Q$ be an Ehresmann quantal frame.
With the above definitions, we have the following.
\begin{enumerate}

\item\label{o151} The right adjoint $u^*$ of $\nu$ is given by $u^*(a)=a\wedge e$ where $a\in Q$. 
It preserves arbitrary joins and defines an open locale map $u \colon C_0\to C_1$ such that $\nu=u_!$. 

\item\label{o152} The right adjoint $d^*$ of $\lambda$ is given by $d^*(f)=1f$ where $f\in e^{\downarrow}$. 
It preserves arbitrary joins and defines an open  locale map $d \colon C_1\to C_0$ such that $\lambda = d_!$. 

\item\label{o153} The right adjoint $r^*$ of $\rho$ is given by $r^*(f)=f1$ where $f \in e^{\downarrow}$. 
It preserves arbitrary joins and defines an open  locale map $r \colon C_1\to C_0$ such that $\rho = r_!$.

\end{enumerate}
\end{lemma}

We will need the following lemma.

\begin{lemma}[Stability Lemma]\label{le:stability} 
Let $Q$ be an Ehresmann quantale.
For any $a\in Q$ and $f \in e^{\downarrow}$,
we have that $a \wedge 1f=af$
and 
$a \wedge f1 = fa$.
\end{lemma} 
\begin{proof}
From $f \leq e$ we get that $af \leq ae = a$.
From $a \leq 1$ we get that $af \leq 1f$.
It follows that $af \leq a \wedge 1f$.
Now let $c \leq a,1f$.
Then $\lambda (c) \leq \lambda (a)$
and $\lambda (c) \leq f$.
Thus $\lambda (c) \leq \lambda (a)f$.
But then $c = c \lambda (c) \leq a \lambda (a)f = af$.
It follows that $af = a \wedge 1f$.
The other claim follows by symmetry.
\end{proof}

\begin{proof}[Proof of Lemma \ref{le: august}] (1) We have to verify that for any $a\in Q$ and $f \in e^{\downarrow}$
$$
\nu(f)\leq a \quad \Longleftrightarrow \quad f \leq u^*(a).
$$
This reduces to
$$
f \leq a \quad \Longleftrightarrow \quad f \leq a\wedge e.
$$
The latter equivalence trivially holds since $f \leq e$. 
The claim about the preservation of joins is clear.
To see that $u$ is open, we check the Frobenius condition:
that is,  for any $f \in e^{\downarrow}$  and $b\in Q$ we check that $\nu(f)\wedge b=\nu(f\wedge u^*(b)).$ 
This reduces to
$$
f \wedge b=f \wedge (b\wedge e).
$$
The latter equality trivially holds, since $f \leq e$.

(2) We verify first that for any $a\in Q$ and $f \in e^{\downarrow}$:
$$
\lambda(a)\leq f \quad \Longleftrightarrow \quad a\leq d^*(f)
$$
which reduces to verifying that
$$
\lambda(a)\leq f \quad \Longleftrightarrow \quad a\leq 1f.
$$
If $\lambda(a)\leq f$ 
then  $a=a\lambda(a)\leq 1f$ since $a \leq 1$ and the order is compatible with the multiplication. 
Conversely, if $a\leq 1f$ then we may write $$\lambda(a)\leq \lambda(1f) =\lambda(\lambda(1)f)=ef=f.$$ 

The map $d^*$ preserves arbitrary joins since we are working in a quantale.

To show that $d$ is open,  we verify the Frobenius condition for $d$. 
This reduces to verifying that for any $a\in O(C_1)$ and $f\in O(C_0)$ we have
$$
\lambda(a\wedge 1f)=\lambda(a)\wedge f.
$$
We may write
\begin{align*}
\lambda(a\wedge 1f) & =\lambda(af) & \text{by Lemma \ref{le:stability}}\\
& =\lambda(\lambda(a)f) & \text{by (EQ4)}\\
& =\lambda(a)f = \lambda (a) \wedge f, & \text{since } a\wedge f\leq e \text{ and by  (EQ2)}
\end{align*}
which completes the proof.

(3) This follows by symmetry from part (2) above.
\end{proof}

We have therefore defined the {\em unit map}
$u \colon C_0\to C_1$, 
the {\em domain map} 
$d \colon C_1\to C_0$,
and the {\em codomain map} 
$r \colon C_1\to C_0$ 
of our putative localic category.

It remains to construct the composition map of our category. 
This depends on the following constructions and results.
Let $Q$ be an Ehresmann quantal frame.
Let $Q\otimes Q$ denote the coproduct of $Q$ with itself in the category of frames. 
The quantale multiplication is  a sup-map $Q \otimes Q \rightarrow Q$.
We may also regard the underlying quantale $Q$ as a right $e^{\downarrow}$-module and a left  $e^{\downarrow}$-module,
under multiplication,  and so we may construct the tensor product sup-lattice $Q\otimes_{e^{\downarrow}} Q$, which is also a frame because $Q\otimes Q$ is.

The following lemma is essentially the same as \cite[Lemma 4.5]{Re}.
The proof follows from the observation that the relations involved in the definition of the pushout
$$(a \wedge d^{*}(f)) \otimes b = a \otimes (r^{*}(f) \wedge b)$$
are equivalent by Lemma~\ref{le: august} to
$$(a \wedge 1f ) \otimes b = a \otimes (f1 \wedge b)$$
which by the Stability Lemma, Lemma~\ref{le:stability}, are equivalent to
$$af \otimes b = a \otimes fb.$$
The role of the Stability Lemma is to link multiplication by a projection with the meet operation.

\begin{lemma}[Bridging Lemma]\label{lem:l16}
The pushout of 
$d^{*} \colon e^{\downarrow} \rightarrow Q$ 
and 
$r^{*} \colon e^{\downarrow} \rightarrow Q$
in the category of frames is the frame $Q\otimes_{e^{\downarrow}} Q$
with the maps $\pi_{1}^{*}(a) = a \otimes 1$ 
and 
$\pi_{2}^{*}(a) = 1 \otimes a$.
\end{lemma}

Our goal is to construct a locale map $m \colon C_{1} \times_{C_{0}} C_{1} \rightarrow C_{1}$.
The locale of composable pairs $C_1\times_{C_0} C_1$ is the pullback of the following diagram
\begin{center}
\begin{tikzpicture}
\node (bl) {$C_1$};
\node (br) [node distance=4cm, right of=bl] {$C_0$};
\node (ul) [node distance=1.8cm, above of=bl] {$C_1\times_{C_0} C_1$};
\node (ur) [node distance=1.8cm, above of=br] {$C_1$};
\path[->]
(bl) edge node[below]{$d$} (br);
\path[->]
(ul) edge node[above]{$\pi_2$} (ur);
\path[->]
(ul) edge node[left]{$\pi_1$} (bl);
\path[->]
(ur) edge node[right]{$r$} (br);
\end{tikzpicture}
\end{center}
which is just the pushout of the following diagram in the category of frames
\begin{center}
\qquad \qquad
\begin{equation}
\begin{tikzpicture}[baseline=(current  bounding  box.center)]
\node (bl) {$O(C_1)$};
\node (br) [node distance=4cm, right of=bl] {$O(C_1 \times_{C_{0}} C_1)$};
\node (ul) [node distance=1.8cm, above of=bl] {$O(C_0)$};
\node (ur) [node distance=1.8cm, above of=br] {$O(C_1)$};
\path[->]
(bl) edge node[below]{$\pi_2^*$} (br);
\path[->]
(ul) edge node[above]{$d^*$} (ur);
\path[->]
(ul) edge node[left]{$r^*$} (bl);
\path[->]
(ur) edge node[right]{$\pi_1^*$} (br);
\end{tikzpicture}
\end{equation}
\end{center}
By the Bridging Lemma, we have that $O(C_1 \times_{C_{0}} C_1)$ coincides with $Q \otimes_{e^{\downarrow}} Q$. 
So in order to reach our goal, we have to construct a frame map 
$m^{*} \colon Q\to Q \otimes_{e^{\downarrow}} Q$.
But the quantale multiplication $Q\otimes Q\to Q$ factors through the tensor product $Q\otimes_{e^{\downarrow}} Q$  
since the multiplication is associative 
and thus respects the relations $af\otimes b = a \otimes fb$, where $f \leq e$. 
It follows that the quantale multiplication induces the sup-map
$\mu \colon  Q \otimes_{e^{\downarrow}} Q  \rightarrow Q$ given by $\mu(a\otimes b) = ab$. 
We require that this map be a direct image map $m_!$ of the locale multiplication $m: C_1\times_{C_0} C_1\to C_1$. 
This is equivalent to the requirement that $\mu$ have a right adjoint
$m^{*} \colon Q \rightarrow Q \otimes_{e^{\downarrow}} Q$
which is a frame map. 
Observe that this map is defined by
$$m^{*} (a) = \bigvee_{xy \leq a} x \otimes y.$$ 
Being a right adjoint, the map $m^{*}$ preserves arbitrary meets but not necessarily arbitrary joins. 
This implies that $m^*$ is a frame homomorphism if and only if it preserves arbitrary joins. 
We now give an example showing that $m^* $ does not preserve joins in general.

\begin{example} \label{ex:joins} {\em We use the Ehresmann quantal frame constructed in Example~\ref{ex:hogmannay}.
Note that $m^*$ is just the inverse image map $m^{-1}$. 
We have $$m^{-1}(\{1\})=\{1\}\otimes\{1\}, \,\,
m^{-1}(\{x\})=\{1\}\otimes \{x\}\cup \{x\}\otimes \{1\}.$$
We see that  $\{x\}\otimes \{x\} \not\in m^{-1}(\{1\})\cup m^{-1}(\{x\})$, but $\{x\}\otimes \{x\}\in m^{-1}(\{1,x\})$.}
\end{example}

\begin{remark} {\em Employing the notion of a multisemigroup \cite{KM}, the above example can be generalized as follows. 
Let $S$ be a  multisemigroup with the identity element $e$. 
Then, as pointed out in \cite{KM}, it induces a quantale structure on ${\mathcal P}(S)$ which is clearly a quantal frame. 
Assume that the multiplication in $S$ is at least single-valued. 
Then the maps $\lambda$ and $\rho$ sending all non-empty sets to $\{e\}$ and $\varnothing$ to itself turn ${\mathcal P}(S)$ into an Ehresmann quantal frame. 
As in the example above, one can show that $m^*$ preserves arbitrary joins if and only if  the multiplication is single-valued, that is $S$ is a semigroup. }
\end{remark}

We shall say that $Q$ is {\em multiplicative} if the map $m^*$ preserves arbitrary joins.

By a {\em quantal localic category} we will mean a localic category where the structure maps $d,r,u$ are open and $m$ is semiopen.

\begin{theorem}\label{th:th2} 
Let $Q$ be a multiplicative Ehresmann quantal frame. 
Then the maps $u,d,r,m$ define a quantal localic category.
\end{theorem}
\begin{proof} The proof follows the lines of the proof of \cite[Theorem 4.8]{Re} with appropriate changes and a few different arguments. 
Observe that $u,d,r$ are open and $m$ is semiopen. 
So it remains to verify that the equations for a category hold  with respect to these maps.

(Cat1) holds.
We show that $du=id$ holds, the equation  $ru=id$ holds by symmetry.
We need to verify $u^*d^*=id$. 
Let $f \leq e$.
Then 
$$
u^*d^*(f) 
=
1f\wedge e 
= 
ef
= f$$
by the construction of $u^*$ and $d^*$
and the Stability Lemma, Lemma \ref{le:stability}.

(Cat2) holds.
The following argument is taken from \cite[page 186]{Re}, but with the order reversed. 
We prove that $(u \times id)^{*}m^{*} = \pi_{2}^{*}$, the equation $(id^{*} \times u^{*})m^{*} = \pi_{1}^{*}$ holds by symmetry.
\begin{equation}\label{eq:def21}
\begin{tikzpicture}[baseline=(current  bounding  box.center)]
\node (ul) {$O(C_0)\otimes_{O(C_0)} O(C_1)$};
\node (ur) [node distance=5.5cm, right of=ul] {$O(C_1)\otimes_{O(C_0)} O(C_1)$};
\node (br) [node distance=1.8cm, below of=ur] {$O(C_1)$};
\path[->]
(ur) edge node[above]{$(u\times id)^*$} (ul);
\path[<-]
(ul) edge node[below]{$\pi_2^*$} (br);
\path[<-]
(ur) edge node[right]{$m^*$} (br);
\end{tikzpicture}
\end{equation}
Note that $(u\times id)^*=u^*\otimes id$, since $id^*=id$. 
Hence  $(u\times id)_!=u_!\otimes id$. 
Further, it is easy to verify that the map $\pi_2 \colon C_0\times_{C_0} C_1 \to C_1$ is an isomorphism with the map $\langle r, id\rangle$ being its inverse. 
It follows that the isomorphisms $\langle r, id\rangle^*$ and  $\pi_2^*$ form an adjoint pair.
In particular, $\langle r, id\rangle_!=\pi_2^*$. 
Since $ea=a$ for any $a\in Q$, we have
$$ea=m_!(u_!(1_{O(C_0)})\otimes a) = m_!(u_!\otimes id) (1_{O(C_0)}\otimes a) = m_!(u_!\otimes id) \pi_2^*(a).
$$
We therefore obtain the equality
$$
m_!(u_!\otimes id) \pi_2^*=id.
$$
Bearing in mind  that $\pi_2^*$ is an isomorphism, the latter means that $m_!(u_!\otimes id)$ is an isomorphism and, moreover, it is the inverse of $\pi_2^*$. 
Since $m_!(u_!\otimes id)=m_!(u\times id)_!$ and since an adjoint of an isomorphism is just its inverse, we obtain $\pi_2^*= (u\times id)^*m^*$, 
which means that the diagram \eqref{eq:def21} commutes as claimed.

(Cat3) holds. We verify that  $dm=d\pi_2$ holds, the equation $rm = r\pi_{1}$ holds by symmetry. 
Thus we need to verify 
that $m^*d^*(f)=\pi_2^*d^*(f)$ for any $f\leq e$. 
The latter equality is equivalent to the equality
$$
\bigvee_{xy\leq 1f} x\otimes y = 1\otimes 1f.
$$ 
Since $11f\leq 1f$ the join on the lefthand-side above is greater than or equal to $1\otimes 1f$. 
It remains to establish the reverse inequality. 
It suffices to show that if $xy\leq 1f$ then $x\otimes y\leq 1\otimes 1f$.
So assume that $xy\leq 1f$. 
Then $$\lambda(xy)\leq \lambda(1f)=\lambda(\lambda(1)f)=\lambda(f)=f.$$
But $\lambda(xy)=\lambda(\lambda(x)y)$. 
From $\lambda(\lambda(x)y)\leq f$ 
we get that $\lambda(x)yf=\lambda(x)y$. 
Then, applying the definition of the tensor product $Q\otimes_{e^{\downarrow}} Q$ and $x\lambda(x)=x$, 
we can write
$$
x\otimes y = x\otimes \lambda(x)y= x\otimes \lambda(x)yf= x\otimes yf \leq 1\otimes 1f.
$$

(Cat4) holds by repeating the arguments from \cite[page 187]{Re}.
\end{proof}

\subsection{Ehresmann quantal frames from quantal localic categories}

Let $C=(C_1,C_0)$ be a quantal localic category.
We shall show that the frame $O(C_{1})$  may be endowed with the structure of a multiplicative Ehresmann quantal frame. 

We first need to record the explicit formulae that determine the pushout frame
$O(C_1\times_{C_0} C_1)$
which we denote by  
$O(C_1)\otimes_{O(C_0)} O(C_1)$.
This is a quotient of the coproduct frame
$O(C_1)\otimes O(C_1)$ 
given by the map that identifies $\pi_2^*r^*(a)$ with $\pi_1^*d^*(a)$ for any $a\in O(C_0)$,
where  
$$\pi_2^*(a)=1\otimes a\,\,
\mbox{ and }
\,\,\pi_1^*(a)=a\otimes 1$$
for all $a\in O(C_0)$. 
Hence $O(C_1)\otimes_{O(C_0)} O(C_1)$ is determined by the identifications, for any $a\in O(C_0)$,
\begin{equation}\label{eq:jul10}
1\otimes r^*(a) = \pi_2^*r^*(a)=\pi_1^*d^*(a)=d^*(a)\otimes 1.
\end{equation}
In fact, these are determined by the {\em defining equations}
\begin{equation}\label{eq:comp}
(b\wedge d^*(a))\otimes c
=
b\otimes (r^*(a)\wedge c).
\end{equation}
The equalities \eqref{eq:jul10} are obtained from \eqref{eq:comp} by calculating the meet with $b\otimes c$, for any $b,c\in O(C_1)$,
whereas \eqref{eq:comp} are obtained from  \eqref{eq:jul10} by going to joins.
Hence $O(C_1)\otimes_{O(C_0)} O(C_1)$ is determined by the defining equations \eqref{eq:comp}, where $b,c$ run through $O(C_1)$ and $a$ through $O(C_0)$.

We now show how to define a binary multiplication operation on the set $O(C_{1})$.
The category multiplication is the map $m: C_1\times_{C_0} C_1 \to C_1$. 
By assumption it is semiopen, so its direct image map $m_!: O(C_1\times_{C_0} C_1) \to O(C_1)$ exists. 
Using Lemma \ref{lem:l16}, we may replace $O(C_1\times_{C_0} C_1)$ by $O(C_1)\otimes_{O(C_0)} O(C_1)$.
Define 
$$m_! \colon O(C_1)\otimes_{O(C_0)} O(C_1) \to O(C_1)$$ to be 
the left adjoint of $m^{*} \colon O(C_{1}) \rightarrow  O(C_{1}) \otimes_{O(C_{0})} O(C_{1})$
and so the direct image of $m$.
Being a left adjoint, $m_{!}$ is a sup-map.
Let $$q \colon O(C_1)\otimes O(C_1) \to O(C_1)\otimes_{O(C_0)} O(C_1)$$ be the quotient map that determines $O(C_1)\otimes_{O(C_0)} O(C_1)$. 
Then we define $m'=m_!q$, which is a map
from $O(C_1)\otimes O(C_1)$ to $O(C_1)$. 
Finally, for $a,b\in O(C_1)$ we set their product to be 
$$
ab=m'(a\otimes b).
$$
Observe that this multiplication is associative, which is established by repeating the arguments from \cite[page~187]{Re} but in the reverse order.

\begin{remark} {\em In what follows, we shall write $a \otimes b$ rather than $q (a \otimes b)$
with the understanding that we may apply the defining equations \eqref{eq:comp}.}
\end{remark}

By definition, the multiplication defined distributes over the join of $O(C_1)$ on both sides and thus $O(C_1)$ is a quantale. 
We make the following additional definitions:
\begin{itemize}

\item $\lambda=u_!d_!$. This map is defined as the composition of two left adjoints and so is itself a left adjoint.
It follows that it is a  sup-lattice map. 

\item $\rho=u_!r_!$.  This is a  sup-lattice map as above. Thus together we know that (EQ1) holds.

\item $e=u_!(1_{O(C_0)})$ where $1_{O(C_0)}$ is the maximum element of the frame $O(C_{0})$.

\end{itemize}

\begin{lemma}\label{lem:aux1} Let $C = (C_{1},C_{0})$ be a quantal localic category.
\begin{enumerate}

\item \label{i:aux2} $f\leq e$ if and only if $f=u_!(a)$ for some $a\in O(C_0)$.

\item \label{i:aux0} If $f \leq e$ then $\lambda(f)=f$ and $\rho(f)=f$. 

\item \label{i:aux0a} $e^{\downarrow}$ is a frame isomorphic to  $O(C_{0})$ with respect to the maps
$f \mapsto d_{!}(f)$ (or $f \mapsto r_{!}(f)$) and $a \mapsto u_{!}(a)$.

\item  \label{i:aux3} $b\wedge d^*(a)=bu_!(a)$ and $b\wedge r^*(a)=u_!(a)b$ for any $a\in O(C_0)$ and $b\in O(C_1)$.

\item \label{i:aux3a}  The meet in $e^{\downarrow}$ is just multiplication.

\item  \label{i:aux6a} $a\lambda(a)=a$ and $\rho(a)a=a$ for any $a\in O(C_1)$.

\item \label{i:aux4}  $O(C_1)$ is an $u_!(O(C_0))$-bimodule with respect to the actions by left and right multiplication. 
The pushout $O(C_1)\otimes_{O(C_0)} O(C_1)$ coincides with the tensor product of frames
$O(C_1)\otimes_{u_!(O(C_0))} O(C_1)$.

\item \label{i:aux13} In $O(C_1)\otimes_{O(C_0)} O(C_1)$ we have $1\otimes a = d^*r_!(a)\otimes a$ and $a\otimes 1=a\otimes r^*d_!(a)$ for any $a\in O(C_1)$.

\item \label{i:aux12} $\lambda(ab)=\lambda(\lambda(a)b)$ and  $\rho(ab)=\rho(a\rho(b))$ for any $a,b\in O(C_1)$.

\item\label{i:aux_new}  $u_!u^*(y)=y\wedge e$ for any $y\in O(C_1)$.
\end{enumerate}
\end{lemma}
\begin{proof} 

\eqref{i:aux2} Let $f\leq e$. 
Applying the Frobenius condition for the open map $u$, we obtain
$$f=u_!(1_{O(C_0)})\wedge f =u_!(1_{O(C_0)}\wedge u^*(f)).$$
Conversely, if $a \in O(C_0)$
then  $u_!(a)\leq u_!(1_{O(C_0)})=e$.

\eqref{i:aux0} Let $f \leq e$. By part \eqref{i:aux2} we have that 
 $f = u_!(a)$ for some $a\in O(C_0)$. 
Then $$\lambda(f)=u_!d_!u_!(a)=u_!(a)=f$$
where we have used the identity $du=id$.  
The equality $\rho(f)=f$ is proved similarly.

\eqref{i:aux0a} follows from part \eqref{i:aux0} and the equalities $du=id$ and $ru=id$.

\eqref{i:aux3} 
We prove the first equality only, for the other one the arguments are similar. 
Let $x=e\wedge r^*(a)$. Then $r_!(x)=r_!(e\wedge r^*(a))=1_{O(C_0)}\wedge a=a$. By part \eqref{i:aux0} we have $x=u_!r_!(x)$. It follows that $x=u_!(a)$.
Using the fact that $e$ is the multiplicative unit, we have
\begin{align*}
b\wedge d^*(a) &= m_!((b\wedge d^*(a)) \otimes e)  & \\
&= m_!(b\otimes (r^*(a)\wedge e)) & \text{by } \eqref{eq:comp}\\
&= m_!(b\otimes u_!(a)) = bu_!(a), & \text{since } e\wedge r^*(a)=u_!(a)
\end{align*}
as required.

\eqref{i:aux3a}  Let $a,b \leq e$.
Then by part \eqref{i:aux3}, we have $ab=au_!d_!(b)=a\wedge d^*d_!(b)$. In particular, $ab\leq e$. Applying part \eqref{i:aux0} and the Frobenius condition for $d$ we obtain
$$
ab=u_!d_!(ab)=u_!d_!(a\wedge d^*d_!(b))=u_!(d_!(a)\wedge d_!(b))=\lambda(a)\wedge \lambda(b)=a\wedge b.
$$


\eqref{i:aux6a} We prove that $a\lambda(a)=a$, and the equality involving $\rho$ is proved similarly. 
Let $x=a\lambda(a)$. Then by part \eqref{i:aux3}, we have $x=au_!d_!(a)=a\wedge d^*d_!(a)$.
Note that $d_!(x)=d_!(a\wedge d^*d_!(a))=d_!(a)\wedge d_!(a)=d_!(a)$. It now follows from part \eqref{i:aux3a} that $x=a$.

\eqref{i:aux4} By part \eqref{i:aux3}, the defining equations take the form 
 $bu_!(a)\otimes c=b\otimes u_!(a)c$
for any $a\in O(C_0)$, $b,c\in O(C_1)$.
This is enough to prove the result, since as $a$ runs through $O(C_0)$,  
the elements $u_!(a)$ run through $e^{\downarrow}=u_!(O(C_0))$ by part \eqref{i:aux2}. 

\eqref{i:aux13} We have
\begin{align*}
1\otimes a & =  1\rho(a)\otimes a  =  1u_!r_!(a)\otimes a\\
& = d^*r_!(a)\otimes a & \text{by part }\eqref{i:aux3}.
\end{align*}
The second equality is proved similarly.

\eqref{i:aux12}
In order to prove that $\lambda (ab) = \lambda (\lambda (a)b)$,
we need to show that 
$$
u_!d_!m_! = u_!d_!m_!(ud\otimes id)_!.
$$
But two maps which have right adjoints are equal if and only if their right adjoints are equal.
So, equivalently, we need to prove the equality
$$
m^*d^*u^*=(d^*u^*\otimes id)m^*d^*u^*.
$$ 
From (Cat3), we have that
$dm = d\pi_{2}$, which means $m^*d^*=\pi_2^*d^*$. 
In view of this, what we actually prove is that
$$
\pi_2^*d^*u^*=(d^*u^*\otimes id)\pi_2^*d^*u^*.
$$
Let $a\in O(C_1)$. Since $\pi_2^*(a)=1\otimes a$, we may write 
$$
\pi_2^*d^*u^*(a)=1\otimes d^*u^*(a).
$$
Now we calculate
\begin{align*}
(d^*u^*\otimes id)\pi_2^*d^*u^*(a) & = (d^*u^*\otimes id)(1\otimes d^*u^*(a)) & \\
& = (d^*u^*\otimes id)(d^*r_!d^*u^*(a)\otimes d^*u^*(a)) & \text{by part } \eqref{i:aux13}\\
& = d^*u^*d^*r_!d^*u^*(a) \otimes d^*u^*(a) & \\
& = d^*r_!d^*u^*(a) \otimes d^*u^*(a) & \text{since } u^*d^*=id \\
& = 1\otimes d^*u^*(a)  & \text{by part } \eqref{i:aux13},
\end{align*}
and we are done.

\eqref{i:aux_new} Since clearly $u_!u^*(y)\leq y,e$ then $u_!u^*(y)\leq y\wedge e$, and we are left to prove the reverse inequality.
By part \eqref{i:aux3} we have $u_!u^*(y)=d^*u^*(y)\wedge e$ so we need to prove that $y\wedge e\leq d^*u^*(y)\wedge e$ which is equivalent to $y\wedge e\leq d^*u^*(y)$. By adjointness this is equivalent to $u_!d_!(y\wedge e)\leq y$. But the latter inequality holds as $u_!d_!(y\wedge e)=y\wedge e$ by part \eqref{i:aux0a}.
\end{proof}

\begin{theorem}\label{th:th1} With the above definitions
$(O(C_1), e, \lambda, \rho)$ is a multiplicative Ehresmann quantal frame.
\end{theorem}
\begin{proof} By parts \eqref{i:aux0}, \eqref{i:aux6a} and  \eqref{i:aux12} of  Lemma \ref{lem:aux1}, we have shown that  $O(C_1)$,  with respect to the left support  $\lambda$, the right support $\rho$ 
and the identity element $e$, satisfies axioms (EQ2), (EQ3) and (EQ4), respectively, and so it is an Ehresmann quantal frame. 
By part \eqref{i:aux4} of Lemma \ref{lem:aux1}, the quantale multiplication $\mu$ has as domain the tensor product $O(C_1)\otimes_{u_!(O(C_0))} O(C_1)$. 
This multiplication is the direct image $m_!$ of $m$. 
This implies that $O(C_1)$ is multiplicative, which completes the proof.
\end{proof}

\subsection{Back and forth}

By Theorem \ref{th:th1} and Theorem \ref{th:th2}, we have the following two constructions.
\begin{itemize}

\item We may associate to a quantal localic category $C$ an Ehresmann quantal frame,  which we denote by ${\mathcal O}(C)$.  

\item We may associate to a  multiplicative Ehresmann quantal frame $Q$ a quantal localic category, which we denote  by  ${\mathcal C}(Q)$.

\end{itemize}
The following theorem shows that these two constructions establish a bijective correspondence.

\begin{theorem}[Correspondence Theorem] \label{th:jul17} \mbox{}
\begin{enumerate}
\item \label{o:1151} Let $Q$ be a multiplicative Ehresmann quantal frame. Then $Q = {\mathcal O}({\mathcal C}(Q))$. 
\item \label{o:1152} Let $C$ be a quantal localic category. Then $C\cong {\mathcal C}({\mathcal O}(C))$.
\end{enumerate}
\end{theorem}
\begin{proof} \eqref{o:1151} Let $Q$ be a multiplicative Ehresmann quantal frame. 
Then it is immediate by our constructions that $Q  = {\mathcal O}({\mathcal C}(Q))$. 

\eqref{o:1152} Let $C=(C_1,C_0)$ be a quantal localic category.  
By Theorem \ref{th:th1},  this gives rise to an Ehresmann quantal frame ${\mathcal O}(C)$. 
Adapting the arguments in  part of the proof of \cite[Theorem 5.11]{Re}, 
we shall prove that ${\mathcal O}(C)$ is multiplicative and that $C\cong {\mathcal C}({\mathcal O}(C))$. 
To the multiplicative Ehresmann quantal frame ${\mathcal O}(C)$, we can apply the procedure of Section~\ref{s:s4} and construct
a localic category $\mathcal{C}(\mathcal{O}(C)) = (\hat{C_1},\hat{C_0})$
where  
$O(\hat{C_1})=O(C_1)$,
and  
$O(\hat{C_0})=e^{\downarrow}$ 
the element $e$ being $u_!(1_{C_0})$.  
By part \eqref{i:aux0} of Lemma~\ref{lem:aux1},
the frame $e^{\downarrow}$ is isomorphic to the frame $O(C_{0})$
by the map 
$$\beta^* \colon O(C_{0}) \rightarrow e^{\downarrow}$$
which is the map $u_!$ with the range restricted to the frame $e^{\downarrow}$. 
We now show that the categories $C$ and ${\mathcal C}({\mathcal O}(C))$ are isomorphic. 
We claim that the pair  $\Phi = (id, \beta)$ is an isomorphism in the sense of an internal functor from ${\mathcal C}({\mathcal O}(C))$ to $C$.
Thus we have to verify (Fun1)--(Fun4).

(Fun1) holds.
By Lemma~\ref{le: august},
we may construct the maps 
$\hat{u} \colon \hat{C_0}\to \hat{C_1}$ 
and
$\hat{d} \colon \hat{C_1}\to \hat{C_0}$ 
and 
$\hat{r} \colon \hat{C_1}\to \hat{C_0}$. 
Let $\hat{m}$ denote the multiplication in the category $\mathcal{C}(\mathcal{O}(C))$.
The frame that underlies the locale of composable pairs of  ${\mathcal C}({\mathcal O}(C))$ is, by Lemma \ref{lem:l16}, the tensor product $O(C_1)\otimes_{e^{\downarrow}} O(C_1)$. 
The frame that underlies the locale of composable pairs of $C$ coincides with this by part \eqref{i:aux4} of Lemma \ref{lem:aux1}. 
It follows that $m=\hat{m}$ since both of them are semiopen maps with the same domains and codomains and a common direct image $m_!$.

We verify (Fun2), with (Fun3) holding by symmetry.
We have to show that $\hat{d}^{*}\beta^{*} = d^{*}$.
This follows from the following calculation
$$\hat{d}^{*}u_{!}(a) = 1u_{!}(a) = 1 \wedge d^{*}(a) = d^{*}(a).$$

It remains to prove (Fun4).
Thus we have to prove that $u_{!}u^{*} = \hat{u}^{*}$.
This follows from the following calculation
$$u_{!}u^{*}(a) = u_{!}(u^{*}(a) \wedge 1_{O(C_{0})}) = a \wedge u_{!}(1_{O(C_{0})}) = a \wedge e = \hat{u}^{*}(a).$$
\end{proof}

\subsection{The Etale Correspondence Theorem}
The goal of this subsection is to identify precisely which quantal localic categories correspond to restriction quantal frames.
Recall that in Section~\ref{sub:qt}, we defined an Ehresmann quantal frame to be \'etale if $1_{Q}$ is a join  of partial isometries. 

\begin{lemma}\label{le:12Jan} An Ehresmann quantal frame
$Q$ is \'etale if and only if both the structure maps $d$ and $r$ of the quantal localic category ${\mathcal C}(Q)$ are \'etale.
\end{lemma}

\begin{proof} 
Let $C={\mathcal C}(Q)$. By construction we have $O(C_1)=Q$ and $O(C_0)=e^{\downarrow}$. It follows in particular that $u_!(f)=f$ for all $f\leq e$.
Assume that both $d$ and $r$ are \'etale. 
For $d$ this means that there is a cover  $1 = \bigvee X$ such that for each $x\in X$ and $a \leq x$ we have $a=d^*(s)\wedge x$ for some $s\in e^{\downarrow}$. 
We have $a=d^*(s)\wedge x=  xs$ by part~\eqref{i:aux3} of Lemma~\ref{lem:aux1}.
Since $r$ is \'etale, we may similarly deduce that 
there is a cover $1 = \bigvee Y$ such that for each $y\in Y$ and $b\leq y$ we have
$b=ty$ for some $t\in e^{\downarrow}$.
We have
$$
1 = 1 \wedge 1 = \left (\bigvee X \right )  \wedge \left (\bigvee Y \right ) = \bigvee_{x\in X, y\in Y} (x\wedge y).
$$
The elements $x\wedge y$ are partial isometries by our calculations above.
Thus the maximum element 1 is a join of partial isometries.

Conversely, assume that  we may write $1 =\bigvee X$ where $X\subseteq {\mathcal{PI}}(Q)$. 
Let $x\in X$ and $a\leq x$. Since $x$ is a partial isometry we have that  $a=xs=d^*(s)\wedge x$ for some $s\leq e$. Hence the map $s\mapsto d^*(s)\wedge x$ from $e^{\downarrow}$ to $x^{\downarrow}$ is surjective. This proves that $d$ is \'etale. We may prove in a similar way that $r$ is \'etale. 
\end{proof}

We now turn to the multiplication map in a localic category.

\begin{proposition} \label{prop:aug30} Let $C$ be a localic category with all structure maps, including $m$, open.
Then the set  ${\mathcal{PI}}({\mathcal{O}}(C))$ is closed under multiplication and forms a complete restriction monoid.
\end{proposition}
\begin{proof} Let $\leq'$ denote the natural partial order on ${\mathcal O}(C)$ and let $\leq$ be the partial order that underlies the frame structure on ${\mathcal O}(C)$.
Let $a,b\in {\mathcal{PI}}({\mathcal O}(C))$ and $c\leq ab$.  Applying the Frobenius condition for $m$ we have
$$c=ab\wedge c =m_!(a\otimes b)\wedge c = m_!(a\otimes b \wedge m^*(c)).$$
This, in turn, equals
\begin{multline*}
m_!\left(  a \otimes b \wedge  \left( \bigvee_{xy \leq c} x\otimes y \right) \right)
= m_! \left( \bigvee_{xy \leq c}   (a\otimes b) \wedge (x\otimes y) \right) \\
= \bigvee_{xy \leq c} m_! (    (a \wedge x) \otimes (b \wedge y) ) = \bigvee_{xy \leq c} (a \wedge x)(b \wedge y).  
\end{multline*}
Now $a \wedge x \leq a$ and $a$ is a partial isometry so that $a \wedge x = af$ for some projection $f$.
Similarly,  $b \wedge y = bg$ for some projection $g$.
Thus $(a \wedge x)(b \wedge y) = afbg$.
But $fb \leq b$ and $b$ is a partial isometry and so there is a projection $f'$ such that $fb = bf'$.
It follows that we may write  $(a \wedge x)(b \wedge y) = abh_{x,y}$ for some projection $h_{x,y}$.
Define $p = \bigvee_{xy \leq c} h_{x,y}$, a projection.
Then by our above calculation, we have that $c = abp$ where $p$ is a projection.
By symmetry, we may write $c = qab$ for some projection $q$.
It follows that $c \leq' ab$ and so $c$ is a partial isometry, as required.
Hence ${\mathcal{PI}}({\mathcal O}(C))$ is closed under multiplication. 
By Proposition~\ref{prop:pi6} it is a complete restriction monoid.
\end{proof}

An Ehresmann quantal frame $Q$ will be called {\em strongly multiplicative} if $Q$ is multiplicative and the multiplication in ${\mathcal C}(Q)$ is open. 
The following statement follows from Proposition \ref{prop:aug30}.

\begin{corollary}\label{cor:dec4} 
If $Q$ is a strongly multiplicative Ehresmann quantal frame then ${\mathcal{PI}}(Q)$ is a complete restriction monoid. 
\end{corollary}

Let us call a localic category {\em open} if all its structure maps are open.
The claim of Theorem \ref{th:jul17}, the Correspondence Theorem,  can be restricted to the following statement.

\begin{theorem}[Open Correspondence Theorem] \label{th:n5} 
There is a  bijective  correspondence between strongly multiplicative Ehresmann quantal frames and open localic categories.
\end{theorem}

We now have the key result.

\begin{proposition} \label{prop:aug302} Let $Q$ be an \'etale Ehresmann quantal frame. 
Then the set of partial isometries of $Q$ is closed under multiplication if and only if $Q$ is strongly multiplicative. 
\end{proposition}
\begin{proof} If $Q$ is strongly multiplicative, then ${\mathcal{PI}}(Q)$ is a complete restriction monoid by Corollary \ref{cor:dec4}.

Assume that  ${\mathcal{PI}}(Q)$ is closed with respect to multiplication. 
We first show that $Q$ is multiplicative. 
Let $\mu \colon Q\otimes_{e^{\downarrow}} Q \to Q$ be the quantale multiplication map. 
We prove that the right adjoint $m^*$ of $\mu$ preserves arbitrary joins.
Let $A\subseteq {\mathcal{PI}}(Q)$. It is enough to show that 
\begin{equation*}
m^*\left (\bigvee A \right ) \leq \bigvee_{a\in A} m^*(a).
\end{equation*}
The opposite inequality holds automatically by the monotonicity of $m^*$ and the definition of the join.
We must therefore prove from the assumption
\begin{equation}\label{eq:caca}
p\otimes q\leq m^*\left (\bigvee A\right),
\end{equation}
where $p$ and $q$ are partial isometries, it follows
that 
\begin{equation}\label{eq:dada}
p\otimes q\leq \bigvee_{a\in A} m^*(a)=\bigvee \{x\otimes y\colon xy\leq a \text{ for some } a\in A\}.
\end{equation}
We now need a small side calculation.
The element $pq$, being a product of partial isometries, is itself a partial isometry. 
Since we are in a restriction monoid, we may write $q\lambda(pq)=\lambda(p)q$. Hence 
$$
p\otimes q=p\lambda(p)\otimes q=p\otimes\lambda(p)q=p\otimes q\lambda(pq).
$$
By adjointness, our assumption \eqref{eq:caca} is equivalent to $pq\leq \bigvee A$.
It follows that $pq=\bigvee_{a\in A}(pq\wedge a)$. 
By Lemma \ref{lem:joins22} $\lambda$ preserves joins and so
$\lambda(pq)=\bigvee_{a\in A}\lambda(pq\wedge a).$ 
We therefore obtain
\begin{multline*}
p\otimes q
=p\otimes q\lambda(pq)
=p\otimes q \left( \bigvee_{a\in A}\lambda(pq\wedge a) \right)=\\
p\otimes \left( \bigvee_{a\in A} q\lambda(pq\wedge a) \right)
= \bigvee_{a\in A} p\otimes q\lambda(pq\wedge a).
\end{multline*}
This and the observation that $pq\lambda(pq\wedge a)=pq\wedge a\leq a$ imply  \eqref{eq:dada}.
It follows that $Q$ is multiplicative and $m^*$ defines a semiopen multiplication map 
$m \colon C_1\times_{C_0} C_1\to C_1$, where $O(C_1)=Q$ and $O(C_0)=e^{\downarrow}$ and $\mu=m_!$ (see  Subsection~\ref{s:s4}).

We now verify that $m$ is open, that is, it satisfies the Frobenius condition. 
It is enough to verify only the inequality
\begin{equation}\label{eq:fafa}
m_!(a\otimes b)\wedge c \leq m_!((a\otimes b)\wedge m^*(c))
\end{equation}
for any $a,b,c\in Q$  since the opposite inequality holds by the monotonicity of $m_!$  and the fact that $m_!m^*\leq id$.
Since $Q$ is \'etale and applying distributivity and  the fact that $m_!$ preserves joins, it is enough to verify the inequality \eqref{eq:fafa} only for the case where $a,b\in {\mathcal{PI}}(Q)$.  So we assume that $a,b\in {\mathcal{PI}}(Q)$.

We shall need the following observation.
Let  $c\leq ab$. 
Since $ab\in {\mathcal{PI}}(Q)$ by assumption, there is some $f\leq e$ such that $c=abf$. 
Now, since $b\in {\mathcal{PI}}(Q)$ and $bf\leq b$ it follows that $bf=gb$ for some $g\leq e$. 
Thus $c=abf=agbf=(ag)(bf)$. 
It follows that $c$ can be written as $c=xy$ with $x\leq a$ and $y\leq b$. 

Now $m_!(a\otimes b)\wedge c=ab\wedge c$. 
Since $ab\in {\mathcal{PI}}(Q)$ and $ab\wedge c\leq ab$, 
the observation above shows that $ab\wedge c=pq$ with $p\leq a$ and $q\leq b$. 
Of course we also have $pq\leq c$.
Consequently, we obtain 
\begin{multline*}
m_!((a\otimes b)\wedge m^*(c)) 
= m_!\left ((a\otimes b) \wedge \bigvee_{xy \leq c}   x\otimes y  \right) 
= \bigvee_{xy \leq c}  m_! ((a\otimes b) \wedge  (x\otimes y))  \\
= \bigvee_{xy \leq c}  (a\wedge x)(b\wedge y) 
\geq (a\wedge p)(b\wedge q)=pq=ab\wedge c.
\end{multline*}
 This implies the inequality \eqref{eq:fafa} and completes the proof. 
\end{proof}

We say that a localic category is an {\em \'{e}tale localic category} if the maps $u,m$ are open and $d,r$ are \'etale.
In the light of the results proved in this subsection, 
the correspondence given in Theorem \ref{th:jul17} restricts to the following.

\begin{theorem}[Etale Correspondence Theorem]\label{the:jan23} 
There is a bijective correspondence between restriction quantal frames and \'etale localic categories.
\end{theorem} 

\section{The Duality Theorem} \label{s:env}

In this section we shall augment the Etale Correspondence Theorem by morphisms to obtain a genuine duality.

\subsection{Properties of morphisms between restriction quantal frames}\label{sub:morphisms}
Adopting the idea that locale maps are defined as frame maps going in the opposite direction, we can readily define morphisms of quantal (open, \'etale) localic categories as morphisms of their corresponding  Ehresmann  quantal frames going in the opposite direction. This is a trivial way to make the correspondences 
established in Theorems~\ref{th:jul17}, \ref{th:n5} and~\ref{the:jan23} functorial.
One would hope to obtain some kinds of functors between categories defined by morphisms of Ehresmann quantal frames, but since morphisms between restriction quantal frames (and thus also between wider classes of Ehresmann quantal frames) are not in general frame maps, they do not define functors between localic categories. The only type of morphisms defined which are frame maps are proper $\wedge$-morphisms between restriction quantal frames.  In this subsection we shall prove that proper $\wedge$-morphisms do give rise to functors, and these are special functors that preserve the \'etale structure of the \'etale category. This will lead to a genuine duality theorem. To establish this duality, we need to make a careful study of the properties of morphisms between restriction quantal frames. This analysis will be also used in  Section~\ref{s:adjun} where we show that all the four types of morphisms between restriction quantal frames give rise to natural classes of relational morphisms between \'etale  topological categories.

Let $C=(C_1,C_0)$ and  $D=(D_1,D_0)$ be \'{e}tale localic categories and 
$$g_1^*\colon {\mathcal O}(D)\to {\mathcal O}(C)$$
a morphism of restriction quantal frames.
We put 
$$g_0^*=d_!g_1^*u_!.$$
Observe that $g_0^*$ is a frame map. Indeed, $u_!$ is a frame map, then the restriction of  $g_1^*$ to the frame $e_{{\mathcal O}(D)}^{\downarrow}$ is a frame map whose image is in the frame $e_{{\mathcal O}(C)}^{\downarrow}$, and finally by part \ref{i:aux0a} of Lemma \ref{lem:aux1} the restriction of $d_!$ to $e_{{\mathcal O}(C)}^{\downarrow}$ is a frame map. Thus $g^*_0$ defines a locale map $g_0\colon C_0\to D_0$. We once again emphasize that $g_1^*$ does not in general define a locale map (as it is not assumed to preserve binary meets).

\begin{proposition} \label{prop:morphisms} Let $C=(C_1,C_0)$ and $D=(D_1,D_0)$ be \'etale  localic  categories and $g_1^*\colon {\mathcal O}(D)\to {\mathcal O}(C)$ be a sup-lattice map. We put $g_0^*=d_!g_1^*u_!$ and assume that it is a frame map. Then $g_1^*$ is a morphism of restriction quantal frames if and only if the following conditions hold:
\begin{enumerate}[{\em (M1)}]
\item $g_{1}^{*}$ maps partial isometries to partial isometries.
\item  $d_{!} g_{1}^{*} = g_{0}^{*}d_{!}$, and dually.
\item $g_{1}^{*}u_{!} = u_{!}g_{0}^{*}$. 
\item $(g_1^*\otimes g_1^*)m^* \leq m^* g_1^*$.
\end{enumerate}
\end{proposition}

We will make use of the following result which was proved as \cite[Theorem 5.14]{Re}.

\begin{lemma}\label{lem:march15} The inequality $(g_1^*\otimes g_1^*) m^* \leq m^* g_1^*$ holds if and only if
$$g_{1}^{*}(a)g_{1}^{*}(b) \leq g_{1}^{*}(ab)$$ for all $a,b \in O(D_1)$. 
\end{lemma}

\begin{proof}[Proof of Proposition \ref{prop:morphisms}]  
Assume that $g_1^*$ is a sup-map, $g_0^*$ is a frame map and that axioms (M1)--(M4) are satisfied. By (M1), we have that $g_1^*$ maps partial isometries to partial isometries. It follows from
(M3) that $g_1^*u_!$ is a frame map.  In particular $g_1^*(e_{{\mathcal O}(D)})=e_{{\mathcal O}(C)}$. In other words, the restriction of  $g_1^*$ to $e_{{\mathcal O}(D)}^{\downarrow}$ is a frame map.

We show that $g_1^*$ preserves the product of partial isometries.
Let first $a,b \in {\mathcal{PI}}({\mathcal O}(D))$ be such that $\lambda (a) = \rho(b)$. Then 
$$\lambda(ab)=\lambda(\lambda(a)b)=\lambda(\rho(b)b)=\lambda(b).$$
By (M4) and Lemma~\ref{lem:march15}, we have that $g_1^*(a) g_1^*(b) \leq g_1^*(ab)$.
From (M2) and (M3) it follows that $g_1^*$ commutes with $\lambda$ and with $\rho$. Using this, we obtain
$$\lambda(g_1^*(ab))=g_1^*(\lambda(ab))=g_1^*(\lambda(b))=\lambda(g_1^*(b)).$$ 
Also,
$\lambda(g_1^*(a))=\rho(g_1^*(b))$ which implies that 
$$\lambda(g_1^*(a)g_1^*(b))=\lambda(\lambda(g_1^*(a)g_1^*(b)))=\lambda(g_1^*(b)).
$$ 
This and the inequality $g_1^*(a)g_1^*(b)\leq g_1^*(ab)$ yield the equality $g_1^*(a)g_1^*(b)= g_1^*(ab)$ by part \eqref{i:apr4b} of Lemma \ref{lem:apr4}. 

Now let $a, b \in {\mathcal{PI}}({\mathcal O}(D))$ be arbitrary.
Put $f = \lambda(a)\rho(b)$. 
Then $ab = (af)(fb)$ and the latter is a restricted product.
Thus 
\begin{equation}\label{eq:gaga1} g_1^*(ab) = g_1^*(af)g_1^*(fb).
\end{equation}
By part \eqref{i:aux3a} of Lemma \ref{lem:aux1} we have $f=\lambda(a)\wedge \rho(b)$. We thus have
\begin{equation} \label{eq:gaga} g_1^*(f) = g_1^*(\lambda(a))\wedge g_1^*(\rho(b))= g_1^*(\lambda(a))g_1^*(\rho(b))=\lambda(g_1^*(a))\rho(g_1^*(b)).
\end{equation}
We prove that $g_1^*(af) = g_1^*(a)g_1^*(f)$.
We have that $g_1^*(af), g_1^*(a)g_1^*(f) \leq g_1^*(a)$.
But $g_1^*(a)$ is a partial isometry and $\lambda(g_1^*(af)) = g_1^*(f) = \lambda(g_1^*(a)g_1^*(f))$.
Thus $g_1^*(af) = g_1^*(a) g_1^*(f)$ by part \eqref{i:apr4b} of Lemma \ref{lem:apr4}.
This equality, together with \eqref{eq:gaga1} and \eqref{eq:gaga}, imply that
\begin{multline*}
g_1^*(ab)=g_1^*(af)g_1^*(fb)=g_1^*(a) g_1^*(f)g_1^*(f)g_1^*(b)=\\g_1^*(a)\lambda(g_1^*(a))\rho(g_1^*(b))g_1^*(b)=g_1^*(a)g_1^*(b).
\end{multline*}
We have therefore proved that $g_1^*$ restricted to the partial isometries is a  homomorphism of complete restriction monoids.  By Theorem \ref{th:quant_psgrps} it follows that $g_1^*$ is morphism of restriction quantal frames.

Conversely, assume that $g_1^*$ is a morphism of restriction quantal frames and verify that axioms (M1)--(M4) are satisfied.  That (M1) holds follows from the definition of a morphism of restriction quantal frames.

(M2) holds. 
We verify only the first  equality, since the second one follows by symmetry.  
Since $g_1^*$ is a morphism of restriction monoids, it commutes with $\lambda$, that is we have $g_1^*u_!d_!=u_!d_!g_1^*$. 
Thus we have
$d_!g_1^*=d_!u_!d_!g_1^*=d_!g_1^*u_!d_!=g_0^*d_!,$
as required. 

(M3) holds. For any $x\in O(D_0)$ we have that $g_1^*u_!(x)\leq g_1^*(e_{{\mathcal O}(D)}) = e_{{\mathcal O}(C)}$. It follows that $u_!g_0^*=u_!d_!g_1^*u_!=g_1^*u_!$ as $\lambda(x)=x$ whenever $x\leq e_{{\mathcal O}(C)}$.

(M4) holds by  Lemma~\ref{lem:march15}.
\end{proof}

In what follows we sometimes apply the characterisation obtained in Proposition \ref{prop:morphisms} without further mention.

\begin{lemma}\label{lem:prop_mor} Let $C=(C_1,C_0)$ and $D=(D_1,D_0)$ be \'etale localic  categories and $g_1^*\colon {\mathcal O}(D)\to {\mathcal O}(C)$ be a morphism between restriction quantal frames. 
\begin{enumerate}
\item \label{i:proper6} If  $g_1^*$ is proper then $d^*g^*_0 =g^*_1d^*$ and $r^*g^*_0 =g^*_1r^*$. 
\item \label{i:wedge6} If $g_1^*$ is a $\wedge$-morphism then $u^{*}g_{1}^{*} = g_{0}^{*}u^{*}$.
\end{enumerate}
\end{lemma}
\begin{proof} \eqref{i:proper6}
We prove the first equality and the second one follows by symmetry.
Let $x \in O(D_{0})$.
By part \eqref{i:aux3} of Lemma~\ref{lem:aux1} we have
$d^{*}(x) = 1_{{\mathcal O}(D)}u_{!}(x)$.
Thus
$$g_{1}^{*}d^{*}(x) 
= g_{1}^{*}(1_{{\mathcal O}(D)} u_{!}(x))
=
1_{{\mathcal O}(C)} g_{1}^{*}(u_{!}(x))
=
1_{{\mathcal O}(C)}u_{!} (g_{0}^{*}(x))
=
d^{*}g_{0}^{*}(x),$$ 
where we have used (M3), the fact that $g_1^*$ preserves multiplication and the assumption that $g_{1}^{*}(1_{{\mathcal O}(D)}) = 1_{{\mathcal O}(C)}$.

\eqref{i:wedge6}
Let $b \in O(C_{1})$. By part \eqref{i:aux_new} of Lemma \ref{lem:aux1} we have that
$u_!u^*(b)=b\wedge e_{{\mathcal O}(C)}$. Applying $d_!$ to both parts and using $du=id$ we obtain
$u^*(b)=d_!(b\wedge e_{{\mathcal O}(C)})$.
For any $a\in O(C_1)$ we now calculate, using the fact that $g_{1}^{*}$ preserves meets together with (M2) and (M3),
$$u^{*}g_{1}^{*}(a) = d_{!}(g_{1}^{*}(a) \wedge e_{{\mathcal O}(C)}) = d_{!}g_{1}^{*}(a \wedge e_{{\mathcal O}(D)})
= g_{0}^{*}d_{!}(a \wedge e_{{\mathcal O}(D)}) = g_{0}^{*}u^{*}(a).$$
\end{proof}

We come to a characterization of proper $\wedge$-morphisms between restriction quantal frames as functors between their corresponding \'etale localic categories going in the opposite direction that satisfy two additional conditions.

\begin{proposition}\label{prop:fun} Let $C=(C_1,C_0)$ and $D=(D_1,D_0)$ be \'etale  localic categories and $g_1^*\colon {\mathcal O}(D)\to {\mathcal O}(C)$ be a map. We put $g_0^*=d_!g_1^*u_!$. Then $g_1^*$ is a proper $\wedge$-morphism of restriction quantal frames if and only if the pair  $(g^*_{1},g^*_{0})$ defines a functor from $C$ to $D$
and conditions (M1) and (M2) hold.
\end{proposition}

\begin{proof} Assume that $g_1^*\colon {\mathcal O}(D)\to {\mathcal O}(C)$ is a proper $\wedge$-morphism of restriction quantal frames. We verify that the equations (Fun1)--(Fun4), defining a functor, hold. We first note that $g_1^*$ being a frame map, defines a locale map $g_1\colon C_1\to D_1$.

To prove (Fun1), we need only, by Lemma~\ref{lem:march15}, prove that $m^*  g_1^*\leq (g_1^*\otimes g_1^*)  m^*$. 
This is equivalent to proving that
\begin{equation}\label{eq:apr6aa}
\bigvee_{pq\leq g_1^*(x)}p\otimes q \leq \bigvee_{ab\leq x}g_1^*(a)\otimes g_1^*(b)
\end{equation}
for any $x \in {\mathcal O}(D)$. Since ${\mathcal O}(D)$ is \'etale and using the fact that the maps $g_1^*$, $m^*$ and $g_{1}^{*} \otimes g_{1}^{*}$ are sup-maps, it easily follows that it is enough to prove the inequality  \eqref{eq:apr6aa} for all $x\in {\mathcal{PI}}({\mathcal O}(D))$.
Next observe that we may assume that $p,q\in {\mathcal{PI}}({\mathcal O}(C))$.
This follows from the fact that ${\mathcal O}(C)$ is \'etale and properties of the tensor products. 
Hence, it is enough to prove the following:
if $x\in {\mathcal{PI}}({\mathcal O}(D))$ and $p,q\in {\mathcal{PI}}({\mathcal O}(C))$ then 
\begin{equation}\label{eq:apr6ab}
pq\leq g_1^*(x)
\,\, \Rightarrow \,\,
p\otimes q \leq \bigvee_{ab\leq x}g_1^*(a)\otimes g_1^*(b).
\end{equation}
 
By Theorem \ref{th:quant_psgrps}, the restriction of $g_1^*$ to ${\mathcal{PI}}({\mathcal O}(D))$ is a proper morphism of complete restriction monoids. 
Hence there are partial isometries $p_i$ and $s_{i}$, $i\in I$, such that 
$$p=\bigvee_{i}p_i  \mbox{ and } p_{i} \leq g_1^*(s_i).$$ 
Similarly, there are partial isometries $q_j$ and $t_{j}$, $j\in J$, such that
$$q= \bigvee_{j}q_j \mbox{  and } q_{j} \leq g_1^*(t_{j})$$
We have the equality $p\otimes q= \bigvee_{i,j} p_i\otimes q_j$.
Thus we shall be done if we can prove the following:
if
$p_{i}q_{j} \leq g_{1}^{*}(x)$,
where $p_{i} \leq g_{1}^{*}(s_{i})$
and $q_{j} \leq g_{1}^{*} (t_{j})$
then 
$$p_i\otimes q_j  \leq g_1^*(s')\otimes g_1^*(t')$$ 
where $s't' \leq x$. 

Put $p'=p_i\rho(q_j)$ and $q'=\lambda(p_i)q_j$. 
Then 
$p_iq_j=p'q'$ 
and
$p_i\otimes q_j = p'\otimes q'$.
Clearly, $p_{i}q_{j} \leq g_{1}^{*}(s_{i}t_{j})$.
By assumption, the map $g_{1}^{*}$ preserves binary meets
and so
$p_{i}q_{j} \leq g_{1}^{*}(s_{i}t_{j} \wedge x)$.
Now $s_it_j\wedge x \leq s_it_j$, a partial isometry.
It follows that 
$s_it_j\wedge x=s't'$ where $s'\leq s_i$, $t'\leq t_j$ and we may assume that $\lambda (s') = \rho (t')$.
We have that
$p'q' \leq g_{1}^{*}(s't')$.
Since $g_1^*$ commutes with $\lambda$ and $\lambda(s't')=\lambda(t')$, we have $\lambda(g_1^*(s't'))=\lambda(g_1^*(t'))$.
Hence $\lambda(q') \leq \lambda(g_1^*(t'))$.
We therefore have that
$$q' = q_j\lambda(q') \leq g_1^*(t_j)\lambda(g_1^*(t')) = g_1^*(t_j\lambda(t')) = g_{1}^{*}(t').$$
By symmetry, we have that $p' \leq g_{1}^{*}(s')$.
We obtain $p_i\otimes q_j = p'\otimes q' \leq g_1^*(s')\otimes g_1^*(t')$ 
where $s't'=s_it_j\wedge x\leq x$, which completes the proof of \eqref{eq:apr6ab}.

(Fun2), (Fun3) and (Fun4)  hold by Lemma~\ref{lem:prop_mor}.  

We now assume that the pair  $(g^*_{1},g^*_{0})$ defines a functor from $C$ to $D$
and that conditions (M1) and (M2) hold. We need to prove that $g_1^*$ is a morphism of restriction quantal frames.
We apply Proposition \ref{prop:morphisms}. Since $g_1^*$ is a sup-map and $g_0^*$ is a frame map we are left to show that conditions (M1)--(M4)  hold. (M1) and (M2) hold by assumption.
(M3) holds by (Fun1).  We show that (M4) holds. We first verify that $u_{!}g_{0}^{\ast}\leq g_{1}^{\ast}u_!$. By adjointness, this is equivalennt to
$g_0^*\leq u^*g_1^*u_!$. Applying $u^*g_1^*=g_0^*u^*$ this inequality reduces to $g_0^*\leq g_0^*u^*u_!$ which holds since $u^*u_!\geq id$. It follows that 
\begin{equation*}\label{eq:apr6ac} e_{{\mathcal O}(C)}=u_!g_0^*(1_{O(D_0)})\leq g_{1}^{\ast}u_!(1_{O(D_0)})=g_1^*(e_{{\mathcal O}(D)}).
\end{equation*} 
Let $x=g_1^*(e_{{\mathcal O}(D)})$. By assumption, this is a partial isometry. From $x\geq e_{{\mathcal O}(C)}$ we obtain $e_{{\mathcal O}(C)}=xe_{{\mathcal O}(C)}=x$. Thus we have proved that $g_1^*(e_{{\mathcal O}(D)})=e_{{\mathcal O}(C)}$. Using the definition of $g_0^*$ the required equality reduces to $u_!d_!g_1^*u_!=g_1^*u_!$. This equality holds since for any $x\in O(D_0)$ we have that 
$$g_1^*u_!(x)\leq g_1^*u_!(1_{O(D_0)})=g_1^*(e_{{\mathcal O}(D)})=e_{{\mathcal O}(C)}
$$
and if $x\leq e_{{\mathcal O}(C)}$ then $\lambda(x)=x$.
\end{proof}

\subsection{Localic sheaf functors and the duality}
Let $Q$ be an Ehresmann quantal frame and let $\leq$ be the partial order underlying the sup-lattice structure on $Q$.
An element $x\in Q$ is called a {\em right} (resp. {\em left}) {\em partial isometry} if the inequality $a\leq x$ implies that $a=x\lambda(a)$ (resp. $a=\rho(a)x$). It follows that $x$ is a partial isometry if and only if it is both a right and a left partial isometry.

We show first how to construct two sheaves of sets from $Q$. Let $f\leq e$.
Define 
$$\Sigma(\lambda)_{f} = \{ a \in \mathcal{RPI}(Q) \colon \lambda (a) = f \}.$$ 
If $g \leq f$ define a map $\Sigma(\lambda)_{f} \rightarrow \Sigma(\lambda)_{g}$ by $a \mapsto ag = a|^{f}_{g}$.
It is easy to check that in this way we obtain a presheaf of sets $\Sigma(\lambda)$
over the frame $e^{\downarrow}$.
In a similar way, we may also construct a presheaf of sets $\Sigma(\rho)$ using the structure map $\rho$.
The argument in the proof of Proposition~\ref{prop:pi6} shows that both of these presheaves are in fact sheaves.

\begin{remark} {\em Observe that $\lambda \colon \mathcal{PI}(Q) \rightarrow e^{\downarrow}$ defines a presheaf
of sets but not a sheaf.
This is because a family of partial isometries that is a matching family with respect to $\lambda$ is not in general a compatible family.}
\end{remark}

We now translate these definitions into the language of  \'etale localic categories.
We first make the following definitions. Let  $C=(C_1,C_0)$ be an \'etale localic category. Then $x\in O(C_1)$ will be called a {\em local bisection} (resp. a {\em right local section}, a {\em left local section}) if $x$, considered as an element of ${\mathcal O}(C)$, is a partial isometry (resp. right partial isometry, left partial isometry).  
 
Let $C=(C_1,C_0)$ and $D=(D_1,D_0)$ be  \'etale localic categories. 
For each $a\in O(C_0)$, define
$S_C(d)_{a}$ to be the set of all right local sections  $x\in O(C_1)$ satisfying the condition $d_!(x)=a$.
Then the assignment $a\mapsto S_{C}(d)_{a}$ defines a sheaf of sets $S_C(d)$ over the frame $O(C_0)$. 
In a similar way, we may use the map $r_{!}$ to define a sheaf of sets $S_{C}(r)$.

\begin{lemma}\label{lem:sheaf_a6} Let $g_1^*\colon {\mathcal O}(D)\to {\mathcal O}(C)$ be a morphism of restriction quantal frames. Then $g_1^*$ induces a {\em sheaf morphism} from $S_D(d)$ to $S_C(d)$ over $g_0^*$ in the sense that the following requirements are satisfied:
\begin{enumerate}[{\em (ShM1)}]
\item $g_1^*$ maps right local sections to right local sections.
\item $d_!g_1^*=g_0^*d_!$.
\item If $a\in O(D_1)$ is a right local section and $f\in O(D_0)$ are such that $f\leq d_!(a)$ then 
$$g_1^*(a|^{d_!(a)}_f)=g_1^*(a)|^{g_0^*(d_!(a))}_{g_0^*(f)}.$$
\end{enumerate}
Similarly, $g_1^*$ induces a sheaf morphism from $S_D(r)$ to $S_C(r)$.
\end{lemma}

\begin{proof} We first show that (ShM1) holds.  Let $x\in {\mathcal O}(D)$ be a right partial isometry. Then $x=\bigvee A$ where $A$ is a compatible family of partial isometries. The elements $g_1^*(a)$, $a\in A$, are  partial isometries by (M1). If $a\sim b$ then $g_1^*(a)\sim g_1^*(b)$ as $\sim$ is expressed in terms of $\lambda$, $\rho$ and the multiplication, and these all are preserved by $g_1^*$. It follows that the elements $g_1^*(a)$ form a matching family of right local sections with respect to $d_!$, and thus  $g_1^*(x)=\bigvee_{a\in A}g_1^*(a)$ is a right local section as a join of a matching family of right local sections. A similar result can be proved for left local sections.

Condition (ShM2) is a part of  (M2) which is satisfied by Proposition \ref{prop:morphisms}.  
We show that condition (ShM3) is satisfied.
Let $a\in O(D_1)$ be a right local section and $f\in O(D_0)$ be such that $f\leq d_!(a)$. 
Since $g_1^*(a)$ is a right local section, we have that $g_1^*(a|^{d_!(a)}_f)=g_1^*(a)u_!d_!g_1^*(a|^{d_!(a)}_f)$. 
Applying (ShM2), the latter equals 
$$g_1^*(a)u_!g_0^*d_!(a|^{d_!(a)}_f)=g_1^*(a)u_!g_0^*(f)=g_1^*(a)|^{g_0^*(d_!(a))}_{g_0^*(f)}.$$
\end{proof}

We arrive at the following characterization of proper $\wedge$-morphisms between restriction quantal frames. 

\begin{proposition}\label{prop:fun} Let $C=(C_1,C_0)$ and $D=(D_1,D_0)$ be \'etale localic  categories and $g_1^*\colon {\mathcal O}(D)\to {\mathcal O}(C)$ be a map. We put $g_0^*=d_!g_1^*u_!$. Then $g_1^*$ is a proper $\wedge$-morphism of restriction quantal frames if and only if the pair  $(g^*_{1},g^*_{0})$ defines a functor from $C$ to $D$
and $g_1^*$ induces sheaf morphisms 
from $S_{C}(d)$ to $S_{D}(d)$, and from $S_{C}(r)$ to $S_{D}(r)$ over $g_0^*$.
\end{proposition}

\begin{proof} Let $g_1^*$ be a proper $\wedge$-morphism. Then by Proposition \ref{prop:fun}  $(g^*_{1},g^*_{0})$ defines a functor from $C$ to $D$ and by Lemma \ref{lem:sheaf_a6} it induces the required two sheaf morphisms. Conversely, assume that $(g^*_{1},g^*_{0})$ defines a functor from $C$ to $D$ and $g^*_{1}$ induces the two sheaf morphisms. It follows that $g_1^*$ maps partial isometries to partial isometries so (M1) holds and also (M2) holds by (ShM2) and its dual condition. By Lemma~\ref{lem:sheaf_a6} $g_1^*$ is a proper $\wedge$-morphism.
\end{proof}

We refer to functors $g\colon C\to D$ such that $g_1^*$ induces the two sheaf morphisms described in Proposition~\ref{prop:fun} as {\em localic sheaf functors}.
By Proposition~\ref{prop:fun}, Theorem~\ref{th:jul17} and Theorem~\ref{the:jan23} we have the following main result.

\begin{theorem}[Duality Theorem] \label{cor:main} The following categories are equivalent.
\begin{enumerate}
\item The category of complete restriction monoids and proper $\wedge$-morphisms.
\item The category of restriction quantal frames  and proper $\wedge$-morphisms.
\item The opposite of the category of \'{e}tale localic categories and localic sheaf functors. 
\end{enumerate} 
\end{theorem}

\section{The involutive setting}\label{s:involutive}
In this section we add involutions to all objects we have studied so far and obtain involutive versions of our results.
This looks interesting on its own right and also makes a preparation for Section \ref{s:inv:setting} where we add a further
requirement that involutions are inversions.
 
Let $S$ be a restriction semigroup and $(-)^*:S\to S$ be a map. 
We call it an {\em involution} on $S$ provided that $(a^*)^*=a$, $(ab)^*=b^*a^*$ and $\lambda(a^*)=\rho(a)$ for all $a,b\in S$.
These axioms imply 
the equality $\rho(a^*)=\lambda(a)$ for all $a\in S$.
We call a complete restriction monoid $S$ {\em involutive} provided that it posesses an involution.

\begin{lemma} \label{lem:mon} Let $S$ be a restriction monoid and $(-)^*$ an involution on $S$. 
\begin{enumerate}
\item $g\leq e$ implies $g^*=g$.
\item $a\leq b$ implies $a^*\leq b^*$.
\end{enumerate}
\end{lemma}

\begin{proof} (1) Since $e$ is a multiplicative identity, $e^*=e$.  
Assume that $g\in S$ is such that $gg^*\leq g^*$. 
Note that $\lambda(gg^*)=\lambda(\lambda(g)g^*)=\lambda(\rho(g^*)g^*)=\lambda(g^*)$. 
Thus we have $gg^*=g^*$. 
Therefore, $g=(g^*)^*=(gg^*)^*=gg^*=g^*$.  
If $g\leq e$ then $ga\leq a$ for all $a\in S$, and so in particular $gg^*\leq g^*$. 
Hence $g\leq e$ implies $g=g^*$.

(2) Let $a\leq b$. 
Then $a=b\lambda(a)$. 
Hence $a^*=\lambda(a)^*b^*=\lambda(a)b^*$, 
where we have $\lambda(a)^*=\lambda(a)$ by the above, since $\lambda(a)\leq e$. 
Therefore, $a^*\leq b^*$.
\end{proof}

\begin{lemma} \label{lem:inv} Assume $S$ is an involutive complete restriction monoid and $(-)^*$ is an involution on $S$.
Let $A$ be a compatible family of elements of $S$. 
Then
$$
\left( \bigvee A \right )^* = \bigvee_{a\in A} a^*.
$$
\end{lemma}

\begin{proof}  
The inequality $(\bigvee A )^* \geq \bigvee_{a\in A} a^*$ follows in a standard way from the definition of a join and
 the fact that $(-)^*$ is monotone, established above in Lemma~\ref{lem:mon}.
Applying Lemma \ref{lem:joins22}, we also have that
\begin{equation*}
\lambda \left(\left( \bigvee A \right )^*\right) = \rho \left( \bigvee A \right ) =
  \bigvee_{a\in A}\rho(a)  = \bigvee_{a\in A} \lambda(a^*) =
 \lambda\left (\bigvee_{a\in A} a^*\right).
\end{equation*}
The required equality now follows from part \eqref{i:apr4b} of Lemma \ref{lem:apr4}.
\end{proof}

Let $Q$ be a quantale. 
An {\em involution} on $Q$ is a sup-lattice endomorphism $(-)^*:Q\to Q$ such that $(a^*)^*=a$ and $(ab)^*=b^*a^*$ for any $a,b\in Q$. 
A quantale which has an involution is called an {\em involutive quantale}. 
A quantal frame which has an involution is called an {\em involutive quantal frame}.
Let $Q$ be an Ehresmann quantale equipped with an involution $(-)^*$. 
We call $Q$ an {\em involutive Ehresmann quantale} if $\lambda(a^*)=\rho(a)$ for all $a\in Q$. 
The axioms imply that $\rho(a^*)=\lambda(a)$ for all $a\in Q$.

\begin{lemma}\label{lem:frame_star} Let $(-)^*$ be an involution on an  Ehresmann quantal frame $Q$. 
Then $(a\wedge b)^*=a^*\wedge b^*$ for any $a,b\in Q$, so that $(-)^*$ is a frame map. 
Consequently, $(-)^*$ is a frame automorphism.
\end{lemma}

\begin{proof} Let $a,b\in Q$. 
Let $\leq$ be the order that underlies the frame $Q$.
Since $a\wedge b\leq a,b$ and the involution is a sup-lattice map, we have
$(a\wedge b)^*\leq a^*\wedge b^*$. 
We also note that $a^*\wedge b^*\leq a^*,b^*$ yields
$(a^*\wedge b^*)^*\leq a\wedge b$. 
Hence, we obtain
$$
a\wedge b=((a\wedge b)^*)^*\leq (a^*\wedge b^*)^*\leq a\wedge b.
$$
The statement now follows since $(a^*)^*=a$ for any $a\in Q$.
\end{proof}

Now let $Q$ be an {\em involutive} quantale with unit $e$.  
An element $a\in Q$ is called a {\em partial unit} \cite{Re} if $aa^*, a^*a\leq e$.

\begin{lemma}\label{lem:p_un} Let $Q$ be an involutive Ehresmann quantale and $a\in Q$ a partial unit. 
\begin{enumerate}
\item \label{i:pu1} $\lambda(a)=a^*a$ and $\rho(a)=aa^*$.
\item \label{i:pu2} $a$ is a partial isometry.
\end{enumerate}
\end{lemma} 

\begin{proof} (1) We have 
$a^*a=\lambda(a^*a)=\lambda(\lambda(a^*)a)=\lambda(\rho(a)a)=\lambda(a)$.
By symmetry $\rho(a)=aa^*$.

(2)
Assume $a$ is a partial unit and $b\leq a$. 
Since $(-)^*$ is a sup-lattice map, $b^*\leq a^*$. 
Since $\leq$ is compatible with the multiplication in $Q$, we obtain $ab^*\leq aa^*$. 
The latter inequality implies that $ab^*\leq' e$ since $\leq'$ and $\leq$ coincide on $e^{\downarrow}$.
Since $\rho$ is a sup-lattice map, $b^*\leq a^*$ implies $\rho(b^*)\leq \rho(a^*)$ which is equivalent to $\rho(b^*)\leq' \rho(a^*)$.
Then $b^*=\rho(b^*)b^*=\rho(a^*)b^*$. Therefore, we have
$$
ab^*=\lambda(ab^*)=\lambda(\rho(a^*)b^*)=\lambda(b^*)=\rho(b).
$$
This and (1) above yield $b=\rho(b)b=ab^*b=a\lambda(b)$. 
Similarly, we may show that $b=\rho(b)a$. 
It follows that $b\leq' a$.
\end{proof}

We also note the following useful fact.

\begin{lemma} \label{lem:aux24o} Let $a$ be a partial isometry and $a_i$, $i\in I$, be partial units such that $a=\bigvee_{i}a_i$.
Then $a$ is a partial unit.
\end{lemma}

\begin{proof} We have $a_i=a\lambda(a_i)$ for all $i\in I$. 
Then $a_i\lambda(a_j)=a\lambda(a_i)\lambda(a_j)=a_j\lambda(a_j)$. 
This shows that elements $a_i$, $i\in I$, form a compatible family.
Then 
$$
a_ia_j^*=a_i\rho(a_j^*)a_j^*=a_i\lambda(a_j)a_j^*=a_j\lambda(a_i)a_j^*\leq a_ja_j^*\leq e.
$$
Similarly, $a_i^*a_j\leq e$.
We then have
$$aa^*=\left (\bigvee_{i}a_i\right )\left (\bigvee_{i}a^*_i\right )=\bigvee_{i,j}a_ia_j^*\leq e$$
and similarly $a^*a\leq e$.
\end{proof}

Let $Q$ be an involutive Ehresmann quantale.  
Denote by ${\mathcal I}(Q)$ the set of partial units of $Q$.

\begin{proposition} \label{prop:inv22} 
Let $Q$ be an involutive Ehresmann quantale.  
Then ${\mathcal I}(Q)$ is an inverse monoid and 
its set of idempotents is $e^{\downarrow}$.
\end{proposition}
\begin{proof}
It is straightforward to verify that ${\mathcal I}(Q)$ is a monoid with the unit $e$. 
By Lemma~\ref{lem:p_un} 
all elements of ${\mathcal I}(Q)$ are partial isometries. Since partial isometries are bi-deterministic, ${\mathcal I}(Q)$ is a restriction monoid.
Let $a\in {\mathcal I}(Q)$. 
From Lemma~\ref{lem:p_un} it follows that $aa^*a=a\lambda(a)=a$. 
Likewise,
$a^*aa^*=a^*$. 
Thus ${\mathcal I}(Q)$ is a regular monoid. 
We show that the set of idempotents $E({\mathcal I}(Q))$ of ${\mathcal I}(Q)$ equals $e^{\downarrow}$. Since $e$ is a multiplicative identity, $e^*=e$.
Observe that $e^{\downarrow}\subseteq E({\mathcal I}(Q))$ because if $f\leq e$ then $f^*\leq e$ and so $ff^*,f^*f\leq e$. 
We now verify the reverse inclusion. 
Let $b\in E({\mathcal I}(Q))$. 
Applying part (1) of Lemma \ref{lem:p_un},  
we have
$$
b^*b=\lambda(b)=\lambda(b^2)=\lambda(\lambda(b)b).
$$
But $\lambda(b)b\leq b$. 
Since $b$ is a partial isometry, we conclude that $\lambda(b)b\leq' b$ where $\leq'$ is the natural partial order on the Ehresmann monoid underlying $Q$. 
This inequality and the equality $\lambda(b)=\lambda(\lambda(b)b)$ imply that $b=\lambda(b)b$ by part \eqref{i:apr4b} of Lemma \ref{lem:apr4}. It follows that $\lambda(b)\geq \rho(b)$. 
Likewise, we may show that $\rho(b)\geq\lambda(b)$. 
It follows that $\lambda(b)=\rho(b)$, or, by part (1) of Lemma \ref{lem:p_un}, $b^*b=bb^*$.
We obtain $b=bb^*b=b^*bb=b^*b\leq e$, and the inclusion $E({\mathcal I}(Q))\subseteq e^{\downarrow}$ follows.
Therefore, ${\mathcal I}(Q)$ is a regular monoid whose idempotents form a commutative submonoid.
It is well known \cite{Law2} that such a monoid is inverse.
\end{proof}

Notice that the proposition above is independent of  \cite[Theorem 3.15]{Re}, since we do not assume the inequality $\rho(a)\leq aa^*$ as is in \cite{Re}.

Let $S$ be an involutive complete restriction monoid, and $(-)^*$ be an involution.
Then  $(-)^*$ can easily be extended to an involution on ${\mathcal L}^{\vee}(S)$
making it an involutive restrictive quantale frame.
Clearly, if $(-)^*$ is an involution on ${\mathcal L}^{\vee}(S)$, its restriction to partial isometries is an involution on $\eta(S)$. 
We now require morphisms to be {\em involutive morphisms}.
We arrive at the involutive analogue of Theorem \ref{th:quant_psgrps}, the~Quantalization Theorem.

\begin{theorem}[Involutive Quantalization Theorem]\label{th:inv_quant_psgrps} 
 For every $k=1,2,3,4$ the category of involutive complete restriction monoids and involutive morphisms 
of type $k$ is equivalent to 
the category of involutive restriction quantal frames and involutive morphisms of type $k$.
\end{theorem}

We now incorporate involutions into Theorem \ref{th:jul17}, the Correspondence Theorem.
Let $C=(C_1,C_0)$ be a localic category. 
We say that it is {\em involutive} if there is an additional structure map $i \colon C_1\to C_1$ satisfying the following axioms:
\begin{enumerate}
\item[(ICat1)] $i^2=id$.
\item[(ICat2)] $di=r$.
\item[(ICat3)] $im=m \langle i\pi_2, i\pi_1\rangle$.
\end{enumerate}
The equalities $di=r$ and $i^2=id$ imply the equality $ri=d$. 
Note that $i^2=id$ implies that $i$ is semiopen with $i^*=i_!$. 
In addition, $i^*$ is a frame isomorphism, which implies that $i$ is open.

\begin{proposition}\label{prop:o24}\mbox{}
\begin{enumerate}
\item \label{i:oct211} Let $Q$ be a multiplicative Ehresmann quantal frame, and $(-)^*$ be an involution on $Q$. 
Then the associated quantal localic category ${\mathcal C}(Q)$ is involutive where the involution $i$  is defined by the frame map $(-)^*$.
\item \label{i:oct212} Let $C=(C_1,C_0)$ be an involutive quantal localic category, and let $i$ denote the involution structure map. 
Then the map $i_!=i^*:O(C_1)\to O(C_1)$ is an involution on the associated Ehresmann quantal frame ${\mathcal O}(C)$.
\end{enumerate}
\end{proposition}
\begin{proof}
\eqref{i:oct211} Let ${\mathcal C}(Q)= (C_1, C_0)$. 
Denote the involution map on $Q$ by $\alpha$. 
That is, we have put $\alpha(a)=a^*$ for all $a\in Q$. 
By Lemma \ref{lem:frame_star} the map $\alpha$ is a frame automorphism, so that its right adjoint $\alpha^*$ equals $\alpha$. 
Let $i \colon C_1\to C_1$ be the locale map defined by $\alpha$. 
Then we have $i_!=\alpha$. 
We verify that $i$ satisfies axioms for a category involution. 

(ICat1) holds.  From $\alpha^2=id$, we get $(i^*)^2=id$ which means that $i^2=id$. 

(ICat2) holds. We verify that $(di)^*=r^*$. 
This is equivalent to $(di)_!=r_!$, or $d_!i_!=r_!$. 
Bearing in mind that $du=id$ and $ru=id$, the latter is equivalent to $u_!d_!i_!=u_!r_!$. 
But this equality is simply $\lambda(a^*)=\rho(a)$ for all $a\in Q$, which is one of the axioms of an involutive Ehresmann quantale. 

(ICat3) holds.
 Let $\langle i\pi_2, i\pi_1\rangle^*=[(i\pi_2)^*, (i\pi_1)^*]: Q\otimes_{e^{\downarrow}} Q \to Q\otimes_{e^{\downarrow}} Q$ 
be the frame map that defines the locale map $\langle i\pi_2, i\pi_1\rangle$. 
Its effect, by definition, is given by
\begin{multline*}
[(i\pi_2)^*, (i\pi_1)^*](a\otimes b)=(i\pi_2)^*(a)\wedge (i\pi_1)^*(b) = \pi_2^*(i^*(a))\wedge \pi_1^*(i^*(b)) =\\ (1\otimes a^*)\wedge (b^*\otimes 1) = b^*\otimes a^*.
\end{multline*}
The latter equality implies that $[(i\pi_2)^*, (i\pi_1)^*]^2=id$, and so this map is a frame automorphism. Thus $\langle i\pi_2, i\pi_1\rangle_!=[(i\pi_2)^*, (i\pi_1)^*]$.
To show $(im)^*=(m \langle i\pi_2, i\pi_1\rangle)^*$,
it is enough to verify that
$(im)_!=(m \langle i\pi_2, i\pi_1\rangle)_!$ or that $i_!m_!=m_! [(i\pi_2)^*, (i\pi_1)^*]$. 
We have that
$i_!m_!(a\otimes b)=i_!(ab)=(ab)^*$. 
Also, 
$$
m_! [(i\pi_2)^*, (i\pi_1)^*](a\otimes b)=m_!(b^*\otimes a^*)=b^*a^*.
$$
But $(ab)^*=b^*a^*$ holds for any $a,b\in Q$ since it is one of the axioms of the involution on $Q$.

\eqref{i:oct212} This is established in a similar fashion to \eqref{i:oct211} above, reversing the order of the arguments appropriately. 
For example, we verify that $i^*(ab)=i^*(b)i^*(a)$ for all $a,b\in {\mathcal O}(C)$.
We use the axiom $im=m \langle i\pi_2, i\pi_1\rangle$. 
We first note, in a way similar to that above, that $\langle i\pi_2, i\pi_1\rangle_!=\langle i\pi_2, i\pi_1\rangle^*$ and that $\langle i\pi_2, i\pi_1\rangle^*(a\otimes b)=i^*(b)\otimes i^*(a)$. 
The category axiom now implies the equality $(im)_!=(m \langle i\pi_2, i\pi_1\rangle)_!$ which quickly reduces to the required equality $i^*(ab)=i^*(b)i^*(a)$.
 \end{proof}
 
Let $C$ and $C'$ be two involutive localic categories. 
We require that a  functor of involutive categories $f$ from $C$ to $C'$ preserve the involution. 
That is, the equality $f_1i=i'f_1$ holds, where $i$ and $i'$ are the involution maps of $C$ and $C'$, respectively.
From Proposition \ref{prop:o24} and Theorem \ref{th:jul17}, one can now easily derive the following involutive analogue of Theorem \ref{th:jul17}.
 
\begin{theorem}[Involutive Correspondence Theorem] \mbox{}
\begin{enumerate}
\item Let $Q$ be an involutive multiplicative Ehresmann quantal frame. 
Then ${\mathcal C}(Q)$ is an involutive quantal localic category and  $Q={\mathcal O}({\mathcal C}(Q))$.
 \item Let $C$ be an involutive quantal localic category. 
Then ${\mathcal O}(C)$ is an involutive multiplicative Ehresmann quantal frame and $C\cong {\mathcal C}({\mathcal O}(C))$.
\end{enumerate}
\end{theorem}

The involutive version of  Theorem  \ref{cor:main} is the following result.

\begin{theorem}[Involutive Duality Theorem] \label{cor:main_inv} \mbox{}
The following categories are equivalent.
\begin{enumerate}
\item The category of involutive complete restriction monoids and involutive proper $\wedge$-morphisms.
\item  The category of involutive restriction quantal frames and involutive proper  $\wedge$-morphisms.
\item The opposite of the category of involutive \'{e}tale localic categories and involutive localic sheaf functors.
\end{enumerate} 
\end{theorem}

\section{Resende's duality theorem} \label{s:inv:setting}

The goal of this section is to show that the results of Resende \cite{Re} emerge naturally from ours.
We recall the definition of a stably supported quantale from \cite{Re}. 
Let $Q$ be a unital involutive quantale with the unit $e$ and the involution $(-)^*$. 
A {\em support} on $Q$ is a sup-lattice endomorphism $\zeta: Q\to Q$ satisfying, for all $a\in Q$, the following axioms:
\begin{enumerate}[(S1)]
\item  $\zeta(a)\leq e$.
\item  $\zeta(a)\leq aa^*$.
\item $a\leq \zeta(a)a$.
\end{enumerate}
A support $\zeta$ is called {\em stable} provided that
\begin{enumerate}[(S4)]
\item  $\zeta(ab)=\zeta(a\zeta(b))$.
\end{enumerate}
A {\em stably supported quantale} is a unital involutive quantale which is equipped with a specific stable support.
A {\em stable quantal frame} is a stably supported quantale which is also a frame.

\begin{proposition}\label{prop:support} Let $Q$ be a stably supported quantale with the stable support $\zeta$. Let $\rho$ be the map $\zeta$ with the codomain restricted to $e^{\downarrow}$ and put $\lambda(a)=\rho(a^*)$ for all $a\in Q$.  
Then $Q$ is an involutive Ehresmann quantale with the right support $\rho$ and the left support $\lambda$.
\end{proposition}

\begin{proof} By (S1), $\rho$ is well defined and of course $\rho(a)=\zeta(a)$ for all $a\in Q$. 
Since $\zeta$ preserves joins, so does $\rho$. 
That $\zeta(a)=a$ for $a\leq e$ and $a=\zeta(a)a$ for all $a\in Q$ is easy to derive from (S1)--(S3), or see \cite[Lemma 3.3]{Re}. 
This and (S4) show that $\rho$ satisfies  axioms (EQ1)--(EQ4). 
From $\lambda(a)=\rho(a^*)$, $a\in Q$, it is immediate that $\lambda$ satisfies axioms (EQ1)--(EQ4), too.  
\end{proof}

Let $Q$ be an involutive Ehresmann quantale. Let $\lambda$, $\rho$ and $(-)^*$ denote its left support, right suport and involution, respectively.
Assume $Q$ is a stably supported quantale, $\zeta$ is the stable support (as is shown in \cite{Re}, $\zeta$, if exists, is unique) and $(-)^*$ is the involution. 
Then, by Proposition \ref{prop:support}, setting $\rho'(a)=\zeta(a)$ and $\lambda'(a)=\zeta(a^*)$ makes $Q$  an involutive Ehresmann quantale. 
This, and the uniqueness of $\lambda$ and $\rho$,  yields that $\rho'=\rho$ and $\lambda'=\lambda$. 
It follows that if there is a structure of a stably supported quantale on an involutive Ehresmann quantale $Q$, then it must be $\zeta(a)=\rho(a)$. 

Comparing the definitions of an Ehresmann quantale and a stably supported quantale we obtain the following statement.

\begin{proposition}\label{prop:stably_sup} Let $Q$ be an involutive Ehresmann quantale. 
Then $Q$ is a stably supported quantale if and only if $\rho$ satisfies axiom (S2). 
If $Q$ is a stably supported quantale then its support $\zeta$ is given by $\zeta(a)=\rho(a)$, $a\in Q$.
\end{proposition}

\begin{example} {\em Assume that the binary relation $A$ on a set $X$ is an equivalence relation. 
The Ehresmann quantal frame $\mathcal{P}(A)$ from Examples \ref{ex:bin_rel},  \ref{ex:bin_rel1}, \ref{ex:bin_rel2} is involutive where the involution
sends a relation to its converse. 
It is easy to see that axiom (S2) holds for $\rho$, so $\mathcal{P}(A)$ is a stable quantal frame.}
\end{example} 

Proposition \ref{prop:stably_sup} tells us that stably supported quantales form a subclass of involutive Ehresmann quantales. 
The following example shows that this subclass is proper.

\begin{example}\label{ex:pm2} {\em Let $M$ be a commutative monoid with the identity element $e$. 
We assume that $M$ is not a $2$-group. 
The powerset ${\mathcal{P}}(M)$ is an Ehresmann quantal frame with projections $\{e\}$ and the empty set. 
Equipped with a trivial involution, it becomes an involutive Ehresmann quantal frame. 
Let $a\in M$ be such that $a^2\neq e$. 
Then $\{e\} \not\subseteq \{a\}\{a\}^*=\{a^2\}$. 
Thus (S2) does not hold.}
\end{example}

We now turn to partial units of stably supported quantales and their relationship to partial isometries.

\begin{example} 
{\em Consider the Ehresmann quantal frame ${\mathcal{P}}(S)$ from Example \ref{ex:joins}. 
With a trivial involution, it becomes an involutive Ehresmann quantal frame.
It is a stable quantal frame since (S2) for $\rho$ trivially holds. 
We have that $\{x\}$ is a partial isometry, but not a partial unit. 
Note that in this example partial isometries are not closed with respect to multiplication since $\{x\}\cdot \{x\}=\{1,x\}$.}
\end{example}

\begin{proposition} \label{prop:isom_units} Let $Q$ be a stably supported quantale.
\begin{enumerate} 
\item Any partial unit of $Q$ is a partial isometry.
\item Assume that the set of partial isometries of $Q$ is closed under multiplication. Then any partial isometry of $Q$ is a partial unit.
\end{enumerate}
\end{proposition}

\begin{proof} The first statement follows from part (2) of Lemma~\ref{lem:p_un}.

Assume now that the set of partial isometries of $Q$ is closed with respect to multiplication. 
Let $\zeta$ be the support map.  
By proposition \ref{prop:stably_sup} we have that $\rho=\zeta$ and $\lambda=\zeta \circ (-)^*$ are the right and the left support of the involutive Ehresmann quantale $Q$. 
Note that $e^*=e$ since $e$ is a multiplicative unit. 
Let $f\leq e$. 
Then $f^*\leq e$. 
We have, applying Lemma \ref{lem:p_un}, $f=\rho(f)=ff^*=f^*f=\rho(f^*)=f^*$. 
Suppose that $a\in Q$ is a partial isometry. 
We first show that  $a^*$ is a partial isometry. 
Let $\leq$ be the partial order that underlies the frame $Q$, and $\leq'$ be the natural partial order of the Ehresmann monoid $Q$. 
Assume that $b\leq a^*$. 
Then $b^*\leq' a$ and so $b^*=fa$ for some $f\leq e$. 
It follows that $b=a^*f^*=a^*f$. 
Likewise, $b=ga^*$ for some $g\leq e$. 
Therefore, $b\leq' a^*$. 
This proves that $a^*$ is a partial isometry. 
Since partial isometries are by the assumption closed with respect to multiplication, $aa^*$ is a partial isometry, too. 
Note that $\rho(aa^*)=\rho(a\rho(a^*))=\rho(a\lambda(a))=\rho(a)=\rho(\rho(a))$.  
By (S2) we have $aa^*\geq \rho(a)$, which yields $aa^*\geq' \rho(a)$.  
Therefore, since partial isometries form a restriction monoid by Proposition \ref{prop:pi6}, it follows that $aa^*=\rho(a)$. 
Hence $aa^*\leq e$. 
Then also $a^*a=(aa^*)^*\leq e$. 
It follows that $a$ is a partial unit. 
\end{proof}

Let us call a multiplicative stable quantal frame $Q$ {\em strongly multiplicative} if it is strongly multiplicative as an Ehresmann quantal frame. 
This is the case precisely when the semiopen multiplication map of the localic category ${\mathcal C}(Q)$, associated to $Q$, is open. 

\begin{corollary} \label{cor:units} Let $Q$ be a strongly multiplicative stable quantal frame. 
Then an element of $Q$ is a  partial isometry if and only if it is a partial unit. 
\end{corollary}
\begin{proof} The statement follows from Proposition \ref{prop:isom_units} and Proposition \ref{prop:aug30}.
\end{proof}
  
An {\em inverse quantal frame} \cite{Re} is a stable quantal frame in which the top element is a join of partial units. 
It follows from \cite[Theorem 4.19]{Re} and \cite[Example 5.6] {Re} that as an Ehresmann quantal frame it is strongly multiplicative. 
By Corollary~\ref{cor:units}, partial units of an inverse quantal frame are the same as partial isometries.
Therefore an inverse quantal frame is an \'etale Ehresmann quantal frame where partial isometries are closed with respect to the multiplication, and hence a restriction quantal frame.
Assume now that $Q$ is an involutive restriction quantal frame. 
Then $Q$ need not be an inverse quantal frame, as (S2) might not hold for $\rho$, see Example~\ref{ex:pm2}. 
However, from Proposition \ref{prop:stably_sup} it is easy to deduce the following.

\begin{proposition}\label{prop:inv_qf} Let $Q$ be an involutive restriction quantal frame. 
Then $Q$ is an inverse quantal frame if and only if $\rho$ satisfies axiom (S2). 
\end{proposition}

We now establish conditions under which an involutive complete restriction monoid is a pseudogroup.
The proof of the following is routine.

\begin{proposition}\label{prop:inv_ps} Let $S$ be an involutive complete restriction monoid. 
Then $S$ is a pseudogroup if and only if $aa^*=\rho(a)$ and $a^*a=\lambda(a)$ for all $a\in S$.
\end{proposition}

Assume $S$ is a pseudogroup. 
Then it is easy to see that ${\mathcal L}^{\vee}(S)$ is such that (S2) holds for $\rho$. 
Indeed, let $a\in {\mathcal L}^{\vee}(S)$. 
Then $a = \vee B$ where $B\subseteq \eta(S)$. 
Then 
$$aa^*=\left (\bigvee B\right )\left (\bigvee_{b\in B}b^*\right )\geq \bigvee_{b\in B}bb^*=\bigvee_{b\in B}\rho(b)=\rho(a).
$$
Conversely, assume that $Q$ is an involutive restriction quantal frame such that (S2) holds for $\rho$. 
Let $a\in {\mathcal{PI}}(Q)$. 
Then $\rho(aa^*)=\rho(\lambda(a)a^*)=\rho(a^*)$. This and (S2) imply that $\rho(a)=aa^*$. It follows that ${\mathcal{PI}}(Q)$ is a pseudogroup. 
It is easy to see that the category of pseudogroups and their morphisms is a full subcategory of the category of involutive complete restriction monoids and their morphisms. A similar remark applies to the category of inverse quantal frames.
This discussion leads to the following consequence of Theorem \ref{th:inv_quant_psgrps}, which for the case of morphisms of type~$1$ was proved in \cite{Re}.

\begin{theorem}\label{th:resende1}
For each $k=1,2,3,4$ the category of pseudogroups and their morphisms of type $k$
is equivalent to the category of inverse quantal frames and their morphisms of type $k$. 
\end{theorem}

Let $C=(C_1,C_0)$ be an involutive quantal localic category. 
This category is a groupoid with respect to $i$ provided that the involution map $i$ is in fact an inversion map. 
That is, the following axioms hold:
$$
m\langle id,i\rangle = ur \mbox{ and } m\langle i,id \rangle = ud.
$$
Note that since $i$ is an involution, each of these axioms is a consequence of the other.
It can be easily shown that these axioms hold if and only if the involutive Ehresmann quantal frame ${\mathcal O}(C)$ satisfies the identities
$$
d^*u^*(a) =\bigvee_{x^*y\leq a} x\wedge y
\mbox{ and }
r^*u^*(a) =\bigvee_{xy^*\leq a}x\wedge y,
$$
where $d,r,u$ are structure maps of $C$. 
Each of these axioms is a consequence of the other.

Observe that $r^*u^*(a)=(a\wedge e)1$ and $d^*u^*(a)=1(a\wedge e)$. 
Indeed, from part (3) of Lemma \ref{lem:aux1} we have $r^*u^*(a)=u_!u^*(a)1$. 
From the Frobenius condition for the open map $u$ we have 
$$u_!u^*(a)=u_!(u^*(a)\wedge 1_{{\mathcal O}(C)}) = a\wedge u_!(1_{{\mathcal O}(C)})=a\wedge e.$$ 
Then $u_!u^*(a)1=(a\wedge e)1$. 
The equality $d^*u^*(a)=1(a\wedge e)$ is shown similarly.

\begin{remark}{\em Note that the inversion axioms, which are used in \cite{Re}, are different from ours in that $d$ and $r$ are switched. 
In our notation $ai(a)=ur(a)$, whereas in the notation from \cite{Re}  we have that $ai(a)=ud(a)$. }
\end{remark}

\begin{proposition} \label{prop:cat} Let $C=(C_1,C_0)$ be an involutive quantal localic category.
\begin{enumerate}
\item \label{i:oa24} The inequalities 
$$ r^*u^*(a) \geq \bigvee_{xy^*\leq a}x\wedge y, \,\,\,\, d^*u^*(a) \geq\bigvee_{x^*y\leq a} x\wedge y
$$
hold for all $a\in {\mathcal O}(C)$ if and only if the map $\rho$ of the involutive Ehrsmann quantal frame ${\mathcal O}(C)$ satisfies axiom (S2).
\item \label{i:ob24} The inequalities 
$$ r^*u^*(a) \leq \bigvee_{xy^*\leq a}x\wedge y, \,\,\,\, d^*u^*(a) \leq \bigvee_{x^*y\leq a} x\wedge y
$$
hold for all $a\in {\mathcal O}(C)$ if and only if every element of ${\mathcal O}(C)$ is a join of partial units.
\item  \label{i:oc24} $C$ is a groupoid if and only if $C$ is an \'etale groupoid if and only if ${\mathcal O}(C)$ is an inverse quantal frame.
\end{enumerate}
\end{proposition}
\begin{proof} (1) Assume that the stated inequalities hold. 
Let $a\in {\mathcal O}(C)$. 
From the first inequality we have 
$$
r^*u^*(aa^*) \geq \bigvee_{xy^*\leq aa^*}x\wedge y \geq a\wedge a=a.
$$
It follows by adjointness that $aa^*\geq u_!r_!(a)=\rho(a)$, and thus (S2) holds. 

Conversely, assume that (S2) holds and let $x,y\in {\mathcal O}(C)$. 
Then
$$u_!r_!(x\wedge y)=\rho(x\wedge y)\leq (x\wedge y)(x\wedge y)^*\leq xy^*.$$ 
Hence, by adjointness,
$x\wedge y\leq r^*u^*(xy^*)$. 
Then $x\wedge y\leq r^*u^*(a)$ whenever $xy^*\leq a$.
This implies the first inequality. The second one follows from the first one.

(2) The proof basically repeats the proof of \cite[Lemma 4.18]{Re}. 

(3) We combine (1) and (2) above. 
Assume $C$ is a groupoid. 
Since partial units of ${\mathcal O}(C)$ are partial isometries, 
(2) above implies that ${\mathcal O}(C)$ is \'etale. 
Since every partial isometry is a join of partial units, we may apply Lemma \ref{lem:aux24o} and conclude that partial isometries coincide with partial units. 
So the set of partial isometries is closed with respect to multiplication.  
By Theorem \ref{prop:aug302} we conclude that ${\mathcal O}(C)$ is strongly multiplicative. 
It follows that ${\mathcal O}(C)$ is a restriction quantal frame. 
Axiom (S2) holds for $\rho$ by (1) above. 
So ${\mathcal O}(C)$ is an inverse quantal frame. 
Conversely, if ${\mathcal O}(C)$ is an inverse quantal frame, partial isometries coincide with partial units by Proposition \ref{prop:isom_units}, 
and so $C$ is a groupoid by (1) and (2) above. 
This groupoid is automatically \'etale by Theorem \ref{cor:main_inv}, since ${\mathcal O}(C)$ is a restriction quantal frame.
\end{proof}

By an {\em \'etale localic groupoid} we mean an involutive  \'etale localic category where the involution map is a groupoid inversion.
In other words, an \'etale localic groupoid is a localic groupoid where the structure maps $u,i,m$ are open and $d,r$ are local homeomorphisms.

Proposition \ref{prop:cat} and Theorem \ref{th:resende1} lead to the following consequence of Theorem~\ref{cor:main_inv}.

\begin{theorem}[Resende's Duality Theorem] \label{th:resende2} \mbox{}
The following categories are equivalent:
\begin{enumerate}
\item The category of pseudogroups and proper $\wedge$-morphisms.
\item  The category of inverse quantal frames with proper involutive $\wedge$-morphisms.
\item The opposite of the category of  localic \'{e}tale groupoids and localic sheaf (groupoid) functors. 
\end{enumerate} 
\end{theorem}

The object part of the result above was proved in \cite{Re} but our result makes Resende's theorem functorial. 

\section{The Adjunction Theorem}\label{s:adjun}

The goal of this section is to establish the Adjunction Thorem, where we establish an adjunction between the category of \'etale localic categories and the category of \'etale topological categories. In fact, varying morphisms, we shall obtain four versions of this adjunction. The idea is to extend the classical spectrum and  open set functors which establish an adjunction between the categories of locales and topological spaces.

\subsection{Objects}
Let $C=(C_1,C_0)$ be a topological  category. Just similarly as with localic categories, we call $C$ {\em \'etale} if the structure maps $d,r$ are local homeomorphisms and $u,m$ are open. 

We first make a passage from localic objects to topological ones. Let $C$ be an \'etale localic category where $d,r,u$ and $m$ are its structure maps.  Applying the functor ${\mathsf{pt}}$ to this data we construct two topological spaces ${\mathsf{pt}}(C_1)$, ${\mathsf{pt}}(C_0)$ and the maps ${\mathsf{pt}}(d)$, ${\mathsf{pt}}(r)$, ${\mathsf{pt}}(u)$  and ${\mathsf{pt}}(m)$.  We aim to prove that the maps ${\mathsf{pt}}(d)$ and ${\mathsf{pt}}(r)$ are local homeomorphisms and the maps  ${\mathsf{pt}}(u)$  and ${\mathsf{pt}}(m)$ are open, so that the obtained data indeed defines an \'etale topological category.
We start from the following useful technical result.

\begin{lemma}\label{lem:homeom21} Let $f:L\to M$ be a semiopen locale map and let $a\in O(L)$.  

\begin{enumerate} 

\item[(1)] \label{i:21a} ${\mathsf{pt}}(f)(X_a)\subseteq X_{f_!(a)}$.
\end{enumerate}

Assume that the frame map $O(M)\to a^{\downarrow}$ given by $x\mapsto f^*(x)\wedge a$ is surjective. 

\begin{enumerate}
\item[(2)] \label{i:21b} The restriction of the map $O(M)\to a^{\downarrow}$, given by $x\mapsto f^*(x)\wedge a$, to $f_!(a)^{\downarrow}$ is an isomorphism of frames whose inverse isomorphism is $f_!|_{a^{\downarrow}}$.

\item[(3)] \label{i:21c} If $x\leq a$ then $x=f^*f_!(x)\wedge a$.

\item[(4)]  \label{i:21d} The map 
$${\mathsf{pt}}(f)|_{X_a}:X_a\to X_{f_!(a)}$$
a is homeomorphism and its inverse 
 is given by
$$
({\mathsf{pt}}(f)|_{X_a})^{-1}(q)=qf_!(-\wedge a).
$$
\end{enumerate}
\end{lemma}

\begin{proof} 
(1) Let $p\in X_a$. Then ${\mathsf{pt}}(f)(p)=pf^*\in X_{f_!(a)}$ since $pf^*f_!(a)\geq p(a)=1$ and $pf^*$ is a frame map.

(2) Assume that the map $O(M)\to a^{\downarrow}$ given by $x\mapsto f^*(x)\wedge a$ is surjective. Since $f^*f_!(a)\wedge a=a$, its restriction to $f_!(a)^{\downarrow}$,  denote it by $\alpha$, is a frame map from $f_!(a)^{\downarrow}$ to $a^{\downarrow}$. The map $\alpha$ is surjective since for any $y\leq a$ and
 $x\in O(M)$ such that $f^*(x)\wedge a=y$ we have $x\wedge f_!(a)\leq f_!(a)$ and
$$f^*(x\wedge f_!(a))\wedge a=f^*(x)\wedge f^*f_!(a)\wedge a=f^*(x)\wedge a =y.
$$
We show that $\alpha$ is injective. Let $x,y\leq f_!(a)$ be such that $f^*(x)\wedge a=f^*(y)\wedge a$. Applying the Frobenius condition we have
$$
f_!(f^*(x)\wedge a)=x\wedge f_!(a) =x.
$$
and similarly $f_!(f^*(y)\wedge a)=y$. It follows that $x=y$.
The inverse bijection $\alpha^{-1}: a^{\downarrow}\to f_!(a)^{\downarrow}$ equals $f_!|_{a^{\downarrow}}$ since for $x\leq f_!(a)$ we have
$f_!(f^*(x)\wedge a)=x$.

(3) The required equality follows from $\alpha\alpha^{-1}=id$.

(4) Due to part (1)  the map ${\mathsf{pt}}(f)|_{X_a}:X_a\to X_{f_!(a)}$ is well-defined. We show that this map is surjective.  Let $q\in X_{f_!(a)}$.  Note that $q(x)=1$ if, and only if, $q(x\wedge f_!(a))=1$. We define a function $p:O(L)\to{\bf 2}$ by
$$
p=qf_!(-\wedge a).
$$
Since $gf_!(a)=1$, we have $p(a)=1$. 
By part (2),  the map  $f_!|_{a^{\downarrow}}$ is a frame isomorphism, so that $p$ is a frame map. The equality $pf^*=q$ follows from the following calculation where we apply the Frobenius condition:
\begin{equation*}
pf^*(y)=qf_!(f^*(y)\wedge a)=q(y\wedge f_!(a))=q(y).
\end{equation*}
We show that the map ${\mathsf{pt}}(f)|_{X_a}$ is injective. Let $x\in O(L)$ and $p\in X_a$. By  part (3), we have $f^*f_!(x\wedge a)\wedge a=x\wedge a$. It follows that 
\begin{equation}\label{eq:calc21}
p(x) = p(x\wedge a) = p(f^*f_!(x\wedge a)\wedge a) = pf^*(f_!(x\wedge a)).
\end{equation}
Therefore, $pf^*=qf^*$ implies $p=q$.  Thus ${\mathsf{pt}}(f)|_{X_a}$ is injective.  It follows that ${\mathsf{pt}}(f)|_{X_a}$ is a continuous and open bijection and so a homeomorphism. It follows from \eqref{eq:calc21} that its inverse homeomorphism is given $q\mapsto qf_!(-\wedge a)$, where $q\in  X_{f_!(a)}$. This completes the proof. 
\end{proof}

We come to the following result.

\begin{proposition}\label{prop:loc_hom} Let $f:L\to M$ be a local homeomorphism of locales.  Then ${\mathsf{pt}}(f):{\mathsf{pt}}(L)\to {\mathsf{pt}}(M)$ is a local homeomorphism of topological spaces.
\end{proposition}

\begin{proof}
Assume that $C\subseteq O(L)$ is such that $1_{O(L)}=\bigvee C$ and
for every $c\in C$ the frame map $O(M)\to c^{\downarrow}$ given by $x\mapsto f^*(x)\wedge c$ is surjective.
Let $p\in {\mathsf{pt}}(L)$. Since $p \in X_{1_{O(L)}}$  then there is some $c\in C$ such that $p\in X_c$. 
By part (4) of Lemma \ref{lem:homeom21} we have that $X_c\to X_{f_!(c)}$ is a homeomorphism, so that $X_c$ is an open neighborhood of $p$ which is homeomorphic to its image under ${\mathsf{pt}}(f)$.
\end{proof}

\begin{remark} {\em It is natural to ask if ${\mathsf{pt}}(f)$ is always an open map whenever $f$ is an open locale map. We suspect that the answer to this question is negative, though we do not have a counterexample, neither have we found this question treated in the literature.}
\end{remark}

The map ${\mathsf{pt}}(m)$ has as its domain the space ${\mathsf{pt}}(C_1\times_{C_0} C_1)$. In the following result we provide a convenient characterization of this space.

\begin{proposition}\label{prop:pullback} Let $C=(C_1,C_0)$ be an \'etale localic category.
The space $${\mathsf{pt}}(C_1\times_{C_0} C_1)$$ is  homeomorphic to the pullback space 
$${\mathsf{pt}}(C_1)\times_{{\mathsf{pt}}(C_0)} {\mathsf{pt}}(C_1)=\{(p,q)\in {\mathsf{pt}}(C_1)\times {\mathsf{pt}}(C_1)\colon {\mathsf{pt}}(d)(p)={\mathsf{pt}}(r)(q)\}.
$$
\end{proposition}

\begin{proof} We first construct two mutually inverse maps between the sets ${\mathsf{pt}}(C_1\times_{C_0} C_1)$ and ${\mathsf{pt}}(C_1)\times_{{\mathsf{pt}}(C_0)} {\mathsf{pt}}(C_1)$. By definition ${\mathsf{pt}}(C_1\times_{C_0} C_1)$ is the set of all frame maps 
$O(C_1\times_{C_0} C_1)\to{\bf 2}$. By part \eqref{i:aux4} of Lemma  \ref{lem:aux1}, $O(C_1\times_{C_0} C_1)$ equals the tensor product of frames
$O(C_1)\otimes_{u_!(O(C_0))} O(C_1)$.  Let $f$ be a frame map from this tensor product to ${\bf 2}$.  We define
\begin{equation}\label{eq:gamma}
\gamma(f)=(f(-\otimes 1),f(1\otimes -))\in {\mathsf{pt}}(C_1)\times {\mathsf{pt}}(C_1).
\end{equation}
Let $x\in O(C_0)$. Applying \eqref{eq:jul10} it follows that 
$$
pd^*(x)=f(d^*(x)\otimes 1)=f(1\otimes r^*(x))=qr^*(x).
$$
Hence ${\mathsf{pt}}(d)(p)={\mathsf{pt}}(r)(q)$ and thus $\gamma(f)\in {\mathsf{pt}}(C_1)\times_{{\mathsf{pt}}(C_0)} {\mathsf{pt}}(C_1)$.

Conversely, let $(p,q)\in {\mathsf{pt}}(C_1)\times_{{\mathsf{pt}}(C_0)} {\mathsf{pt}}(C_1)$ and define 
$$\delta(p,q):O(C_1)\otimes_{u_!(O(C_0))} O(C_1)\to{\bf 2}$$ by
\begin{equation}\label{eq:delta}
\delta(p,q)(a\otimes b)=p(a\rho(b))\wedge q(\lambda(a)b).
\end{equation}
It is routine to verify that $\delta(p,q)$ is a well-defined frame map, so that $\delta(p,q)\in {\mathsf{pt}}(C_1\times_{C_0} C_1)$. 

We now show that the assignments $\gamma$ and $\delta$ are mutually inverse. Let $f\in {\mathsf{pt}}(C_1\times_{C_0} C_1)$. 
Then
\begin{multline*}
\delta\gamma(a\otimes b)=f(a\rho(b)\otimes 1)\wedge f(1\otimes \lambda(a)b)=f((a\rho(b)\otimes 1)\wedge (1\otimes \lambda(a)b))\\ =
f(a\rho(b)\otimes \lambda(a)b)=f(a\otimes b).
\end{multline*}
Now let $(p,q)\in {\mathsf{pt}}(C_1)\times_{{\mathsf{pt}}(C_0)} {\mathsf{pt}}(C_1)$. Then $\gamma\delta(p,q)=(f(-\otimes 1), f(1\otimes -))$.  For any $x\in O(C_1)$ we calculate
\begin{align*}
f(x\otimes 1) & = f(x\otimes r^*d_!(x))& \text{by part \eqref{i:aux13} of Lemma \ref{lem:aux1}}\\
& = p(x)\wedge qr^*d_!(x)  = p(x) \wedge pd^*d_!(x) & \text{since } qr^*=pd^*\\
&\geq p(x) & \text{since } d^*d_!\geq id.
\end{align*}
From this calculation, we also obtain that $p(x)\geq p(x)\wedge q(u_!d_!(x)1) =f(x\otimes 1)$. Hence 
$p=f(-\otimes 1)$. By symmetry we also have $q= f(1\otimes -)$. 

It follows that the  assignments $\gamma$ and $\delta$ are mutually inverse bijections between the sets ${\mathsf{pt}}(C_1\times_{C_0} C_1)$ and ${\mathsf{pt}}(C_1)\times_{{\mathsf{pt}}(C_0)} {\mathsf{pt}}(C_1)$. We now show that these bijections are in fact homeomorphisms.  The basis of the topology of the space ${\mathsf{pt}}(C_1\times_{C_0} C_1)$ is  formed by the sets $X_{a\otimes b}$ where $a\otimes b$ runs through
$O(C_1)\otimes_{u_!(O(C_0))} O(C_1)$. Using the fact that the category is \'etale, $a$ and $b$ can be chosen local bisections. The topology on ${\mathsf{pt}}(C_1)\times_{{\mathsf{pt}}(C_0)} {\mathsf{pt}}(C_1)$ is the subspace topology of the product topology on ${\mathsf{pt}}(C_1)\times {\mathsf{pt}}(C_1)$ whose basis is, by definition, formed by the sets $X_a\times X_b$, where $a,b$ run through local bisections of $O(C_1)$. For $a,b\in O(C_1)$ we put
$$(X_a\times X_b)'=(X_a\times X_b)\cap ({\mathsf{pt}}(C_1)\times_{{\mathsf{pt}}(C_0)} {\mathsf{pt}}(C_1)).
$$
The sets $(X_a\times X_b)'$ then form a basis of the topology on ${\mathsf{pt}}(C_1)\times_{{\mathsf{pt}}(C_0)} {\mathsf{pt}}(C_1)$. 
We have
$$
(X_a\times X_b)'=\{(p,q)\in X_a\times X_b\colon pd^*=qr^*\}.
$$

We now show that  $$\gamma(X_{a\otimes b})=(X_a\times X_b)'.$$
Let $f\in X_{a\otimes b}$, that is $f(a\otimes b)=1$. Then clearly $f(a\otimes 1)=1$ and $f(1\otimes b)=1$ which implies that 
$$(f(-\otimes 1), f(1\otimes -))\in (X_a\times X_b)'.$$
Conversely, assume that $(p,q)\in (X_a\times X_b)'$. We show that $p(a\rho(b))=1$. We have
\begin{align*}
p(au_!r_!(b)) & = p(a\wedge d^*r_!(b)) & \text{ by part \eqref{i:aux3} of  Lemma \ref{lem:aux1}}\\
& = pd^*r_!(b) & \text{since } p(a)=1\\
& = qr^*r_!(b) & \text{since } pd^*=qr^*\\
& \geq q(b)=1 & \text{since } r^*r_!\geq id.
\end{align*}
Dually, we also have that $q(\lambda(a)b)=1$. Then 
$f(a\otimes b)=p(a\rho(b))\wedge q(\lambda(a)b)=1.$
This implies that $\gamma$ and $\delta$ are homeomorphisms, which completes the proof.
\end{proof}

For $(p,q)\in {\mathsf{pt}}(C_1)\times_{{\mathsf{pt}}(C_0)} {\mathsf{pt}}(C_1)$ we put $p\otimes q=\delta(p,q)$.
From now on we identify the spaces ${\mathsf{pt}}(C_1\times_{C_0} C_1)$ and ${\mathsf{pt}}(C_1)\times_{{\mathsf{pt}}(C_0)} {\mathsf{pt}}(C_1)$ via the homeomorphisms $\gamma$ and $\delta$ constructed in the proof of Proposition \ref{prop:pullback}. In particular, this yields a convention that the map ${\mathsf{pt}}(m)$ has as its domain the space ${\mathsf{pt}}(C_1)\times_{{\mathsf{pt}}(C_0)} {\mathsf{pt}}(C_1)$. Let $(p,q)\in {\mathsf{pt}}(C_1)\times_{{\mathsf{pt}}(C_0)} {\mathsf{pt}}(C_1)$. Since $(p,q)=\gamma(p\otimes q)$ where $p\otimes q= \delta(p,q)$ is given by \eqref{eq:delta}, it follows that under our convention the action of the map ${\mathsf{pt}}(m)$ is given by the formula
\begin{equation}\label{eq:ptm}
{\mathsf{pt}}(m)(p,q)(x)=(p\otimes q)m^*.
\end{equation}

We will make use of the following technical lemma.

\begin{lemma}\label{lem:u21}
Let $C=(C_1,C_0)$ be an \'etale localic category and let $p,q\in {\mathsf{pt}}(C_1)$ be such that $pd^*=qd^*$ (resp. $pd^*=qr^*$). If $a,b\in O(C_1)$ are local bisections satisfying the condition $p(a)=q(b)=1$ then for any $x\leq a$ and $y\leq b$ with $d_!(x)=d_!(y)$ (resp. $d_!(x)=r_!(x)$) we have $p(x)=q(y)$. Similar statements also hold if the symbols $d$ and $r$ are interchanged. 
\end{lemma}

\begin{proof}
By part (3) of Lemma \ref{lem:homeom21} we have $x=d^*d_!(x)\wedge a$ and $y=d^*d_!(y)\wedge b$. Then, in view of $p(a)=1$ and $q(b)=1$, we have
$$
p(x)=p(d^*d_!(x)\wedge a)=pd^*d_!(x)=qd^*d_!(y)=q(d^*d_!(y)\wedge b)=q(y).
$$
The second statement follows similarly.
\end{proof}

\begin{theorem} Let $C=(C_1, C_0)$ be an  \'etale localic category. Then it gives rise to an \'etale topological  category 
${\mathsf{Pt}}(C)=({\mathsf{pt}}(C_1), {\mathsf{pt}}(C_0))$ whose structure maps are ${\mathsf{pt}}(u)$, ${\mathsf{pt}}(d)$, ${\mathsf{pt}}(r)$ and ${\mathsf{pt}}(m)$.\end{theorem}

\begin{proof} 
Let $a\in O(C_0)$ and show that 
\begin{equation}\label{eq:ptu} {\mathsf{pt}}(u)(X_a)=X_{u_!(a)}.
\end{equation}
 By part~(1) of Lemma~\ref{lem:homeom21} it is enough to prove that $X_{u_!(a)}\subseteq {\mathsf{pt}}(u)(X_a)$. Let $g\in X_{u_!(a)}$. For $x\in O(C_0)$ we define $f(x)=gu_!(x\wedge a)$. This is clearly a frame map and $f(a)=gu_!(a)=1$. To verify that $g(x)=fu^*(x)$ we calculate
$$
fu^*(x)  =gu_!(u^*(x)\wedge a) = gu_!u^*(x)\wedge gu_!(a)=gu_!u^*(x)=g(x\wedge e)=g(x),
$$
as required, where we have applied part \eqref{i:aux_new}~of Lemma \ref{lem:aux1} and the fact that $g(e)=1$ since $e\geq u_!(a)$ and $gu_!(a)=1$. We have proved that the map ${\mathsf{pt}}(u)$ is open.

Since $d$ and $r$ are local homeomorphisms of locales, ${\mathsf{pt}}(d)$ and ${\mathsf{pt}}(r)$ are local homeomorphisms of topological spaces by Proposition \ref{prop:loc_hom}. 

We now prove that the map ${\mathsf{pt}}(m)$ is open. For this we prove that 
$${\mathsf{pt}}(m)((X_a\times X_b)')=X_{m_!(a\otimes b)}.
$$  
By part (1) of Lemma \ref{lem:homeom21} it is enough to prove that $X_{ab}\subseteq {\mathsf{pt}}(m)((X_a\times X_b)')$. We may also assume that $a$ and $b$ are local bisections. The product $ab$ is a local bisection as well. Passing, if necessary, from $a$ to $a\rho(b)$ and from $b$ to $\lambda(a)b$ we may assume without loss of generality that $d_!(a)=r_!(b)$, that is that $ab$ is a restricted product.  
Since $\lambda(ab)=\lambda(b)$ we have $d_!(ab)=d_!(b)$. By symmetry we also have $r_!(ab)=r_!(a)$. Composing homeomorphisms constructed in the proof of Proposition~\ref{prop:loc_hom} we can construct the following two paths between $X_{ab}$ and $X_{d_!(a)}$:
$$
X_{ab}- X_{r_!(a)}- X_a - X_{d_!(a)};
$$
$$
X_{ab} - X_{d_!(b)} - X_b - X_{r_!(b)}.
$$
Let $p\in X_{ab}$. We calculate its images in $X_{d_!(a)}=X_{r_!(b)}$ if we compose the homeomorphisms along each of these paths. Consider the first path. The image of $p$ in $X_{r_!(a)}$ is $pr^*$, then the image of $pr^*$ in $X_a$ is $pr^*r_!(-\wedge a)$,  and finally the image of $pr^*r_!(-\wedge a)$ in $X_{d_!(a)}$ is $pr^*r_!(d^*(-)\wedge a)$. Similarly, for the second path we obtain that the image of $p$ equals $pd^*d_!(r^*(-)\wedge b)$. We now notice that the resulting elements of $X_{d_!(a)}$  are equal.  Indeed, if $x\in O(C_0)$ we have
\begin{align*}
pr^*r_!(d^*(x)\wedge a) & = pr^*r_!(au_!(x)) & \text{by part \eqref{i:aux3} of Lemma \ref{lem:aux1}} \\
& = p(r^*r_!(au_!(x))\wedge ab) & \text{since } p(ab)=1\\
& = p(u_!r_!(au_!(x))ab) & \text{by part \eqref{i:aux3} of Lemma \ref{lem:aux1}} \\
& = p(au_!(x)b), & 
\end{align*} 
where we have used the equality $\rho(au_!(x))a=au_!(x)$. This equality holds for the following reasons: (i) $a$ is a local bisection and thus belongs to complete restriction monoid formed by the local bisections, (ii) $au_!(x)\leq a$  and (iii) in a restriction monoid the inequality $b\leq c$ implies $b=\rho(b)c$. By symmetry, we also have that $pd^*d_!(r^*(x)\wedge b)=p(au_!(x)b)$. Thus the two elements are indeed equal. We put 
$$q=pr^*r_!(d^*(-)\wedge a) =pd^*d_!(r^*(-)\wedge b)= p(au_!(-)b).$$
We also put $s=pr^*r_!(-\wedge a)$ and  $t=pd^*d_!(-\wedge b)$.  We have that $s\in X_a$ and $t\in X_b$. Moreover, $s$ and $t$ are the images of $p$ in $X_a$ and $X_b$ under compositions of appropriate homeomorphisms along the paths above. In particular, we have that $pr^*=sr^*$, $pd^*=td^*$ and $sd^*=tr^*$. The latter equality means that $(s,t)$ belongs to  the set $(X_a\times X_b)'$.

We now prove that $p={\mathsf{pt}}(m)(s,t)$. By \eqref{eq:ptm} we need to show that for every $y\in O(C_1)$ we have
\begin{multline*}
p(y)=(s\otimes t)\left(\bigvee\{x\otimes z\colon xz\leq y, \lambda(x)=\rho(z)\}\right)=\\
\bigvee \{s(x)\wedge t(z)\colon xz\leq y, \lambda(x)=\rho(z)\}.
\end{multline*}
Assume that for some $x,z$ such that $xz\leq y$ and $\lambda(x)=\rho(z)$ we have $s(x)\wedge t(z)=1$. Since $s\in X_a$ and $t\in X_b$ this yields $s(x\wedge a)=1$ and $t(z\wedge b)=1$. Since $q=sd^*$, we have $qd_!(x\wedge a)=sd^*d_!(x\wedge a)\geq s(x\wedge a)=1$ and dually $qr_!(x\wedge b)=1$. Setting $c=d_!(x\wedge a)\wedge r_!(x\wedge b)$ we obtain $q(c)=1$. Since $u_!(c)=\lambda(x\wedge a)\wedge \rho(x\wedge b)\leq \lambda(x\wedge a)$ it follows that
$(x\wedge a)u_!(c)=au_!(c)$. Thus by part \eqref{i:aux3} of Lemma  \ref{lem:aux1}
we obtain
$$
s((x\wedge a)u_!(c))=s(au_!(c))=s(d^*(c)\wedge a)=q(c)=1.
$$
 Since $x\wedge a$, $z\wedge b$ and their product are local bisections,
applying the equality $\rho(mn)=\rho(m\rho(n))$ we have $r_!((x\wedge a)(z\wedge b))=r_!((x\wedge a)u_!(c))$.
Applying Lemma~\ref{lem:u21} we obtain $p((x\wedge a)(z\wedge b))=s((x\wedge a)u_!(c))=1$.
But then $$p(y)\geq p(xz)\geq p((x\wedge a)(z\wedge b))=1.$$
Conversely, assume that $p(y)=1$. Then $p(y\wedge ab)=1$. Put $x=\rho(y\wedge ab)a$ and $z=b\lambda(y\wedge ab)$. We have $r_!(x)=r_!(y\wedge ab)$, $x\leq a$ and $y\wedge ab\leq ab$. Since also $sr^*=pr^*$, by Lemma \ref{lem:u21} we obtain $s(x)=p(y\wedge ab)=1$. Dually we obtain that $t(z)=1$. This completes the proof that ${\mathsf{pt}}(m)$ is open.

That the axioms of an internal category hold is immediate using the functoriality of the assignment ${\mathsf{pt}}$.
\end{proof}

Let $C$ be an \'etale topological category.
We define a {\em local bisection} of $C$ as an open subset $A\subseteq C_1$ such that the restrictions of the maps $d$ and $r$ to $A$ are injective. 

\begin{lemma} Local bisections form a basis of the topology on $C_1$.
\end{lemma}

\begin{proof} Let $A$ be an open set in $C_1$ and $a\in A$. It is enough to show that there is a local bisection $B_a$ containing $a$. Then $A\cap B_a$ is a local bisection containing $a$ as well and $A=\bigcup_{a\in A} (A\cap B_a)$.
As $d$ and $r$ are local homeomorphisms, there are  open neighborhoods $P_a$ and $Q_a$ of $a$ such that $P_a$ is homeomorphic to $d(P_a)$ and $Q_a$ to $r(Q_a)$. It follows that $P_a\cap Q_a$ is a local bisection containing $a$ and we can put $B_a=P_a\cap Q_a$.
\end{proof}
 
Observe that the frame $\Omega (C_0)$ acts on  $\Omega(C_1)$ on the right and left by $A\mapsto Au(B)$ and $A\mapsto u(B)A$.
Let $\Omega(C_1)\otimes_{\Omega(C_0)} \Omega(C_1)$ be the tensor product with respect to these actions.

\begin{lemma}\label{lem:pullback23} The frame $\Omega(C_1\times_{C_0}C_1)$ is isomorphic to the tensor product frame
$\Omega(C_1)\otimes_{\Omega(C_0)} \Omega(C_1)$.
\end{lemma}

\begin{proof} Recall that a basis of the topology on $C_1\times_{C_0}C_1$ is formed by the sets
$X_{A,B}=\{(a,b)\in A\times B\colon r(b)=d(a)\}$
where $A$ and $B$ run through local bisections of $C$. The set $r(B)\cap d(A)$ is open and so we can consider the restriction $A'$ of $A$ to $r(B)\cap d(A)$ with respect to $d$ and $B'$ of $B$ to $r(B)\cap d(A)$ with respect to $r$.
We have  that $A'=Au(r(B)\cap d(A))$ and $B'=u(r(B)\cap d(A))B$ are local bisections and $X_{A,B}=X_{A',B'}$. It is now immediate that the map $X_{A,B}\mapsto A\otimes B$ establishes the need isomorphism.
\end{proof}

In what follows we will identify the frames $\Omega(C_1\times_{C_0}C_1)$  and
$\Omega(C_1)\otimes_{\Omega(C_0)} \Omega(C_1)$ via the isomorphisms established in the proof of Lemma \ref{lem:pullback23}. 

\begin{lemma} \label{lem:top_lh} Let $f:X\to Y$ be a local homeomorphism of topological spaces. Then $\Omega(f)\colon \Omega(X)\to \Omega(Y)$ is a local homeomorphism of locales.
\end{lemma}

\begin{proof} As $\Omega(X)$ and $\Omega(Y)$ are spatial locales and $f$ is open, $\Omega(f)$ is an open map. Let $A\subseteq C_1$ be a local bisection and $B$ its open subset. Then $B=f^{-1}f(B)\cap A$ which implies that the map $\Omega(Y)\to A^{\downarrow}$ given by $Y\mapsto f^{-1}(Y)\cap A$ is surjective. The statement follows.
\end{proof}

\begin{proposition} Let $C=(C_1,C_0)$ be an \'etale topological category. Then it gives rise to an \'etale localic category
$\overline{\Omega}(C)=(\Omega(C_1),\Omega(C_0))$ with the structure maps $\Omega(u)$, $\Omega(d)$, $\Omega(r)$ and $\Omega(m)$.
\end{proposition}

\begin{proof} The semiopen maps $\Omega(u)$, $\Omega(d)$, $\Omega(r)$ and $\Omega(m)$  between spatial locales are open. By Lemma \ref{lem:top_lh}  $\Omega(d)$ and $\Omega(r)$ are local homeomorphisms. 
Axioms of an internal category hold since the assignment $\Omega$ is functorial.
\end{proof}

A restriction semigroup is said to be {\em right ample} (resp. {\em left ample}) if $ac = bc$ implies $a\rho (c) = b \rho (c)$ (resp. $ca = cb$ implies $\lambda(c)a = \lambda(c)b$). It is {\em ample} if it is both right and left ample.
A topological category $C$ is said to be {\em right cancellative} (resp. {\em left cancellative}) if
for any $a,b,c\in C_1$ with $d(a)=d(b)=r(c)$ we have that $ac=bc$ implies that $a=b$ (resp. for any $a,b,c\in C_1$ with $r(a)=r(b)=d(c)$ we have that $ca=cb$ implies that $a=b$). It is said to be {\em cancellative} if it is both right and left cancellative.

\begin{proposition}\label{prop:ample} Let $S$ be a left ample (resp. right ample, ample) complete restriction monoid. Then the topological category
${\mathsf{Pt}}({\mathcal C}(\mathcal{L}^{\vee}(S)))$ is left cancellative (resp. right cancellative, cancellative).
\end{proposition}

\begin{proof}
We prove the statement for the case when $S$ is left ample. Then the case when $S$ is right ample follows by symmetry, and the case when $S$ is ample follows as a combination of these two cases. We put $C={\mathcal C}(\mathcal{L}^{\vee}(S))$.
Then $S$ is isomorphic to the complete restriction monoid ${\mathcal{PI}}({\mathcal O}(C))$, by Theorems~\ref{th:quant_psgrps} and \ref{the:jan23}.  Let $u,d,r,m$ denote the structure maps of $C$. We first prove the following lemma.

\begin{lemma}\label{lem:new7}
 Let $p,q\in {\mathsf{pt}}(C_1)$ be such that ${\mathsf{pt}}(d)(p)={\mathsf{pt}}(r)(q)$.
If $p(a)=1$ and $q(b)=1$ then $p(a\rho(b))=1$ and $q(\lambda(a)b)=1$.
\end{lemma}

\begin{proof}
We have $a\rho(b)=au_!r_!(b)=a\wedge d^*r_!(b)$. Then
$$p(a\rho(b))=pd^*r_!(b)=qr^*r_!(b)\geq q(b)=1$$ where we have used that $pd^*=qr^*$ and $r^*r_!\geq id$.
The equality $q(\lambda(a)b)=1$ is proved similarly. 
\end{proof}

Assume that $p,q,s\in {\mathsf{pt}}(C_1)$ are such that ${\mathsf{pt}}(d)(p)={\mathsf{pt}}(r)(q)={\mathsf{pt}}(r)(s)$ and ${\mathsf{pt}}(m)(p,q)={\mathsf{pt}}(m)(p,s)$. We aim to show that $q=s$. By symmetry, it is enough to show that $s\geq q$. So we assume that $q(z)=1$ and aim to show that $s(z)=1$. Clearly, it is enough to assume that $z\in O(C_1)$ is a local bisection.

We have that ${\mathsf{pt}}(m)(p,q)=(p\otimes q)m^*$, that is for every $x\in O(C_1)$
$$
{\mathsf{pt}}(m)(p,q)(x)=\bigvee\{p(y)\wedge q(z)\colon yz\leq x, \lambda(y)=\rho(z)\},
$$
and we can write a similar expression for ${\mathsf{pt}}(m)(p,s)$. Let $y$ be a local bisection such that $p(y)=1$ and put
$x=yz$.  Let $y'=y\rho(z)$ and $z'=\lambda(y)z$. Since $p(y)\wedge q(z)=1$ we have that $p(y')\wedge q(z')=1$  by Lemma \ref{lem:new7}. Since $p(y')\wedge q(z')=1$, $y'z'\leq x$ and $\lambda(y')=\rho(z')$, it follows that ${\mathsf{pt}}(m)(p,q)(x)=1$. By assumption then also ${\mathsf{pt}}(m)(p,s)(x)=1$. This means that there are local bisections $a,b\in O(C_1)$ such that $\lambda(a)=\rho(b)$, $ab\leq x$ and also
$p(a)=1$ and $s(b)=1$. 

We put $A=(y'\wedge a)b$ and $B=(y'\wedge a)z'$. Since 
$A,B\leq x$, then $A\sim B$ by Lemma~\ref{lem:aug28}. But 
$$\rho(A)=\rho((y'\wedge a)\rho(b)) = \rho(y'\wedge a),$$
where we have used the inequality $\rho(b)\geq \rho(y'\wedge a)$, and similarly $\rho(B)=\rho(y'\wedge a)$. It follows that $A=B$. Applying the assumption that $S$ is left ample it follows that 
$$\lambda(y'\wedge a)b=\lambda(y'\wedge a)z'.$$

Since $p(y'\wedge a)=q(z')=1$ it follows by Lemma \ref{lem:new7} that $q(\lambda(y'\wedge a)z')=1$. 
Similarly, as $p(y'\wedge a)=s(b)=1$ it follows by Lemma \ref{lem:new7} that $s(\lambda(y'\wedge a)b)=1$.
Applying $A=B$ we obtain $s(\lambda(y'\wedge a)z')=1$. It follows that $s(z)=1$, because $z\geq \lambda(y'\wedge a)z'$. 
\end{proof}

\begin{proposition}\label{prop:ample1}
Let $C$ be a left cancellative  (resp. right cancellative, cancellative) \'etale topological category. Then ${\mathcal{PI}}(\Omega(C))$ is a left ample (resp. right ample, ample) complete restriction monoid.
\end{proposition}

\begin{proof} Again, it is enough to prove the statement only for the case when $C$ is left cancellative. Let $A,B,C\in \Omega(C_1)$ be local bisections and assume that $AB=AC$. Then 
$$\lambda(\rho(AB)A)= \lambda(A\rho(B))=\lambda(A)\rho(B).$$
Similarly, $\lambda(\rho(AC)A)=\lambda(A)\rho(C)$. It follows that $ud(A)ur(B)=ud(A)ur(C)$ which implies that $u(d(A)\wedge r(B))=u(d(A)\wedge r(C))$ and so $d(A)\wedge r(B)=d(A)\wedge r(C)$.
Let  $X=d(A)\wedge r(B)=d(A)\wedge r(C)$.  Then for every $x\in X$ there are the only elements $a_x\in A$, $b_x\in B$ and  $c_x\in C$ with $d(a_x)=r(b_x)=r(c_x)=x$ and
$$
AB=\{a_xb_x\colon x\in X\},\,\,\,\, AC=\{a_xc_x\colon x\in X\}.
$$
Since $AB=AC$ then $a_xb_x=a_xc_x$ for each $x\in X$. As $C$ is cancellative, it follows that $b_x=c_x$ for all $x\in X$. But this implies that $\lambda(A)B=\lambda(A)C$.  It follows that ${\mathcal{PI}}(\Omega(C))$ is a left ample complete restriction monoid.
\end{proof}


\subsection{Morphisms}
We now define the assignment ${\mathsf{Pt}}$ on morphisms. We define a morphism of \'etale localic categories as a morphism of the corresponding restriction quantal frames going in the opposite direction, see the discussion at the beginning of Subsection~\ref{sub:morphisms}. Thus, if $C=(C_1,C_0)$ and $D=(D_1,D_0)$ are \'etale localic categories, a morphisms from $C$ to $D$ is the opposite map to a map $f_1^*$ of restriction quantal frames from ${\mathcal O}(D)$ to ${\mathcal O}(C)$. We will denote this opposite map simply by $f_1$. 

We now  introduce the relevant class of maps between \'etale topological categories. 
Let $C=(C_1,C_0)$ and $D=(D_1,D_0)$ be \'etale topological categories. We define a {\em relational covering morphism} from $C$ to $D$ as a pair $f=(f_1,f_0)$, where 
\begin{itemize}
\item $f_0:C_0\to D_0$ is a continuous map,
\item $f_1:C_1\to {\mathcal{P}}(D_1)$ is a function,
\end{itemize}
 and the following axioms are satisfied:
\begin{enumerate}[(RM1)]
\item If  $b\in f_1(a)$ where $a\in C_1$ then $d(b)=f_0d(a)$ and $r(b)=f_0r(a)$. 
\item If $(a,b)\in C_1\times_{C_0}C_1$ and $(c,d)\in D_1\times_{D_0}D_1$ are such that $c\in f_1(a)$ and $d\in f_1(b)$ then $cd\in f_1(ab)$.
\item If $d(a)=d(b)$ (or $r(a)=r(b)$) where $a,b\in C_1$ and $f_1(a)\cap f_1(b)\neq \varnothing$ then $a=b$. 
\item If $p=f_0(q)$ and $d(s)=p$ (resp. $r(s)=p$) where $q\in C_0$ and $s\in D_1$ then there is $t\in C_1$ such that $d(t)=q$ (resp. $r(t)=q$) and $s\in f_1(t)$.
\item For any $A\in O(D_1)$: $f_1^{-1}(A)=\{x\in C_1\colon f_1(x)\cap A\neq \varnothing\}\in O(C_1)$.
\item $uf_0(t)\in f_1u(t)$ for any $t\in C_0$.

\end{enumerate}

Axiom
(RM2) is a weak form of preservation of multiplication;
(RM3) tells us that $f_1$ is {\em star-injective} and (RM4) that it is {\em  star-surjective};
(RM5) tells us that $f_1$ is a {\em lower-semicontinuous relation}. 

We remark that a relational covering morphism $f=(f_1,f_0)$ is entirely determined by $f_1$ and we have the equality $f_0(a)=df_1u(a)$ for any $a\in C_0$. Indeed,  by (RM6) we have that $uf_0(a)\in f_1u(a)$.  Then by (RM1) $duf_0(a)=df_1u(a)$. But $du=id$, so that $f_0(a)=df_1u(a)$.

\begin{lemma}\label{lem:l23} Let $C=(C_1, C_0)$ and $D=(D_1,D_0)$ be \'etale localic categories and
$f_1^*\colon {\mathcal{O}}(D)\to {\mathcal{O}}(C)$  a morphism of restriction quantal frames and put $f_0^*=d_!f_1^*u_!$. For  $q\in {\mathsf{pt}}(C_1)$ and a local bisection $c\in O(D_1)$ such that $qf_1^*(c)\neq 0$ we put
$$
p_c=({\mathsf{pt}}(d)|_{X_c})^{-1}({\mathsf{pt}}(f_0)({\mathsf{pt}}(d)(q))).
$$
Then $p_c\in {\mathsf{pt}}(D_1)$ is well defined and
$$
qf_1^*=\bigvee\{p_c\colon c\in O(D_1) \text{ is a local bisection and } qf_1^*(c)\neq 0\},
$$
where $(\bigvee_{c\in C}p_c)(x)=\bigvee_{c\in C}p_c(x)$.
\end{lemma}

\begin{proof} Using $d_!f_1^*=f_0^*d_!$, we observe that 
$${\mathsf{pt}}(f_0)({\mathsf{pt}}(d)(q))(d_!(c))=qd^*f_0^*d_!(c)=qd^*d_!f_1^*(c)\geq qf_1^*(c)=1,
$$
so that $p_c$ is well defined. For any $x\in O(D_1)$ we have
$f_1^*(x\wedge c)=d^*d_!f_1^*(x\wedge c)\wedge f_1^*(c)$ by part (3) of Lemma \ref{lem:homeom21}. Then 
\begin{multline*}
p_c(x)=qd^*f_0^*d_!(x\wedge c)=qd^*d_!f_1^*(x\wedge c)=qd^*d_!f_1^*(x\wedge c)\wedge qf_1^*(c) \\=
q(d^*d_!f_1^*(x\wedge c)\wedge f_1^*(c))=qf_1^*(x\wedge c)\leq qf_1^*(x).
\end{multline*}
It follows that $qf_1^*\geq \bigvee p_c$, where the join is taken over the set of all local bisections $c$ such that $qf_1^*(c)\neq 0$. On the other hand, we have $qf_1^*(x)\leq \bigvee p_c$, as if $qf_1^*(x)=1$ for a local bisection $x$ then $p_x(x)=1$.
\end{proof}

\begin{remark} {\em It is easy to see, interchanging $d$ and $r$ in the proof of Lemma \ref{lem:l23}, that for a local bisection
$c\in O(D_1)$ such that $qf_1^*(c)\neq 0$ we also have
$$
p_c=({\mathsf{pt}}(r)|_{X_c})^{-1}({\mathsf{pt}}(f_0)({\mathsf{pt}}(r)(q))).
$$}
\end{remark}

We may therefore define a function ${\mathsf{pt}}'(f_1):{\mathsf{pt}}(C_1)\to {\mathcal{P}}({\mathsf{pt}}(D_1))$ by putting
$$
{\mathsf{pt}}'(f_1)(q)=\{p_c \colon c\in O(D_1) \text{ is a local bisection and } qf_1^*(c)\neq 0\}.
$$
From now on, we take a convention that in the case when ${\mathsf{pt}}'(f_1)(q)=\{p\}$ is a one-element set, we identify it with its the only element $p$. 

\begin{proposition} Let $C=(C_1, C_0)$ and $D=(D_1,D_0)$ be \'etale localic categories and
$f_1^*\colon {\mathcal{O}}(D)\to {\mathcal{O}}(C)$  a morphism of restriction quantal frames. Put $f_0^*=d_!f_1^*u_!$.  Then 
$${\mathsf{Pt}}(f_1)=({\mathsf{pt}}'(f_1),{\mathsf{pt}}(f_0)):{\mathsf{Pt}}(C)\to {\mathsf{Pt}}(D)$$
is a relational covering morphism.  The assignment ${\mathsf{Pt}}$ is functorial.
\end{proposition}

\begin{proof}
Applying Proposition \ref{prop:morphisms} we can assume that axioms (M1)--(M4) are satisfied. We verify axioms (RM1)--(RM5).
(RM1) holds by the definition of ${\mathsf{pt}}'(f_1)$ and Lemma \ref{lem:l23}.

(RM2) Let $(a,b)\in {\mathsf{pt}}(C_1)\times_{{\mathsf{pt}}(C_0)} {\mathsf{pt}}(C_1)$ and $(p,q)\in {\mathsf{pt}}(D_1)\times_{{\mathsf{pt}}(D_0)}{\mathsf{pt}}(D_1)$ be such that $p\in {\mathsf{pt}}'(f_1)(a)$ and $q\in {\mathsf{pt}}'(f_1)(b)$. This means that $af_1^*\geq p$ and $bf_1^*\geq q$. We need to show that ${\mathsf{pt}}(m)(p,q)\in {\mathsf{pt}}'(f_1)({\mathsf{pt}}(m)(a,b))$. This is equivalent to $(p\otimes q)m^*\leq (a\otimes b)m^*f_1^*$ or to
$$
\bigvee_{(y,z)\in A}p(y)\wedge q(z)\leq \bigvee_{(y,z)\in B}a(y)\wedge b(z)
$$
for every local bisection $x\in O(D_1)$ where $A$ is the set of pairs of local bisections $(y,z)$ such that $\rho(z)=\lambda(y)$ and $yz\leq x$, and $B$ is the set of pairs of local bisections $(y,z)$ such that $\rho(z)=\lambda(y)$ and $yz\leq f_1^*(x)$. So let $x,y,z\in O(D_1)$ be local bisections such that $(y,z)\in A$ and assume that $p(y)=q(z)=1$. Since $f_1^*$ preserves multiplication, 
we have  $f_1^*(y)f_1^*(z)= f_1^*(yz)\leq f_1^*(x)$. Since  $f_1^*$ preserves $\lambda$ and $\rho$, we have 
$\rho(f_1^*(z))=f_1^*(\rho(z))=f_1^*(\lambda(y))=\lambda(f_1^*(y))$
and $a(f_1^*(y))\geq p(y)=1$, $b(f_1^*(z))\geq q(z)=1$. The required inequality follows.

(RM3) Assume that ${\mathsf{pt}}(d)(p)={\mathsf{pt}}(d)(q)$ and $c\in {\mathsf{pt}}'(f_1)(p)\cap {\mathsf{pt}}'(f_1)(q)$. The latter means that 
$pf_1^*, qf_1^*\geq c$. Let $a$ be a local bisection such that $c\in X_a$ and put $b=f_1^*(a)$.
Then $pf_1^*(a)=qf_1^*(a)=1$ as $c(a)=1$. Hence $b$ is a local bisection such that $p,q\in X_b$ and ${\mathsf{pt}}(d)(p)={\mathsf{pt}}(d)(q)$. Since the map ${\mathsf{pt}}(d)|_{X_b}$ is injective, this yields $p=q$. A dual result also follows if $d$ is replaced by $r$.

(RM4) $p={\mathsf{pt}}(f_0)(q)$ and ${\mathsf{pt}}(d)(s)=p$. We need to show that there is $t\in {\mathsf{pt}}(C_1)$ such that ${\mathsf{pt}}(d)(t)=q$ and $s\in {\mathsf{pt}}'(f_1)(t)$. Let $a$ be a local bisection such that $s\in X_a$ and let $b=f_1^*(a)$. Since $p\in X_{d_!(a)}$ we have
$$
q(d_!(b))=qd_!f_1^*(a)=qf_0^*d_!(a)=p(d_!(a))=1.
$$
Put $t=({\mathsf{pt}}(d)|_{X_b})^{-1}(q)$. Then ${\mathsf{pt}}(d)(t)=q$ and also $s=({\mathsf{pt}}(d)|_{X_a})^{-1}({\mathsf{pt}}(f_0)({\mathsf{pt}}(d)(t))$
implying that $tf_1^*\geq s$.

(RM5) Let $X\in O({\mathsf{pt}}(D_1))$. Without loss of generality we may assume that $X=X_a$ for some local bisection $a$.
We show that $${\mathsf{pt}}'(f_1)^{-1}(X_a)=\{p\in {\mathsf{pt}}(C_1)\colon {\mathsf{pt}}'(f_1)(p)\cap X_a\neq \varnothing\}= X_{f_1^*(a)}.$$
Let $c\in {\mathsf{pt}}'(f_1)(p)\cap X_a$. Then $c(a)=1$ and $pf_1^*\geq c$. Hence $pf_1^*(a)=1$ and thus $p\in X_{f_1^*(a)}$.
Conversely, assume that $p\in X_{f_1^*(a)}$. Then $f_1^*(a)\neq 0$ and by  Lemma~\ref{lem:l23} we have that $p_a\in {\mathsf{pt}}'(f_1)(p)$.

(RM6) Let $s={\mathsf{pt}}(u)(t)$ and show that ${\mathsf{pt}}(u){\mathsf{pt}}(f_0)(t)\in {\mathsf{pt}}'(f_1)(s)$. That is, we need to show that 
$tu^*f_1^*\geq tf_0^*u^*$. It is enough to show that $u^*f_1^*\geq f_0^*u^*$. By adjointness, this is equivalent to
$f_1^*\geq u_!f_0^*u^*$. Applying $u_!f_0^*=f_1^*u_!$ we get the inequality $f_1^*\geq f_1^*u_!u^*$ which holds since $id\geq u_!u^*$.
 \end{proof}


In the following lemma, we give two important consequences of the definition of a relational covering morphism.

\begin{lemma}\label{lem:conseq}
Let $f=(f_1,f_0):C\to D$ be a relational covering morphism between \'etale topological categories. Then
\begin{enumerate}
\item \label{i:cons1} $df_1^{-1}(A)=f_0^{-1}d(A)$ and dually for any $A\in O(D_1)$.
\item \label{i:cons2} $f_1^{-1}u(A)=uf_0^{-1}(A)$ for any $A\in O(D_0)$.
\end{enumerate}
\end{lemma}

\begin{proof}
\eqref{i:cons1} Since we are in an \'etale category and the maps preserve joins, we may assume that $A$ is a local bisection. Let $x\in df_1^{-1}(A)$. Then $x=d(y)$ where $f_1(y)\cap A\neq \varnothing$. Applying (RM1) it follows that $|f_1(y)\cap A|=1$. Let $f_1(y)\cap A=\{a\}$. By (RM1) we have $d(a)=f_0(x)$. Thus $x\in f_0^{-1}d(A)$ so that we have proved that $df_1^{-1}(A)\subseteq f_0^{-1}d(A)$. 

Conversely, assume that $x\in f_0^{-1}d(A)$. Then $f_0(x)\in d(A)$. Let $a$ be the only element in $A$ with $d(a)=f_0(x)$. By (RM4) we have that there is $b\in C_1$ with $d(b)=x$ such that $a=f_1(b)$. It follows that $x\in df_1^{-1}(A)$, and we have proved the inclusion $f_0^{-1}d(A)\subseteq df_1^{-1}(A)$. A dual equality for $r$ follows by symmetry.

\eqref{i:cons2} We show that $u(A)$ is a local bisection. Let $x,y\in u(A)$ and $d(x)=d(y)$. Assume that $x=u(s)$ and $y=u(t)$. Applying $du=id$ we obtain $$s=du(s)=d(x)=d(y)=du(t)=t$$ which yields $x=y$, as required. Let $x\in f_1^{-1}u(A)$. Since $u(A)$ is a local bisection, $|f_1(x)\cap u(A)|=1$ and let $f_1(x)\cap u(A)=\{a\}$. Put $y=ud(x)$. Then $uf_0d(x)\in f_1(y)$ by (RM6). On the other hand, $f_0d(x)=d(a)$ by (RM1). Hence $uf_0d(x)=ud(a)$. But $ud(a)=a$ since $a\in u(A)$ so that we obtain  $uf_0d(x)=a$. It follows that $a\in f_1(x)\cap f_1(y)$. But $d(x)=d(y)$, so that (RM3) yields $x=y$. Hence $x=ud(x)$ and thus $x\in uf_0^{-1}(A)$, so that we have proved the inclusion $f_1^{-1}u(A)\subseteq uf_0^{-1}(A)$. 

Conversely, let $x\in uf_0^{-1}(A)$. Then $x=u(t)$ where $f_0(t)\in A$. Since $du=id$ we have $t=d(x)$. Hence $uf_0(t)\in f_1(x)$ by (RM6), and so $x\in f_1^{-1}u(A)$. We have therefore proved that $uf_0^{-1}(A)\subseteq f_1^{-1}u(A)$.
\end{proof}

\begin{proposition}\label{prop:mor27} Let $f=(f_1,f_0):C\to D$ be a relational covering morphism between \'etale topological categories.  Then $f_1^{-1}$ is a morphism between restriction quantal frames from ${\mathcal O}(\overline{\Omega}(D))$ to ${\mathcal O}(\overline{\Omega}(C))$. The assignment $$f\mapsto \overline{\Omega}(f)=(f_1^{-1})^{op}$$ is functorial, where $(f_1^{-1})^{op}$ is the opposite map to $f_1^{-1}$.
\end{proposition}

\begin{proof} By Proposition \ref{prop:morphisms} it is enough to show that axioms (M1)--(M4) hold.
(M1)~holds. Let $A\in \Omega(D)$ be a local bisection. Then $f_1^{-1}(A)$ is open by (RM5) and so a local bisection by (RM3). (M2) and (M3) follow from  Lemma \ref{lem:conseq}. (M4)~holds. In view of Lemma \ref{lem:march15} 
this is equivalent to $f_1^{-1}(A)f_1^{-1}(B)\subseteq f_1^{-1}(AB)$ where $A,B\in \Omega(D_1)$ are local bisections. The latter inequality easily follows from (RM2). Functoriality is straightforward to verify.
\end{proof}

A relational covering morphism $f=(f_1,f_0):C\to D$ between \'etale topological categories is {\em at least single-valued} if $|f_1(c)|\geq 1$ for each $c\in C_1$. It is {\em at most single-valued} if $|f_1(c)|\leq 1$ for each $c\in C_1$. Finally, it is {\em single-valued} provided that $|f_1(c)|=1$ for each $c\in C_1$.

\begin{lemma}\label{lem:mor26} \mbox{}
\begin{enumerate}
\item \label{i:mor26a} Let $C=(C_1, C_0)$ and $D=(D_1,D_0)$ be \'etale localic categories and let
$f_1^*\colon {\mathcal{O}}(D)\to {\mathcal{O}}(C)$  be a proper morphism (resp. a $\wedge$-morphims) of restriction quantal frames. Then ${\mathsf{Pt}}(f_1):{\mathsf{Pt}}(C)\to {\mathsf{Pt}}(D)$ is a at least single-valued (resp. at most single-valued).
\item \label{i:mor26b} Let $f=(f_1,f_0)\colon C\to D$ be at least single-valued  (resp. at most single-valued) relational covering morphism between \'etale topological categories. Then $f_1^{-1}\colon {\mathcal O}(\overline{\Omega}(D))\to {\mathcal O}(\overline{\Omega}(C))$ is a proper morphism (resp. a $\wedge$-mor\-phism) of restriction quantal frames.
\end{enumerate}
\end{lemma}

\begin{proof}
\eqref{i:mor26a} Let
$f_1^*$  be a proper morphism and show that for every $q\in {\mathsf{pt}}(C_1)$ there is some local bisection $c\in O(D_1)$ such that $qf_1^*(c)\neq 0$. We have $f_1^*(1_{\mathcal{O}(D)})=1_{\mathcal{O}(C)}$ by assumption and $q(1_{\mathcal{O}(C)})=1$ as $q$ is a frame map, so that $qf_1^*(1_{\mathcal{O}(D)})=1$. The existence of an element $c$ with the required property now follows by \'etaleness and since  $qf_1^*$ is a sup-map.

Let $f_1^*$  be a  $\wedge$-morphism. Assume that $a\in {\mathsf{pt}}(C_1)$ and $p,q\in {\mathsf{pt}}(D_1)$ are such that
$p,q\in {\mathsf{pt}}'(f_1)(a)$. Let $p\in X_b$, $q\in X_c$ where $b,c$ are local bisections. By (RM1), we have ${\mathsf{pt}}(f_0){\mathsf{pt}}(d)(a)={\mathsf{pt}}(d)(p)$. Since $f_1^*(b\wedge c) =f_1^*(b)\wedge f_1^*(c)$ it follows that $af_1^*(b\wedge c)=1$. But then, due to Lemma \ref{lem:l23}, ${\mathsf{pt}}(f_0){\mathsf{pt}}(d)(a)\in X_{d_!(b\wedge c)}$. It follows that ${\mathsf{pt}}(d)(p)\in X_{d_!(b\wedge c)}$ and thus $p\in X_{a\wedge b}=X_a\cap X_b$. This is a contradiction as $p\not\in X_b$ (the latter is because $q\in X_b$ and $X_b$ is a local bisection).

\eqref{i:mor26b} Let $f\colon C\to D$ be at least single-valued and show that $f_1^{-1}(D_1)=C_1$. Let $x\in C_1$. By assumption $f_1(x)\neq \varnothing$ and thus
$f_1^{-1}(D_1)=\{x\colon f_1(x)\cap D_1 \neq \varnothing\}=C_1$, as required.

Let $f\colon C\to D$ be at most single-valued and let $b,c\in \Omega(D_1)$ be local bisections. We show that $f_1^{-1}(b\cap c)=f_1^{-1}(b)\cap f_1^{-1}(c)$. As $b\cap c\subseteq b,c$ we have that $f_1^{-1}(b\cap c)\subseteq f^{-1}(b)\cap f^{-1}(c)$. Assume that there is some $a\in C_1$ such that $a\in f^{-1}(b)\cap f^{-1}(c)$ but
$a\not\in f_1^{-1}(b\cap c)$. Then $f_1(a)\cap b\neq \varnothing$, $f_1(a)\cap c\neq \varnothing$ and $f_1(a)\cap (b\cap c)=\varnothing$. Let $p$ and $q$ be the only elements of $f_1(a)\cap b$ and $f_1(a)\cap c$, respectively. We have $d(p)=d(q)$ by (RM1) but on the other hand $p\neq q$ as $p\not\in b\cap c$. This contradicts our assumption that $f_1$ is at most single-valued.
\end{proof}

A functor $f=(f_1,f_0)\colon C\to D$ between \'etale topological categories is said to be $d$-{\em injective} if $f_1(a)=f_1(b)$ and $d(a)=d(b)$ imply that $a=b$. It is said to be $d$-{\em surjective} if $f_0(a)=b$ and $d(q)=b$ imply that there is $p\in C_1$ such that $d(p)=a$ and $f_1(p)=q$. $f$ is said to be  $d$-{\em bijective} if it is both $d$-injective and $d$-surjective. We make dual definitions involving $r$. A {\em covering functor} is one which is both $d$-bijective and $r$-bijective.
The functor $f$ is called {\em continuous} if the map $f_1$ is continuous (which yields that also $f_0$ is continuous).
It follows from the definitions that $f$ is a continuous covering functor if and only if it is a single-valued relational covering morphism.

From Lemma \ref{lem:mor26} we obtain the following corollary.

\begin{corollary} \label{cor:mor26m}\mbox{}
\begin{enumerate}
\item  \label{i:mor8a} Let $C=(C_1, C_0)$ and $D=(D_1,D_0)$ be \'etale localic categories and let
$f_1^*\colon {\mathcal{O}}(D)\to {\mathcal{O}}(C)$  be a proper $\wedge$-morphism of restriction quantal frames.
Then ${\mathsf{Pt}}(f_1):{\mathsf{Pt}}(C)\to {\mathsf{Pt}}(D)$ is a continuous covering functor. 
\item \label{i:mor8b} Let $f\colon C\to D$ be a continuous covering functor between \'etale topological categories. Then $f_1^{-1}\colon {\mathcal O}(\overline{\Omega}(D))\to {\mathcal O}(\overline{\Omega}(C))$ is a proper $\wedge$-morphism of restriction quantal frames.
\end{enumerate}
\end{corollary}

We define the following types of morphisms between \'etale topological categories:
\begin{itemize}
\item Type $1$: relational covering morphisms.
\item Type $2$: at least single-valued  relational covering morphisms. 
\item Type $3$: at most single-valued  relational covering morphisms.
\item Type $4$: single-valued relational covering morphisms, or, equivalently, continuous covering functors.
\end{itemize}

\subsection{Adjunction Theorems}
Let $C$ and $D$  be \'etale localic categories and let $f_1^*\colon {\mathcal O}(D)\to {\mathcal O}(C)$ be a morphism of restriction quantal frames. For each $k=1,2,3,4$ we say that $f_1\colon C\to D$ is a morphism of type $k$ if $f_1^*$ is a morphism of type $k$ of restriction quantal frames. For each $k=1,2,3,4$ by ${\mathsf{EL}}_k$ we denote the
category of \'etale localic categories and their morphisms of type $k$ by and by ${\mathsf{ET}}_k$ we denote  the category of \'etale topological categories and their morphisms of type~$k$. We also write ${\mathsf{EL}}$ for ${\mathsf{EL}}_1$ and ${\mathsf{ET}}$ for ${\mathsf{ET}}_1$. Let further ${\mathsf{Pt}}_k$ and $\overline{\Omega}_k$ denote the restrictions of the functors ${\mathsf{Pt}}$ and $\overline{\Omega}$ to morphisms of type $k$ (in particular we have ${\mathsf{Pt}}_1={\mathsf{Pt}}$ and $\overline{\Omega}_1=\overline{\Omega}$).

\begin{theorem}[Adjunction Theorem I] \label{th:adj1}
For each $k=1,2,3,4$ the functor 
$${\mathsf{Pt}}_k:{\mathsf{EL}}_k\to {\mathsf{ET}}_k$$
 is a right adjoint of the functor $$
 \overline{\Omega}_k:{\mathsf{ET}}_k\to {\mathsf{EL}}_k.$$
  The component $\eta_C=(\eta_{C,1},\eta_{C,0})$  of the unit $\eta$ of each of the adjunctions is single-valued and $\eta_{C,i}(a)(B)=1$ if and only if $a\in B$, $i=0,1$. 
 \end{theorem}

\begin{proof}
We emphasize that according to our convention the only element of the set $\eta_{C,1}(a)$ is denoted by $\eta_{C,1}(a)$. We first show that $\eta_C$ is a continuous covering functor. Since $\eta_{C,i}^{-1}(X_B)=B$, the maps $\eta_{C,i}$ are continuous. It is straightforward to verify that $\eta_C$ it is a functor. We show that it is $d$-injective. Let $a,b\in C_1$ be such that $d(a)=d(b)$ and $\eta_{C,1}(a)=\eta_{C,1}(b)$. Let $B\in \Omega(C_1)$ be a local bisection such that $\eta_{C,1}(a)\in X_B$. This means that $a,b\in B$ and together with the fact that $B$ is a local bisection and $d(a)=d(b)$ implies that $a=b$. We show that it is $d$-surjective. Let $b=\eta_{C,0}(a)$ and ${\mathsf{pt}}(d)(q)=b$.
Let $B$ be any local bisection of $C$ such that $q(B)=1$. This means that $q\in X_B$ and hence $b\in X_{d(B)}$. As  $\eta_{C,0}^{-1}(X_{d(B)})=d(B)$ we obtain that $a\in d(B)$. Let $p$ be the only element in $B$ such that $d(p)=a$ and let $A$ be a local bisection of $C$. We have that $\eta_{C,1}(p)(A)=1$ if and only if $p\in A$ if and only if $p\in A\cap B$ if and only if $a\in d(A\cap B)$ if and only if $b\in X_{d(A\cap B)}$ if and only if $q\in X_{A\cap B}$ if and only if $q\in X_A$. Therefore, $q=\eta_{C,1}(p)$ and $d$-bijectivity of $\eta_C$ is established. By symmetry we also obtain that $\eta_C$ is $r$-bijective, and hence a covering functor.

We first treat the case of the functors ${\mathsf{Pt}}$ and $\overline{\Omega}$. To verify naturality of $\eta$, we need to show that for any \'etale topological categories $C,D$ and any morphism $f:C\to D$ we have that $\eta_Df={\mathsf{Pt}}\, \overline{\Omega}(f)\eta_C$. The equality in the second component follows from Theorem \ref{th:pr1}. For the first component the equality we need quickly reduces to 
$$\bigvee_{q\in f_1(c)}\eta_{D,1}q(B)=\eta_{C,1}(c)f_1^{-1}(B),$$
 where $c\in C_1$ and $B$ is a local bisection of $D_1$. This equality holds since either of its sides equals $1$ if and only if $f_1(c)\cap B\neq\varnothing$.

Let $C\in {\mathrm{Ob}}({\mathsf{ET}})$, $D\in {\mathrm{Ob}}({\mathsf{EL}})$ and $f=(f_1,f_0):C\to {\mathsf{Pt}}(D)$ be a relational covering morphism.  We define $g^*_1\colon  {\mathcal O}(D) \to {\mathcal O}(\overline{\Omega}(C))$ by $$g_1^*(B)=f_1^{-1}(X_B),\, B\in O(D_1).$$ 

We put $g_0^*=dg_1^*{\mathsf{pt}}(u)$. Then for any $B\in O(D_0)$ we have 
$$
g_0^*(B)=df_1^{-1}(X_{u_!(B)})=f_0^{-1}(X_{d_!u_!(B)})=f_0^{-1}(X_B).
$$
To verify that $g_1^*$ is a morphism of restriction quantal frames we apply Proposition~\ref{prop:morphisms}. The map $g_1^*$ is a sup-map and $g_0^*$ is a frame map by definition. We verify that axioms (M1)--(M4) are satisfied.

(M1) Let $B\in O(D_1)$ and show that $f_1^{-1}(X_B)$ is a local bisection. By part (4) of Lemma \ref{lem:homeom21}
$X_B$ is a local bisection, and thus so is  $f_1^{-1}(X_B)$ by (RM5) and (RM3). 
(M2) reduces to the equality $df_1^{-1}(X_B)=f_0^{-1}{\mathsf{pt}}(d)(X_B)$, where $B\in O(D_1)$, which follows by part \eqref{i:cons1} of Lemma \ref{lem:conseq}. 
(M3) reduces to $f_1^{-1}{\mathsf{pt}}(u)(X_{A})=uf_0^{-1}(X_A)$, where $A\in O(D_0)$, which follows by part \eqref{i:cons2} of Lemma \ref{lem:conseq}. 
(M4) In view of Lemma~\ref{lem:march15}~this is equivalent to $f_1^{-1}(X_A)f_1^{-1}(X_B)\subseteq f_1^{-1}(X_AX_B)$, where $A$ and $B$ are local bisections. This easily follows from (RM2).

We now verify that
$f={\mathsf{Pt}}(g)\eta_C$. The equality $f_0={\mathsf{pt}}(g_0)\eta_{C,0}$ holds by Theorem~\ref{th:pr1}. We are left to verify that $f_1={\mathsf{pt}}'(g_1)\eta_{C,1}$.

Let $a\in C_1$. We need to show that $f_1(a)={\mathsf{pt}}'(g_1)\eta_{C,1}(a)$. We first observe that 
\begin{equation}\label{eq:aux26a}
f_1(a)\cap X_B\neq \varnothing  \,\, \text{ if and only if }\,\, \eta_{C,1}(a)g_1^*(B)=1
\end{equation}
for any $B\in O(D_1)$. Indeed,  $f_1(a)\cap X_B\neq \varnothing$ holds if and only if $a\in f_1^{-1}(X_B)$ which is equivalent to $a\in g_1^*(B)$, which, in turn, is equivalent to $\eta_{C,1}(a)g_1^*(B)=1$, as required. 
We also observe that 
\begin{equation}\label{eq:aux26}
f_0d(a)={\mathsf{pt}}(g_0){\mathsf{pt}}(d)\eta_{C,1}(a).
\end{equation}
Indeed by (RM1) we have the equality $\eta_{C,0}d(a)={\mathsf{pt}}(d)\eta_{C,1}(a)$ which reduces the needed equality to 
$f_0d(a)={\mathsf{pt}}(g_0)\eta_{C,0}d(a)$ which holds since $f_0={\mathsf{pt}}(g_0)\eta_{C,0}$.

Let $p\in f_1(a)$ and show that $p\in {\mathsf{pt}}'(g_1)\eta_{C,1}(a)$. Let $B\in O(D_1)$ be a local bisection such that $p\in X_B$.
It follows from \eqref{eq:aux26a} that $\eta_{C,1}(a)g_1^*(B)=1$ and thus 
$$
p_B=({\mathsf{pt}}(d)|_{X_B})^{-1}{\mathsf{pt}}(g_0){\mathsf{pt}}(d)\eta_{C,1}(a)\in {\mathsf{pt}}'(g_1)\eta_{C,1}(a).
$$
Note that $p_B$ may be characterized as the only element $q\in X_B$ satisfying the equality
${\mathsf{pt}}(d)(q)={\mathsf{pt}}(g_0){\mathsf{pt}}(d)\eta_{C,1}(a)$. Using \eqref{eq:aux26} we obtain that
$p_B$ is the only element $q\in X_B$ such that ${\mathsf{pt}}(d)(q)=f_0d(a)$. But ${\mathsf{pt}}(d)(p)=f_0d(a)$ by (RM1). It follows that $p_B=p$ and so $p\in {\mathsf{pt}}'(g_1)\eta_{C,1}(a)$.

Let $p\in {\mathsf{pt}}'(g_1)\eta_{C,1}(a)$ and show that $p\in f_1(a)$. By definition we have that $p=p_B$ for some $B\in O(D_1)$ such that $\eta_{C,1}(a)g_1^*(B)=1$. As above we observe that $p_B$ is the only element $q\in X_B$ such that
${\mathsf{pt}}(d)(q)=f_0d(a)$. By \eqref{eq:aux26a} there is $q\in f_1(a)\cap X_B$. Then ${\mathsf{pt}}(d)(q)=f_0d(a)$ by (RM1). It follows that $q=p_B=p$ and thus $p\in f_1(a)$. This completes the proof of the claim that the functor ${\mathsf{Pt}}$ is a right adjoint of the functor $\overline{\Omega}$.

For the remaining functors, we follow the lines of the proof of part \eqref{i:mor26b} of Lemma~\ref{lem:mor26} and obtain that if $f=(f_1,f_0):C\to {\mathsf{Pt}}(D)$ is at least (resp. at most) single-valued then $g_1^*$ is a proper morphism (resp. a $\wedge$-morphism). This and  Lemma \ref{lem:mor26} imply the needed statements.
\end{proof}

The Adjunction Theorem, together with Theorems \ref{th:jul17} and \ref{the:jan23} yield the following result.
\begin{theorem}[Adjunction Theorem II] \label{th:adj8a}
For each $k=1,2,3,4$ there is a dual adjunction between the category of complete restriction monoids and their morphisms of type $k$ and the category of \'etale topological categories and their morphisms of type $k$. 
\end{theorem}

In view of Propositions \ref{prop:ample} and \ref{prop:ample1} we also obtain the following result.

\begin{corollary}\label{cor:adj8b} For each $k=1,2,3,4$ there is a dual adjunction between the category of left ample (resp. right ample, ample) complete restriction monoids and their morphisms of type $k$ and the category of left cancellative (resp. right cancellative, cancellative) \'etale topological categories and morphisms of type $k$. 
\end{corollary}

We spell out the explicit constructions of  functors establishing the adjunctions in Theorem \ref{th:adj8a} and Corollary~\ref{cor:adj8b}. Consider, for example, the adjunctions between the category of complete restriction monoids and the category of \'etale topological categories. For objects, we have the following constructions. Let $C$ be an \'etale topological category. Then it is mapped to the complete restriction monoid  ${\mathcal{PI}}({\mathcal O}(\overline{\Omega}(C)))$. Conversely, let $S$ be a comlete restriction monoid. Then it is mapped to the \'etale topological category ${\mathsf{Pt}}(\mathcal C(\mathcal{L}^{\vee}(S)))$. Similarly, one can keep track of the maps of morphisms.

\subsection{The involutive and groupoid settings}
All established adjunction theorems can be readily  extended to the involutive and groupoid settings. 
For completeness, we include some details.

Let $C$ be an involutive \'etale localic category and let $i$ be the involution structure map. Then it is a direct consequence of the functoriality of the assignment ${\mathsf{pt}}$ that ${\mathsf{pt}}(i)$ is an involution of the \'etale topological category ${\mathsf{Pt}}(C)$. A similar remark applies also in the reverse direction. We now treat morphisms. A relational covering morphism $f=(f_1,f_0):C\to D$ between involutive \'etale topological categories is called {\em involutive} if 
for every $a\in C_1$ we have that $f_1(i(a))=i(f_1(a))$.

\begin{lemma}\mbox{}
\begin{enumerate}
\item \label{i:inv8a} Let $C$ and $D$ be involutive \'etale localic categories and let $f_1^*\colon {\mathcal O}(D)\to {\mathcal O}(C)$ be an involutive morphism of restriction quantal frames. Then ${\mathsf{Pt}}(f_1)$ is an involutive morphism of \'etale topological categories. 
\item \label{i:inv8b} Let $C$ and $D$ be involutive \'etale topological categories and $f\colon C\to D$ an involutive relational covering morphism. Then $f_1^{-1}$ is an involutive morphism of restriction quantal frames.
\end{enumerate}
\end{lemma}

\begin{proof} \eqref{i:inv8a} To verify that ${\mathsf{pt}}'(f_1){\mathsf{pt}}(i)(a)={\mathsf{pt}}(i){\mathsf{pt}}'(f_1)(a)$ we need to verify that $f_1^*i^*a=i^*f_1^*a$, where $a\in {\mathsf{pt}}(C_1)$. But by part \eqref{i:oct212} of Proposition \ref{prop:o24} we have that $i^*=i_!$ is precisely the involution on a restriction quantal frame. It follows that the required equality holds as $f_1^*$ is an involutive morphism.

\eqref{i:inv8b} Since $i=i^{-1}$ we need to verify that $f_1^{-1}i(A)=if_1^{-1}(A)$ for any $A\in \Omega(D_1)$.
We have that $a\in f_1^{-1}i(A)$ holds if and only if when $f_1(a)\cap i(A)\neq \varnothing$. This is equivalent to
$i(f_1(a))\cap A\neq\varnothing$, which, applyint $if_1=f_1i$, is equivalent to $f_1(i(a))\cap A\neq\varnothing$.
This, in turn, is equivalent to $i(a)\in f_1^{-1}(A)$, or $a\in if_1^{-1}(A)$.
\end{proof}

To obtain the involutive analogue of Theorem \ref{th:adj1}, we need to verify that the map $g_1^*$ from the proof of Theorem \ref{th:adj1} is involutive whenever $f$ is, and also that $\eta_C$ is an involutive morphism. But this is straightforward to verify. 

In the case where involutive categoreis are groupoids,  it is immediate by functoriality of the assignments ${\mathsf{pt}}$ and $\Omega$ that the assignments ${\mathsf{Pt}}$ and $\overline{\Omega}$ map groupoids to groupoids.  Also, an involutive morphism of groupoids is just a groupoid morphism and an involutive morphism of  pseudogroups is just a pseudogroup morphism. In the next and some further statements by a morphism of an involutive object we always understand an involutive moprhism. 

\begin{corollary} \label{cor:adj:inv}\mbox{}
\begin{enumerate}
\item For each $k=1,2,3,4$ there is an adjunction between the category of involutive \'etale localic categories (resp. \'etale localic groupoids) and their morphisms of type $k$ and the category of involutive \'etale topological categories (resp. \'etale topological groupoids) and their morphisms of type $k$.
\item 
For each $k=1,2,3,4$ there is a dual adjunction between the category of involutive complete restriction monoids (resp. pseudogroups) and their morphisms of type $k$ and the category of \'etale topological categories (resp. \'etale topological groupoids) and their  morphisms of type $k$. 
\end{enumerate}
\end{corollary}

\begin{remark} {\em It is easy to see that the dual adjunction between the category of \'etale topological groupoids and their continuous covering functors (type $4$) and the category of pseudogroups and their proper $\wedge$-morphisms (type $4$) given by Corollary~\ref{cor:adj:inv} is precisely the one given in \cite[Theorem 2.22]{LL1}. What we have called {\em callitic morphisms} in \cite{LL1} are in fact proper $\wedge$-morphisms. Our current work therefore extends and clarifies that to be found in \cite{LL1}.}
\end{remark}
 
\section{Non-commutative Stone dualities}\label{s:dualities}

\subsection{Sober, spectral and Boolean categories}
An \'etale topological category $C=(C_1,C_0)$ is called {\em sober} if the spaces $C_0$  and $C_1$ are sober. 
An \'etale localic category $C=(C_1,C_0)$ is called {\em spatial} if the locales $C_0$ and $C_1$ are spatial. 
A complete restriction monoid is called {\em spatial} if the frames $e^{\downarrow}$ and ${\mathcal L}^{\vee}(S)$ are spatial. From these definitions and the results of Section~\ref{s:adjun} we immediately deduce the following.

\begin{theorem}\label{th:adjg} {\em [Topological duality theorem I] }For each $k=1,2,3,4$ the following categories are equivalent:
\begin{enumerate}
\item The category of  spatial \'etale localic categories (resp. involutive categories, groupoids) and their morphisms of type $k$;
\item The category of sober \'etale topological categories (resp. involutive categories, groupoids) and their morphisms of type $k$;
\item The opposite of the category of spatial complete restriction monoids (resp. involutive complete restriction monoids, pseudogroups) and their morphisms of type $k$.
\end{enumerate}
\end{theorem}

\begin{lemma}\label{lem:sober} Let $C=(C_1,C_0)$ be an \'etale topological category. Then the space $C_0$ is sober if and only if the space $C_1$ is sober.
\end{lemma}

\begin{proof} Assume that $C_1$ is sober. This means that the map $\eta_{C,1}$ is a bijection. We show that the map $\eta_{C,0}$ is a bijection. Let $x,y\in C_0$ be such that $x\neq y$. Since $\eta_{C,1}$ is single-valued and by (RM6) we have $u\eta_{C,0}(x)=\eta_{C,1}(x)\neq \eta_{C,1}(y)=u\eta_{C,0}(y)$. Since $u$ is a bijection, it follows that $\eta_{C,0}(x)\neq \eta_{C,0}(y)$. Let $x\in {\mathsf{pt}}\Omega(C_0)$. Then $u(x)\in {\mathsf{pt}}\Omega(C_1)$. As $\eta_{C,1}$ is surjective we have that $u(x)=\eta_{C,1}(y)$ for some $y\in C_1$. But then $x\in \eta_{C,0}(d(x))$ by (RM1). It follows that $C_0$ is sober.

Assume that $C_0$ is sober. Let $x,y\in C_1$ be such that $x\neq y$ and assume that $\eta_{C,1}(x)=\eta_{C,1}(y)$. Then $d(x)\neq d(y)$ by (RM3). Also $\eta_{C,0}d(x)=\eta_{C,0}d(y)$ by (RM1) which is a contradiction. Thus the map $\eta_{C,1}$ is injective. Let $x\in {\mathsf{pt}}\Omega(C_1)$. Since $\eta_{C,0}$ is surjective there is some $y\in C_0$ such that
$\eta_{C,0}(y)=d(x)$. By (RM4) we now obtain that there is some $z\in C_1$ with $\eta_{C,1}(z)=x$, so that $\eta_{C,1}$ is surjective. We have proved that $C_1$ is sober. 
\end{proof}

\begin{lemma}\label{lem:spatial} Let $C=(C_1,C_0)$ be an \'etale localic category. The locale $C_0$ is spatial if and only if the locale $C_1$ is spatial.
\end{lemma}

\begin{proof} Assume that the locale $C_0$ is spatial and show that the locale $C_1$ is.
We show that $X_a=X_b$ implies that $a=b$ for any $a,b\in O(C_1)$.
Assume first that $a$ and $b$ are local bisections and  $X_a=X_b$. Then $X_a=X_a\cap X_b=X_{a\wedge b}$, so that ${\mathsf{pt}}(d)(X_a)={\mathsf{pt}}(d)(X_{a\wedge b})$.
 It follows that $X_{d_!(a)}=X_{d_!(a\wedge b)}$ by part (4) of Lemma \ref{lem:homeom21} as $d$ is a local homeomorphism. Since $O(C_0)$ is a spatial frame, we have $d_!(a)=d_!(a\wedge b)$. This equality, together with part (3) of Lemma \ref{lem:homeom21} and the inequality $d^*d_!\geq id$, yields
 $$
 a\wedge b=d^*d_!(a\wedge b)\wedge a = d^*d_!(a) \wedge a =a. 
 $$
 It follows that $b\geq a$. By symmetry we also obtain $a\geq b$. Hence $a=b$.

  We now consider the general case. Assume $a,b\in O(C_1)$ are such that  $X_a=X_b$ and show that $a=b$. By symmetry, it is enough to show that $b\geq a$.
 Since the restriction quantal frame ${\mathcal O}(C)$ is \'etale, there are local bisections $a_i$, $i\in I$, such that $a=\bigvee_{i}a_i$.  Since $X_{a_i}\cap X_a=X_{a_i}$, where $i\in I$, and $X_a=X_b$ it follows that for each admissible $i$ we have
 \begin{equation}\label{eq:spatial} X_{a_i}=X_{a_i}\cap X_{b}=X_{a_i\wedge b}.
 \end{equation} 
 Local bisections in $C$ are partial isometries in ${\mathcal O}(C)$. The latter form an order ideal by Lemma \ref{le: order_ideal}.
 It follows that $a_i\wedge b$ is a local bisection. Thus from \eqref{eq:spatial} and the special case considered above we obtain 
 $a_i=a_i\wedge b$. It follows that $b\geq a_i$ for each~$i$. Consequently,
$b\geq \bigvee_i a_i=a$, as required. It follows that the map $a\mapsto X_a$ from $O(C_1)$ to $\Omega{\mathsf{pt}}(C_1)$ is injective. This map is surjective as, by definition, the sets $X_a$, $a\in A$, constitute the topology on  ${\mathsf{pt}}(C_1)$. Therefore, the locale $C_1$ is spatial.

Assume now that the locale $C_1$ is spatial and show that so is the locale $C_0$. Let $a,b\in O(C_0)$ and $X_a=X_b$. In view of \eqref{eq:ptu}, we have that 
$$
X_{u_!(a)}={\mathsf{pt}}(u)(X_a)={\mathsf{pt}}(u)(X_b)=X_{u_!(b)}.
$$
Since $C_1$ is a spatial locale, it follows that $u_!(a)=u_!(b)$. Therefore, $$a=d_!u_!(a)=d_!u_!(b)=b.$$ 
It follows that the map $a\mapsto X_a$ from $O(C_0)$ to $\Omega{\mathsf{pt}}(C_0)$ is injective. Similarly as above, we conclude that this map is also surjective. Hence the locale $C_1$ is spatial.
\end{proof}

An \'etale topological category $C=(C_1,C_0)$ will be called {\em spectral} (resp. {\em strongly spectral}) if the space $C_0$ (resp. $C_1$) is spectral. It will be called {\em Boolean} (resp. {\em strongly Boolean}) if the space $C_0$ (resp. $C_1$) is Boolean. 
An \'etale localic category $C=(C_1,C_0)$ will be called {\em coherent} (resp. {\em strongly coherent}) if the locale $C_0$ (resp.~$C_1$) is coherent.

\begin{lemma}\label{lem:spectral} A strongly spectral  (resp. strongly Boolean) \'etale topological category is spectral (resp. Boolean). \end{lemma} 

\begin{proof} Let $C=(C_1,C_0)$ be an \'etale topological category and the space $C_1$ be spectral. Let $X$ be a basis of compact-open subsets of $C_1$ closed under finite non-empty intersections. Then every local bisection of $C_1$ is a union of sets from $X$. But an open subset of a local bisection is a local bisection. So $C_1$ is a union of sets of $X$ which are compact-open bisections. As open bisections form a basis for the topology on $C_1$ it follows that we may assume that $X$ consists of compact-open local bisections. Let $A\subseteq C_0$ be an open set. Then $u(A)$ is an open local bisection. Then $u(A)=\bigvee_{i} B_i$ where $B_i\in X$ for each $i$. It follows that $A=\bigvee_{i} d(B_i)$. Hence $d(B)$, where $B$ runth through $X$, form a basis of the topology on $C_0$. Using the fact that $A\mapsto d(A)$ is a homeomorphism for a local bisection $A$, it easily follows that $d(A)$ is compact-open if and only if $A$ is. If $B_1=d(A_1)$ and $B_2=d(A_2)$ where $A_1,A_2$ are compact-open local bisections then $u(B_1)\cap u(B_2)=u(B_1\cap B_2)$ is a compact-open local bisection. Thus $B_1\cap B_2$ is compact-open. It follows that $C_0$ has a basis of compact-open sets that is closed under finite non-empty intersections. Since $C_1$ is sober, then by Lemma \ref{lem:sober} so is $C_0$. It follows that $C$ is spectral.

Let $C=(C_1,C_0)$ be an \'etale topological category and assume that the space $C_1$ is Boolean. It follows from the previous paragraph  that $C_0$ is a spectral space. We show that it is Hausdorff. Let $a,b\in C_0$ and let $A$ and $B$ be compact-open local bisections such that $u(a)\in A$, $b\in B$ and $A\cap B=\varnothing$. Then $a\in d(A)$, $b\in d(B)$, $d(A)\cap d(B)=\varnothing$ and $d(A)$ and $d(B)$ are compact-open sets. It follows that $C_0$ is Hausdorff and so it is Boolean. Thus $C$ is Boolean.
\end{proof}

\begin{corollary} \label{cor:coherent1} A strongly coherent \'etale localic category is coherent.
\end{corollary}

\begin{proof} Let $C=(C_1,C_0)$ be strongly coherent. Then by Theorem \ref{th:pr2} we have that ${\mathsf{Pt}}(C)=({\mathsf{pt}}(C_1), {\mathsf{pt}}(C_0))$ is strongly spectral. It is then spectral by Lemma \ref{lem:spectral}. Thus it is sober. Applying Theorem \ref{th:pr1a} it follows that $C$ is spectral.
\end{proof}

The following result is interesting by itself and also clarifies the connection of our work with that to be found in \cite{LL1, LL2}.

\begin{lemma}Let $C=(C_1,C_0)$ be a Boolean \'etale topological category. The following statements are equivalent:
\begin{enumerate}
\item  \label{i:1aa} $C$ is strongly Boolean;
\item \label{i:1ab} $C_1$ is a spectral space;
\item \label{i:1ac} $C_1$ is a Hausdorff space.
\end{enumerate}
\end{lemma}

\begin{proof}
The implication \eqref{i:1aa} $\Rightarrow$ \eqref{i:1ac} is trivial. The implication \eqref{i:1ac} $\Rightarrow$ \eqref{i:1ab} follows from the fact that in a Hausdorff space the intersection of two compact sets is closed which implies that the intersection of two compact-open local bisections is a compact-open local bisection. We now prove the implication \eqref{i:1ab} $\Rightarrow$ \eqref{i:1aa}. So we assume that $C_1$ is a spectral space and prove that it is Boolean. By definition, it is enough to prove that it is Hausdorff. Let $a,b\in C_1$.  If $d(a)\neq d(b)$ then $d(a)$ and $d(b)$ can be separated by compact-open  sets, $P$ and $Q$, since $C_0$ is Hausdorff. Let $A$ and $B$ be compact-open local bisections containing $a$ and $b$, respectively. Then the restrictions of $A$ and $B$ to $P\cap d(A)$  and $Q\cap d(B)$, respectively, are disjoint compact-open local bisections containing $a$ and $b$, respectively. Now assume that $d(a)=d(b)$. Let $A$ and $B$ be compact-open bisections containing $a$ and $b$, respectively. Passing to restrictions, if necessary, we may assume that $d(A)=d(B)$. By assumption $A\cap B$ is a compact-open local bisection which contains neither $a$ nor $b$. Let $C=d(A)\setminus d(A\cap B)$. This is a compact-open set as the space $C_0$ is Boolean. Then the restrictions of $A$ and $B$ to $C$ are disjoint compact-open local bisections containing $a$ and $b$, respectively. This completes the proof.
\end{proof}

\begin{lemma}\label{lem:apr16}\mbox{}
\begin{enumerate}
\item \label{i:16a}
 A spectral (or strongly spectral) \'etale topological category is sober.
\item  \label{i:16b}
A coherent (or strongly coherent) \'etale localic category is spatial
\end{enumerate}
\end{lemma}
\begin{proof} 
\eqref{i:16a} follows from Lemma \ref{lem:sober} and the fact that a spectral space is sober.
\eqref{i:16b}~follows from Lemma \ref{lem:spatial} and the fact that a coherent locale is spatial.
\end{proof}

A relational covering morphism $f=(f_1,f_0)$ between spectral  \'etale topological categories is called {\em coherent} if the map $f_0$ is coherent. 

A complete restriction monoid $S$ is called {\em coherent} (resp. {\em strongly coherent}) if the frame  $e^{\downarrow}$ (resp. ${\mathcal L}^{\vee}(S)$) is coherent. This is clearly a translation of the requirement that the category ${\mathcal{C}}(\mathcal{L}^{\vee}(S))$ is coherent (resp. strongly coherent).
A map between complete restriction monoids is called {\em coherent} if it maps finite projections to finite projections.
We thus have the following important result.

\begin{theorem}[Topological duality theorem II]\label{th:adjga} For each $k=1,2,3,4$ the category of (strongly) spectral \'etale topological categories (resp. involutive categories, groupoids) and their coherent morphisms of type $k$
is dually equivalent to the category of (strongly) coherent complete restriction monoids (resp. involutive complete restriction monoids, pseudogroups) and their coherent morphisms of type $k$. 
\end{theorem}

\subsection{The locally compact setting}

Recall that a restriction semigroup $S$ is called {\em distributive} if the semilattice of projections $E$ is a distributive lattice and if $a,b\in S$ are compatible elements then $a\vee b$ exists. It is called a {\em distributive $\wedge$-semigroup} provided that in addition $a\wedge b$ exists for any $a,b\in S$. If $S$ is a distributive restriction monoid then its identity element $e$ is a projection and $E=e^{\downarrow}$.

To work out dualities, involving distributive restriction semigroups and $\wedge$-semi\-groups at the algebraic side, we link them with coherent complete restriction monoids. Thus complete restriction monoids appear as mediators between finitaty and topological objects. 

We first concentrate on objects. Let $S$ be a coherent complete restriction monoid. An element $a\in S$ is called {\em finite}, if $a=\bigvee A$, $A\subseteq S$, implies that there is a finite subset $F\subseteq A$ such that $a=\bigvee F$. Note that $a=\bigvee A$ in a restriction semigroup implies, by Lemma~\ref{lem:aug28}, that the set $A$ is compatible. We will need the following characterization of finite elements of complete restriction monoids.

\begin{lemma}\label{lem:29m} Let $S$ be a complete restriction monoid with unit $e$ and $a\in S$. The following are equivalent:
\begin{enumerate}
\item \label{i:29m1} $a$ is a finite element;
\item \label{i:29m2} $\lambda(a)$ is a finite element of $e^{\downarrow}$;
\item \label{i:29m3} $\rho(a)$ is a finite element of $e^{\downarrow}$.
\end{enumerate}
\end{lemma}

\begin{proof} \eqref{i:29m1}$\Leftrightarrow$ \eqref{i:29m2} We use the fact that the maps $t\mapsto \lambda(t)$ and $f\mapsto af$ establish mutually inverse order-isomorphisms between the sets $a^{\downarrow}$ and $\lambda(a)^{\downarrow}$. Assume that $a$ is finite and $\lambda(a)=\bigvee A$, $A\subseteq e^{\downarrow}$. Then
$a=\bigvee_{f\in A} af$ and so there is a finite $F\subseteq A$ such that $a\bigvee_{f\in F}af$. It follows that $\lambda(a)=\bigvee_{f\in F}\lambda(f)$, so that $\lambda(a)$ is a finite element. The converse implication is established similarly. The equivalence \eqref{i:29m1}$\Leftrightarrow$ \eqref{i:29m3} follows by symmetry.
\end{proof}

Let $\mathsf{K}(S)$ denote the set of all finite elements of $S$.

\begin{proposition}\label{prop:m31}
Let $S$ be a coherent (resp. strongly coherent) complete restriction monoid. 
Then $\mathsf{K}(S)$ is a distributive restriction semigroup (resp. distributive restriction $\wedge$-semigroup) with respect to the semilattice of projections $\mathsf{K}(e^{\downarrow})$. 
\end{proposition}

\begin{proof}  We first show that $\mathsf{K}(S)$ is a semigroup. We make several applications of Lemma \ref{lem:29m}. Let $a,b\in \mathsf{K}(S)$. It is enough to prove that $\lambda(ab)\in \mathsf{K}(e^{\downarrow})$. We have $\lambda(ab)=\lambda(\lambda(a)b)$ and $\lambda(\lambda(a)b)\in \mathsf{K}(e^{\downarrow})$ is equivalent to $\lambda(a)b\in \mathsf{K}(S)$. The latter, in turn, is equivalent to $\rho(\lambda(a)b)\in \mathsf{K}(e^{\downarrow})$. This reduces to $\lambda(a)\rho(b)\in \mathsf{K}(e^{\downarrow})$ which holds as $\lambda(a)\rho(b)=\lambda(a)\wedge \rho(b)$ and $\mathsf{K}(e^{\downarrow})$ is a distributive lattice and is thus closed with respect to taking binary meets.
It is clear that $\mathsf{K}(S)$ is closed with respect to $\lambda$ and $\rho$, so that it is a restriction subsemigroup of $S$. It is a submonoid if and only if $e$ is a finite element. Since in addition $\mathsf{K}(e^{\downarrow})$ is a distributive lattice, $\mathsf{K}(S)$ is a distributive restriction semigroup.

Assume that $S$ is strongly coherent and $a,b\in  \mathsf{K}(S)$. Then $\eta(a), \eta(b)\in \mathsf{K}({\mathcal L}^{\vee}(S))$ where $\eta: s\mapsto s^{\downarrow}$ is the inclusion map of $S$ into ${\mathcal L}^{\vee}(S)$. Since $\mathsf{K}({\mathcal L}^{\vee}(S))$ is a distributive lattice we have that $x=\eta(a)\wedge \eta(b)\in \mathsf{K}({\mathcal L}^{\vee}(S))$. As $\eta(S)$ is an order ideal in ${\mathcal L}^{\vee}(S)$ it follows that $x=\eta(c)$ for some $c\in S$. Using the fact that $\eta: S\to \eta(S)$ is an order-isomorphism it now easily follows that $a\wedge b$ exists and equals $c$.
\end{proof}

Let $S$ be a distributive restriction semigroup. By $\mathsf{Idl}(S)$ we denote the set of all compatible order ideals of $S$ which are closed with respect to taking binary joins.

\begin{proposition}\label{prop:dist_coh} Let $S$ be a distributive restriction semigroup (resp. $\wedge$-semigroup).
Then $\mathsf{Idl}(S)$ is a coherent (resp. strongly coherent) complete restriction monoid with the frame of projections $\mathsf{Idl}(E)$.
\end{proposition}

\begin{proof} We first show that $\mathsf{Idl}(S)$ is closed under multiplication. Let $A,B\in \mathsf{Idl}(S)$ and assume that $a\in A$, $b\in B$ and $c\in S$ are such that $c\leq ab$. Then $c=ab\lambda(c)$. Since $B$ is an order ideal, $b\lambda(c)\in B$ and hence $c\in AB$. We have proved that $AB$ is an order ideal. Since by Lemma~\ref{lem:aug27} the compatibility relation on $S$ is compatible with multiplicaition, it follows that $AB$ is compatible. To show that $ab\vee cd\in AB$ we observe that by Lemma~\ref{lem:molly}
we have $(a\vee c)(b\vee d)=(ab\vee cb)\vee (ad\vee cd)\geq ab\vee cd$. 

Note that $\mathsf{Idl}(E(S))\subseteq \mathsf{Idl}(S)$. For  $A\in \mathsf{Idl}(S)$ we put 
$$\lambda(A)=\{\lambda(a)\colon a\in A\} \text{ and } \rho(A)=\{\rho(a)\colon a\in A\}.
$$
It is easy to check that $\lambda(A),\rho(A)\in \mathsf{Idl}(E)$. It is easy to verify that all axioms of a restriction semigroup are satisfied and it is clear that $E(S)$ is a unit element. Hence $\mathsf{Idl}(S)$ is a restriction monoid with respect to $\mathsf{Idl}(E)$. We show that joins of any compatible families exist in $\mathsf{Idl}(S)$. First, it is easy to verify that $A\sim B$ in $\mathsf{Idl}(S)$ if and only if $a\sim b$ for any $a\in A$ and $b\in B$. Let $X$ be a compatible family of elements. Then it is immediate that the set 
$$
A=\{a_1\vee \dots \vee a_n \colon n\geq 1, a_k\in \bigcup X\text{ for each }k\} 
$$
is the join of the family $X$. It follows that $\mathsf{Idl}(S)$ is a complete restriction monoid.
$\mathsf{Idl}(S)$ is coherent as $\mathsf{Idl}(E)$ is a coherent frame. 

Assume that $S$ is a $\wedge$-semigroup. Then the space $C_1$ of the spectral \'etale topological category $C={\mathsf{Pt}}({\mathcal C}({\mathcal L}^{\vee}(\mathsf{Idl}(S)))$ has a basis of compact-open local bisections which is closed under finite non-empty intersections. It follows that its frame of opens, which is isomorphic to ${\mathcal L}^{\vee}(\mathsf{Idl}(S))$, is coherent.
\end{proof}

\begin{lemma}\label{lem:compact31} let $S$ be a distributive restriction semigroup. Then $$\mathsf{K}(\mathsf{Idl}(S))=\{s^{\downarrow}\colon s\in S\}.$$ 
\end{lemma}

\begin{proof}
Show that $s^{\downarrow}$ is a finite element of  $\mathsf{Idl}(S)$ for each $s\in S$.  Note that $s^{\downarrow}=\bigvee T$ is equivalent to the fact that $s$ is expressible as a finite join of members of members of $T$. This is equivalent to $s^{\downarrow}=\bigvee F$ for some finite $F\subseteq T$. Every element of $\mathsf{Idl}(S)$ is trivially expressible as a join of principal ideals and so every element is a join of finite elements. It follows that any finite element can be expressed as a finite join of the elements of the form $s^{\downarrow}$.
Since for $s,t$ compatible we have $s^{\downarrow}\vee t^{\downarrow}=(s\vee t)^{\downarrow}$, we obtain that any finite element of $\mathsf{Idl}(S)$ coincides with some $s^{\downarrow}$. 
\end{proof}

We now connect coherent morphisms  between coherent complete restriction monoids with morphisms between distributive restriction semigroups. By a {\em morphism} $f:S\to T$ between distributive restriction semigroups we understand a  homomorphism of restriction semigroups whose restriction to the lattice of projections $E$ is a morphism of distributive lattices.

\begin{lemma}\label{lem:morphisms31}\mbox{}
\begin{enumerate}
\item \label{i:mor31a} Let $f\colon S\to T$ be a coherent morphism of coherent complete restriction monoids. Then $f$ takes finite elements to finite elements and $$
\mathsf{K}(f)=f|_{\mathsf{K}(S)}\colon \mathsf{K}(S)\to \mathsf{K}(T)
$$ is a morphism of distributive restriction semigroups.
\item \label{i:mor31b} Let $f\colon S\to T$ be a morphism of distributive restriction semigroups. For $A\in \mathsf{Idl}(S)$ we set
$\overline{f}(A)=\bigvee_{a\in A}f(a)^{\downarrow}$. Then $$\mathsf{Idl}(f)=\overline{f}\colon \mathsf{Idl}(S)\to \mathsf{Idl}(T)$$ is a coherent morphism of coherent complete restriction monoids.
\end{enumerate}
\end{lemma}

\begin{proof}
\eqref{i:mor31a} Let $a\in S$ be a finite element. Then $\lambda(a)$ is finite by Lemma \ref{lem:29m}. Hence, by definition, $f(\lambda(a))$ is finite. But $f$ commutes with $\lambda$, so that $\lambda(f(a))$ is finite. Then by Lemma  \ref{lem:29m} we have that $f(a)$ is finite. The statement follows. 

\eqref{i:mor31b} Show that $\overline{f}(A)\in \mathsf{Idl}(T)$. Clearly $\overline{f}(A)$ is an order ideal. If $a\sim b$ and $x\leq f(a)$, $y\leq f(b)$ then $x\sim y$. It follows that $\overline{f}(A)$ is compatible. If $x,y\in \overline{f}(A)$ then $x\leq f(a)$ and $y\leq f(b)$ for some $a,b\in A$ and so $x\vee y\leq f(a\wedge b)\in \overline{f}(A)$. It is routine to verify that $\overline{f}$ preserves the unary operations. If $z\leq f(ab)=f(a)f(b)$ then $z=f(a)f(b)\lambda(z)$ so that $z=xy$ with $x\leq a$ and $y\leq b$. It now easily follows that $\overline{f}$ preserves multiplication. By Theorem~\ref{th:pr2}  its restriction to $\mathsf{Idl}(E)$ is a morphism of distributive lattices. Since $\overline{f}(s^{\downarrow})=f(s)^{\downarrow}$ for each $s\in S$ we have that $\overline{f}$ is coherent by Lemma \ref{lem:compact31}.
\end{proof}

We now use the assignments $\mathsf{K}$ and $\mathsf{Idl}$ to translate various types of morphisms between coherent complete restriction monoids to appropriate types of morphisms between distributive restriction semigoups.  We start from proper morphisms. A morphism $f:S\to T$ between distributive restriction semigroups will be called {\em proper} if for every $t\in T$ there is $n\geq 1$ and $t_1,\dots, t_n\in T$, $s_1\dots, s_n\in S$ such that $t=\bigvee t_i$ and $f(s_i)\geq t_i$ for every $i=1,\dots, n$.

\begin{lemma}\label{lem:proper31}\mbox{}
\begin{enumerate}
\item \label{i:proper31a} Let $f\colon S\to T$ be a proper coherent morphism between coherent complete restriction monoids. Then $f|_{\mathsf{K}(S)}$ is a proper morphism.
\item \label{i:proper31b} Let $f\colon S\to T$ be a proper morphism between distributive restriction semigroups. Then $\overline{f}$ is a proper coherent morphism.
\end{enumerate}
\end{lemma}

\begin{proof} \eqref{i:proper31a} Let $t\in \mathsf{K}(S)$. Since $f$ is proper there are $t_i\in T$ and $s_i\in S$, $i\in I$, with $t=\bigvee_I t_i$ and $f(s_i)\geq t_i$. But $t$ is finite and so $t=\bigvee_J t_j$ where $J\subseteq I$ is finite.

 \eqref{i:proper31b} Let $B\in \mathsf{Idl}(S)$. Then $B=\bigvee_{b\in B}b^{\downarrow}$. Since $f$ is proper, for  every $b\in B$ there are $n\geq 1$ and $b_1,\dots, b_n\in T$, $a_1,\dots, a_n\in S$ such that $f(a_i)\geq b_i$ for each $i$. This means that
 $\overline{f}(a_i^{\downarrow})\geq b_i^{\downarrow}$ for each $i$. Since $b^{\downarrow}=\bigvee b_i^{\downarrow}$, the statement follows.
 \end{proof}
 
We now turn to $\wedge$-morphisms. Let $f\colon S\to T$ be a coherent $\wedge$-morphism between coherent complete restriction monoids. Then $f|_{\mathsf{K}(S)}$ preserves existing meets and `potential' non-existing meets. We formalize this as follows.  We say that a morphism $f \colon S \rightarrow T$ of distributive restriction semigroups
is {\em weakly meet preserving} if given $t \leq f (a), f (b)$, there exists
$c \leq a,b$ such that $t \leq f(c)$.

\begin{lemma}\label{lem:wedge31}\mbox{}
\begin{enumerate}
\item \label{i:wedge31a} Let $f\colon S\to T$ be a coherent $\wedge$-morphism between coherent complete restriction monoids. Then $f|_{\mathsf{K}(S)}$ is a weakly meet-preserving morphism. 
\item \label{i:wedge31b} Let $f\colon S\to T$ be a weakly meet preserving morphism between distributive restriction semigroups. Then $\overline{f}$ is a coherent $\wedge$-morphism.
\end{enumerate}
\end{lemma}

\begin{proof}
\eqref{i:wedge31a} Let $a,b\in \mathsf{K}(S)$ and $t\in \mathsf{K}(S)$ be such that $t\leq f(a), f(b)$. This is equivalent to $t\leq f(a)\wedge f(b)=f(a\wedge b)$.  It follows that
$t\leq \bigvee \{f(s)\colon s\in X\}$ where $X=\{s\colon s\leq a\wedge b, s\in \mathsf{K}(S)\}$.
Since $t$ is finite it follows that there is a finite subset $Y\subseteq X$ with $t\leq \bigvee \{f(s)\colon s\in Y\}=f(\bigvee Y)$. The statement follows as the element $c=\bigvee Y$ is finite as a finite join of finite elements.

\eqref{i:wedge31b} Let $A,B\in \mathsf{Idl}(S)$. We need to show that $\overline{f}(A\cap B)=\overline{f}(A)\cap \overline{f}(B)$. This reduces to
$$
\bigvee_{x\in A\cap B}f(x)^{\downarrow}=\bigvee_{y\in A,z\in B} (f(y)^{\downarrow}\cap f(z)^{\downarrow}).
$$
The inclusion $\subseteq$ clearly holds. For the reverse inclusion assume that $t\leq f(y),f(z)$ where $y\in A$ and $z\in B$. Since $f$ is weakly meet-preserving we have that there is $c\leq y,z$ such that $t\leq f(c)$, which completes the proof.\end{proof}

We thus define the following types of morphisms between distributive restriction semigroups.
\begin{itemize}
\item type $1$: morphisms;
\item type $2$: proper morphisms;
\item type $3$: weakly meet-preserving morphisms;
\item type $4$: proper and weakly meet-preserving morphisms.
\end{itemize}

\begin{lemma} A morphism between distributive restriction $\wedge$-semigroups is weakly meet preserving if and only if it is meet preserving.
\end{lemma}

\begin{proof} It is immediate that a meet-preserving morphism is also weakly meet preserving. Assume that $f:S\to T$ is weakly meet preserving and let $a,b\in S$. Since $f$ is monotone, it is easy to see that $f(a\wedge b)\leq f(a)\wedge f(b)$. Let $t=f(a)\wedge f(b)$. By definition there is $c\leq a,b$ such that $t\leq f(c)$. 
Then $c\leq a\wedge b$ and we obtain $f(a)\wedge f(b)\leq f(a\wedge b)$. \end{proof}

\begin{theorem}\label{th:equiv_dist} For each $k=1,2,3,4$ the category of coherent (resp. strongly coherent) complete restriction monoids and their morphisms of type $k$ is equivalent to the category of distributive restriction semigroups (resp. distributive restriction $\wedge$-semigroups) and their morprisms of type $k$.  
\end{theorem}

\begin{proof}
By Lemmas \ref{lem:morphisms31}, \ref{lem:proper31} and \ref{lem:wedge31} the assignments $\mathsf{K}$ and $\mathsf{Idl}$ map morphisms of type $k$ to morphisms of type $k$ for each $k=1,2,3,4$ and it is immediate that these assignments are functorial.  Let $S$ be a distributive restriction semigroup. Define a map $\gamma: S\to \mathsf{K}(\mathsf{Idl}(S))$ be $s\mapsto \{s^{\downarrow}\colon s\in S\}$. Applying Lemma \ref{lem:compact31} it easily follows that $\gamma$ is a proper and weakly meet preserving isomorphism. Let $S$ be a coherent complete restriction monoid. Define a map $\theta: S\to  \mathsf{Idl}(\mathsf{K}(S))$ by $s\mapsto \{t\in \mathsf{K}(S) \colon t\leq s\}$. Then $\theta$ is a proper and $\wedge$-isomorphism. 
 \end{proof}

A coherent complete restriction monoid will be called {\em compact} if $e$ is a finite element. We can easily deduce a  monoid version of Theorem \ref{th:equiv_dist}.

\begin{corollary}\label{cor:equiv_dist} For each $k=1,2,3,4$ the category of compact coherent (resp. strongly coherent) complete restriction monoids and their morphisms of type $k$ is equivalent to the category of distributive restriction monoids (resp. distributive restriction $\wedge$-monoids) and their morprisms of type $k$.
\end{corollary}

\begin{proof}
Assume that a coherent complete restriction monoids $S$ is compact.  Then clearly $\mathsf{K}(S)$ is a monoid with unit $e$. Conversely, assume that $\mathsf{K}(S)$ is a monoid with the unit $f$. Then $f$ is a finite element of $E(S)$. We show that $f=e$. Let $a\in S$. As $e^{\downarrow}$ is coherent, $\lambda(a)=\bigvee f_i$ is a join of finite elements. But then $a=\vee af_i$ and the elements $af_i$ are finite by Lemma \ref{lem:29m} as $\lambda(af_i)=f_i$. It follows that $af=\bigvee a_if = \bigvee a_i=a$. By symmetry we also obtain $fa=a$. It follows that $f$ is the unit of $S$ and so $f=e$ by uniqueness of the unit. It is immediate from the constructions in Lemmas \ref{lem:morphisms31}, \ref{lem:proper31}, \ref{lem:wedge31} that coherent morphisms of each of types $1, 2, 3, 4$ of compact coherent complete restriction monoids correspond to monoids moprhisms of respective types. Moreover,  if $S$ is a distributive restriction monoid then the assignment $\gamma$ of the proof of  Theorem \ref{th:equiv_dist} is a monoid morphism. 
\end{proof}

Theorem \ref{th:equiv_dist} and Corollary \ref{cor:equiv_dist} can be readily adapted to the involutive, cancellative and inverse settings. Combining these results with Theorem \ref{th:adjga} we obtain the following.

\begin{theorem}[Non-commutative Stone Duality I] \mbox{}\label{th:duality:1}
\begin{enumerate} 
\item For each $k=1,2,3,4$ the category of distributive restriction semigroups (resp. $\wedge$-semigroups, monoids, $\wedge$-monoids) is dually equivalent to the category of spectral (resp. strongly spectral, compact spectral, compact strongly spectral) \'etale topological categories and their morphisms of type $k$. Under this equivalence, classes of distributive ample semigroups correspond to respective classes of cancellative spectral \'etale topological categories.
\item For each $k=1,2,3,4$ the category of Boolean restriction semigroups (resp. $\wedge$-semigroups, monoids, $\wedge$-monoids) is dually equivalent to the category of Boolean (resp. strongly Boolean, compact Boolean, compact strongly Boolean) \'etale topological categories and their morphisms of type $k$. Under this equivalence, classes of Boolean ample semigroups correspond to respective classes of cancellative Boolean \'etale topological categories.
\end{enumerate} 
\end{theorem}

An adaptation of the theorem above to the inverse setting is the following result which extends the duality theorem by Lawson and Lenz \cite{LL1, LL2} to wider classes of morphisms and to the distributive $\wedge$-setting.

\begin{theorem}[Non-commutative Stone Duality II] \mbox{}
\begin{enumerate} 
\item For each $k=1,2,3,4$ the category of distributive inverse semigroups (resp. $\wedge$-semigroups, monoids, $\wedge$-monoids) is dually equivalent to the category of spectral (resp. strongly spectral, compact spectral, compact strongly spectral) \'etale topological groupoids and their morphisms of type $k$. 
\item For each $k=1,2,3,4$ the category of Boolean inverse semigroups (resp. $\wedge$-semigroups, monoids, $\wedge$-monoids) is dually equivalent to the category of Boolean (resp. strongly Boolean, compact Boolean, compact strongly Boolean) \'etale topological groupoids and their morphisms of type $k$. 
\end{enumerate} 
\end{theorem}

\end{document}